\documentclass{article}

\usepackage{amsmath}
\usepackage{amssymb}
\usepackage{amsthm}
\usepackage{amscd}
\usepackage{amsfonts}
\usepackage{pgf}
\usepackage{tikz}
\usepackage{graphicx}
\usepackage{tikz-3dplot}
\usepackage{hyperref}

\usepackage{soul}
\usepackage{mathrsfs}
\usepackage{caption}
\usepackage{subcaption}
\usepackage{algorithm}
\usepackage[noend]{algpseudocode}
\usepackage{bm}
\usepackage{tocloft}
\usepackage{enumitem}
\usepackage{multirow}
\usepackage[skip=1ex]{caption}
\usepackage{authblk}

\usetikzlibrary{decorations.markings}
\tikzset{middlearrow/.style={
        decoration={markings,
            mark= at position 0.5 with {\arrow{#1}} ,
        },
        postaction={decorate}
    }
}
\usetikzlibrary{calc,cd}

\graphicspath{ {./images/} }

\providecommand{\keywords}[1]
{
  \small	
  \textbf{\textit{Keywords:}} #1
}

\theoremstyle{plain}
\newtheorem{theorem}{Theorem}
\newtheorem{proposition}{Proposition}

\newtheorem{assumption}{Assumption}
\newtheorem{lemma}{Lemma}

\theoremstyle{definition}
\newtheorem{definition}{Definition}
\newtheorem{remark}{Remark}
\newtheorem{example}{Example}

\theoremstyle{remark}
\newtheorem{case}{Case}

\renewcommand{\dim}{{\operatorname{dim}}}
\renewcommand{\deg}{{\operatorname{deg}}}

\renewcommand\path{\mathfrak{P}}

\def\Cf1{{\mathfrak C^1}}
\def\Cf{{\mathfrak C}}

\renewcommand\vec{\mbox{vec}}
\newcommand{\indep}{\perp \!\!\! \perp}
\newcommand{\dep}{\not\!\perp\!\!\!\perp}
\newcommand{\adj}{\mbox{adj}}

\title{Identifiability of VAR(1) model in a stationary setting}
\author{Bixuan Liu\footnote{Corresponding author. Email: \texttt{bixuan.liu@sorbonne-universite.fr}}}
\affil{Sorbonne Université and Université Paris Cité, CNRS, Laboratoire de Probabilités, Statistique et Modélisation, F-75005 Paris, France}

\date{\today}

\begin{document}

\maketitle

\begin{abstract}
We consider a classical First-order Vector AutoRegressive (VAR(1)) model, where we interpret the autoregressive interaction matrix as influence relationships among the components of the VAR(1) process that can be encoded by a weighted directed graph. A majority of previous work studies the structural identifiability of the graph is based on time series observations and therefore relies on dynamical information. By constract, this work aims to incorporate the sampling insufficiency problem in ecological research, that is, due to limited resources, time-series observations are often unavailable and what we have in practice is independently, identically distributed data sampled from the stationary distribution, making classical time-series methods inapplicable. Within this framework, we assume that an equilibrium of VAR(1) exists, and study instead the identifiability of the graph from the stationary distribution, meaning that we seek a way to reconstruct the influence graph underlying the dynamic network using only static information. Importantly, the graph associated to the support of the autoregressive interaction matrix of the VAR(1) model does not coincide with the graphical model of the stationary distribution, making standard tools developed for Gaussian Graphical Models (GGMs) irrelevent. Therefore, this paper applies an approach from algebraic statistics that characterizes models using the Jacobian matroids associated with the parametrization of the models. We successfully derive easily applicable sufficient graphical conditions under which different graphs yield distinct steady-state distributions through the lens of maximal class, a concept that is initially introduced by this paper. Additionally, we illustrate how our results could be applied to characterize networks inspired by ecological research.
\end{abstract}

\vspace{1cm} 

\section{Introduction}

\paragraph{VAR(1) model.}This paper considers a classical First-order Vector AutoRegressive (VAR(1)) model:
\begin{equation}\label{eq:VAR}
	\left\{\begin{array}{l}
		X_0 = \epsilon_0, \\
		X_t = \Lambda^TX_{t-1} + \epsilon_t\quad \mbox{for }t\in\mathbb{N}, \\
	\end{array}\right.
\end{equation}
where $X_t$ and $\epsilon_t$ are random vectors in $\mathbb{R}^n$ for some $n\in\mathbb{N}$. 
$\epsilon_t$ is independent of $X_{t-1}$, and $(\epsilon_t)_{t \in \mathbb{N}}$ are i.i.d..
$\Lambda = \left(\lambda_{ij}\right) \in M_n\left(\mathbb{R}\right)$, i.e., it is a deterministic $n\times n$ matrix with values in $\mathbb{R}$, which we call the interaction matrix. Each element $\lambda_{ij}$ represents the direct influence of the $i^{\mbox{th}}$ coordinate of $X_{t-1}$ on the $j^{\mbox{th}}$ coordinate of $X_t$. In this paper, the error term $\epsilon_t$ is centered with covariance matrix $\omega I_n$, where $I_n\in M_n(\mathbb{R})$ is the identity matrix, and $\omega\in\mathbb{R}^+$ is a positive constant.

If the eigenvalues of $\Lambda$ are all smaller than $1$ in absolute value, then as $t$ goes to infinity, $x_t$ converges in distribution to a centered distribution with covariance matrix $\Sigma = \left(\sigma_{ij}\right)\in M_n(\mathbb{R})$ (\cite{anderson2000note}). $\Sigma$ satisfies (\cite{young2019identifying}):
\begin{equation}\label{eq:var1_sig_lam}
    \Sigma = \Lambda^T\Sigma\Lambda + \omega I_n.
\end{equation}
Apply the vec operator (\cite[Section~10.2]{petersen2008matrix}) on both sides, (\ref{eq:var1_sig_lam}) is equivalent to:
\begin{equation}\label{eq:vec_Sigma}
    \vec\left(\Sigma\right) = \left(I_n-\Lambda^T\otimes\Lambda^T\right)^{-1}\vec\left(\omega I_n\right),
\end{equation}
where $\otimes$ is the matrix tensor product (\cite{young2019identifying}). (\ref{eq:vec_Sigma}) indicates that $\Sigma$ is unique given a unique set of parameters $\Lambda$ and $\omega$. 

\paragraph{Problem statement.}This paper studies the identifiability of the VAR(1) model in a stationary setting, that is, whether different parameters ($\Lambda$ and $\omega$) produce the same $\Sigma$, using (\ref{eq:var1_sig_lam}). We restrict the problem to the identifiability of the \textit{support} of $\Lambda$, i.e., the location of its non-zero elements, which corresponds to the reconstruction of the underlying graph structure. Importantly, the properties of Gaussian Graphical Models (GGMs) (see \cite{lauritzen1996graphical} for an introduction) are inapplicable to our setting because the graphical model of the stationary distribution is the graph corresponding to the support of the inverse of the covariance matrix $\Sigma^{-1}$, and it is not the one induced by the support of $\Lambda$. 
For instance, in Section \ref{sec:gene_iden}, Example \ref{eg:glb_iden} shows that $\Sigma^{-1}$, which is computable from $\Sigma$, can have full support (i.e., all elements are non-zero) even when $\Lambda$ does not. For this reason, our work does not study conditional independence relations as in GGMs. Instead, different methods are needed to treat our problem.

To be more precise about the graphical model, we consider in this paper, the VAR(1) model (\ref{eq:VAR}) is associated with a directed graph $G = \left(V,\mathfrak{E}_G\right)$ with node set $V=\left\{1,2,\cdots,n\right\}$, each corresponding to one of the components of $X_t$, and edge set $\mathfrak{E}_G$ consists of ordered pairs $\left(i,j\right)$ or $i\rightarrow j$, which represents the direct influence among components of $X_t$. For any $i,j\in\left[n\right]$, an edge $\left(i,j\right)$ in $\mathfrak{E}_G$ indicates $\lambda_{ij}\neq 0$, meaning that there is a direct effect of the $i^{\mbox{th}}$ coordinate of $X_{t-1}$ on the $j^{\mbox{th}}$ coordinate of $X_t$. Figure \ref{fig:grap_mod} shows an example of a graphical VAR(1) model with $4$ nodes, and the corresponding graph $G$. Because $\Lambda$ is deterministic, i.e., it does not change over time, the dynamic model in $(a)$ is encoded by the graph $G$ in $(b)$ for simplicity.  
Throughout the paper, we assume that \(G\) is a \textit{reflexive directed graph}, that is, $(i,i)\in \mathfrak{E}_G,\ \forall i\in [n]$. This convention is natural in many real-life scenarios such as ecological networks where we expect the evolution of species to be dependent on itself.

\begin{figure}
\centering
\begin{subfigure}{.58\textwidth}
  \centering
  \includegraphics[width=0.55\linewidth]{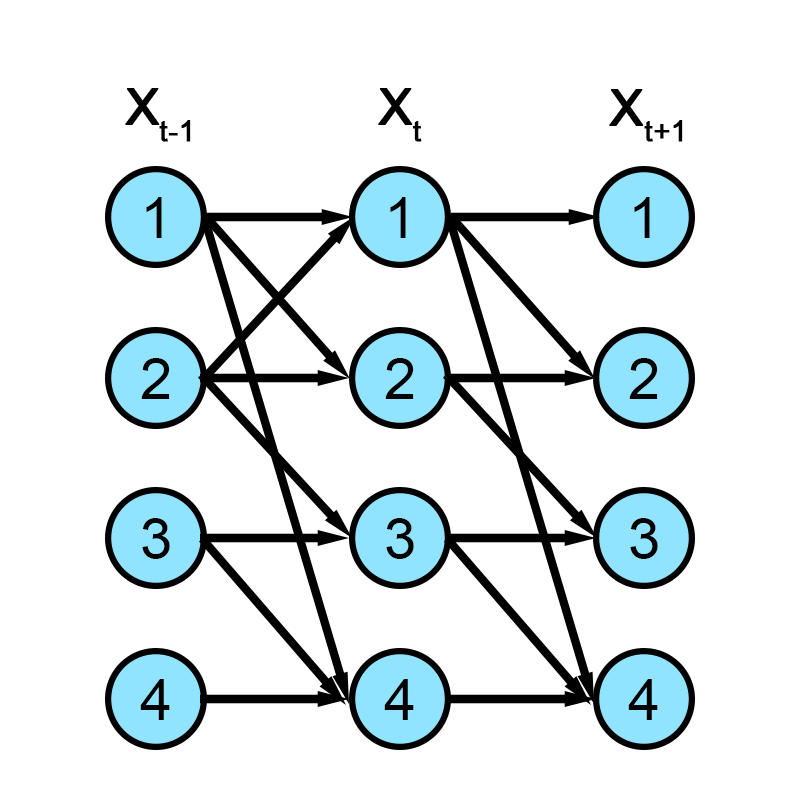}
  \caption{The graphical VAR(1) model}
\end{subfigure}
\begin{subfigure}{.38\textwidth}
  \centering
  \includegraphics[width=0.45\linewidth]{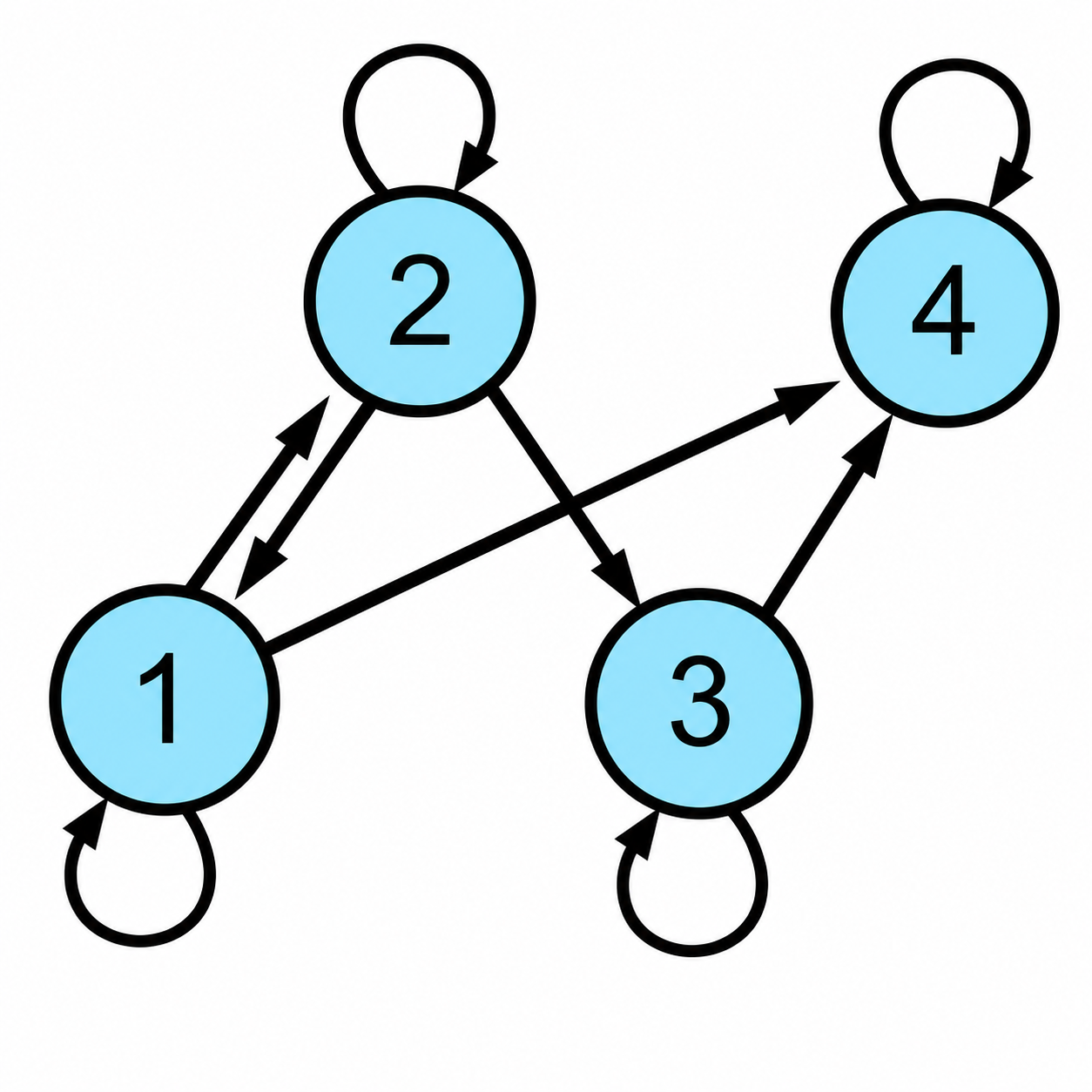}
  \caption{The encoded graph $G$}
\end{subfigure}
\caption{Example of a graphical VAR(1) model and the encoded graph $G$}
\label{fig:grap_mod}
\end{figure}

\paragraph{Motivation.}The motivation of this paper can be found in ecological research. Ecosystems are often modeled by dynamical systems. For example, in the model of Robert May (\cite{may1972will}), the linearization of the dynamics of abundancies or densities of $n$ interacting species around their equilibrium leads to the system:
\begin{equation}\label{eq:eco_may}
    \frac{\mathrm{d}X}{\mathrm{d}t} = AX,
\end{equation}
where $A$ is a deterministic $n\times n$ matrix that represents interactions among species through various mechanisms such as competition, predation or facilitation. The equation \eqref{eq:VAR} can be seen as a stochastic and discrete in time version of \eqref{eq:eco_may}.

VAR(1) models can also arise when we consider the (vectorial) fluctuation process $X_t$ associated with individual-based models converging to Lotka-Volterra ordinary differential equations. The fluctuation process is then shown to satisfy a continuous time version of \eqref{eq:VAR} with a noise having a diagonal covariance matrix (see \cite{akjouj2024complex} for more precisions).
More importantly, it is reasonable to start from a basic and classical model also because, as pointed out later in this section, the study of dynamical system with only information from a steady state is a subject rarely explored.

The reason for which we are particularly considering a stationary setting lies in the sampling insufficiency problem of ecological research. It is a central topic for ecologists to reconstruct the dynamic interaction network, such as the support of $\mathbf{A}$ in (\ref{eq:eco_may}) or the support of $\Lambda$ in (\ref{eq:VAR}), from the abundance data captured in the ecosystem, since studying the dynamic influence among species is crucial to understanding the evolution of the abundances of species and thus helps to prevent extinction. However, as pointed out by \cite{legendre2010community}, fine-scale replication sampling such as time-series sampling is generally too costly and are often sacrificed in practice. Instead, the limited resources (manpower, time, money) are used to cover as much ground as possible during each sampling campaign. Consequently, a considerable amount of abundance data available covers a large number of locations but each of them is sampled only once, making classical time-series analysis inapplicable. For example, \cite{chaffron2021environmental} inferred a global ocean cross-domain plankton co-occurrence network with one-time-samplings of $200$ stations. It is true that there are cases where each of the sampling sites is sampled multiple times, but the temporal gap between each sampling campaign is generally too large. For instance, \cite{zhang2023habitat} sampled for plant–aphid–ant interactions on $18$ forest islands once a year from 2020 to 2021. There are also researches that collect data regularly, e.g. monthly, for each sampling site, but the data points are not enough to conduct classical time-series analysis, such as \cite{simmons2019abundance}, which collected plant and humming bird abundance over $19$ sites monthly but only over a little more than a year. 

To resolve the sampling insufficiency problem and potentially make the planning of sampling campaigns easier for ecologists, we seek for a way to reconstruct the dynamic interaction network of an ecosystem with only i.i.d. samples of a final stationary distribution of the dynamical system, instead of time-series observations. As a first step of solving this problem, this paper assumes that we have the final state of the VAR(1) model, i.e., the matrix $\Sigma$ in (\ref{eq:var1_sig_lam}), from which we try to recover the dynamic interaction network, i.e., the support of $\Lambda$. Although in practice, we will have i.i.d. observations from the stationary distribution and never the true $\Sigma$, it is worth verifying that the network is identifiable with the true $\Sigma$, because if not, it's highly unlikely to be identifiable with real data.

\paragraph{Related work.}
To the best of our knowledge, the identifiability of VAR(1) parameters in a stationary setting has rarely been explored in the existing literature. A majority of previous work that has dealt with the identifiability of the VAR(1) model, and achieved significant and systematic results, relied on observations that were sampled as time series (e.g. \cite{gong2015discovering}, \cite{tank2019identifiability}). This is reasonable because the parameters involve dynamic interactions among the components of the random process $X_t$, making it natural to expect that it should only be inferred from dynamic data. Nevertheless, there is a literature that deals with the identifiability of the interaction matrix $\Lambda$ of VAR(1) in a stationary setting. \cite{young2019identifying} provides a necessary identifiability condition based on the number of parameters to estimate by counting the parameters. However, this condition only excludes a small subset of cases where the interaction matrix is not identifiable, while we are more interested in sufficient conditions for identifiability. 

In contrast to the comparatively limited existing literature on stationary VAR(1) models, the structural identifiability of two closely related, yet distinct, classes of models -- linear structural equation models (SEMs) and continuous Lyapunov models (LMs) -- have been more extensively studied and remain an active area of research. This line of work typically relies on algebraic-statistical methods, for which a comprehensive introduction can be found in \cite{sullivant2018algebraic}.

A linear structural equation model (SEM) specifies a random vector $X$ through the system
\begin{equation*}
    X = \Lambda^{T}X + \epsilon,\ \epsilon \sim \mathcal{N}(0,\Omega),
\end{equation*}
where $\Lambda$ is a deterministic matrix representing the direct effects among random variables. When $I-\Lambda$ is invertible, the system has a unique solution $X=(I-\Lambda)^{-T}\epsilon$. Consequently, $X$ is Gaussian with covariance matrix $\Sigma$. Its inverse, the precision matrix $K$, satisfies
\begin{equation}\label{eq:sem_sig_lam}
    K = \Sigma^{-1}
    = (I-\Lambda)\Omega^{-1} (I-\Lambda)^{T}.
\end{equation}
The parameter identifiability of linear SEMs has been extensively studied; see, for instance, \cite{drton2018algebraic} for a review of recent developments. Important contributions include \cite{drton2025identifiability}, which establishes generic identifiability results for graphs allowing directed cycles, and \cite{bongers2021foundations}, which studies identifiability in the presence of latent variables. 

A continuous-time linear model, more precisely an \textit{Ornstein-Uhlenbeck process}, describes a stochastic process $\{X_t\}_{t\geq 0}$ as the solution of the stochastic differential equation
\begin{equation*}
    dX_t = \Lambda^{T}X_t \, dt + \Omega^{1/2} \, dW_t,
\end{equation*}
where $W_t$ is a standard $n$-dimensional Brownian motion. If $\Lambda$ is stable, namely if all eigenvalues of $\Lambda$ have strictly negative real parts, then the process admits a unique stationary distribution. It is centered Gaussian with covariance matrix $\Sigma$, which is characterized by the continuous Lyapunov equation
\begin{equation}\label{eq:continuous_LM_sigma_lam}
    \Lambda^{T}\Sigma + \Sigma \Lambda + \Omega^{1/2}\Omega^{T/2} = 0.
\end{equation}
The identifiability of the drift matrix $\Lambda$ from the stationary covariance matrix $\Sigma$ has received substantial attention. Recent contributions include \cite{amendola2025structural}, which characterizes model equivalence classes for acyclic continuous LMs; \cite{recke2026identifiability}, which proves generic parameter identifiability of $\Lambda$ given any connected graph in a non-Gaussian setting; and \cite{dettling2023identifiability}, which provides sparsity conditions ensuring global identifiability of $\Lambda$.

It is worth noting that the results on continuous LMs do not directly imply analogous results in our setting because the stationary covariance equation is different. Continuous LMs correspond to (\ref{eq:continuous_LM_sigma_lam}), whereas stationary VAR(1) models satisfy (\ref{eq:var1_sig_lam}). For instance, literatures in continuous LM like \cite{amendola2025structural,dettling2023identifiability} define the model as the set of covariance matrices $\Sigma$ satisfying the continuous Lyapunov equation (\ref{eq:continuous_LM_sigma_lam}), and study the algebraic properties of this set. Replacing this equation by the stationary VAR(1) covariance equation changes the underlying parametrization map, and hence the algebraic structure of the resulting model. Therefore, establishing results in our setting requires a separate analysis.

One of the most recently developed method to test generic identifiability of a parametrized statistical model is through distinguishing the Jacobian matroid of the parametrization map. \cite{hollering2021identifiability} develop an algorithm to numerically study the Jacobian matroid and apply it to phylogenetics models, while \cite{drton2025identifiability} treat the problem in a theoretical manner where they characterize mathematically properties of the Jacobian matroid of a linear SEM and successfully derive graphical conditions for generic identifiability. However, the Jacobian matroid is specific to the parametrization map of the model. Results obtained for phylogenetic models or linear SEMs therefore do not directly imply analogous statements for stationary VAR(1) models, nor do they yield generic identifiability results in our setting.

\paragraph{Our contribution.} Our work takes the same starting point as in Drton’s work (\cite{drton2025identifiability}) and seeks to achieve graphical conditions for generic identifiability of VAR(1) from its steady state. Our main contributions include
\begin{enumerate}
    \item \textbf{Targeting the unexplored problem of structural identifiability of VAR(1) from stationary information alone.} Although VAR(1) models are fundamental in multivariate time-series analysis, the extent to which their directed interaction graphs can be recovered from stationary information has received limited attention. Our work provides, to the best of our knowledge, one of the first systematic identifiability results for stationary VAR(1) graph models, clarifying when the underlying dynamic structure can be inferred without access to time-series data.
    \item \textbf{Addressing a more complex identifiability setting induced by discrete-time dynamics.} Although our problem of recovering $\Lambda$ from the stationary covariance matrix $\Sigma$ is closely related to identifiability problems for linear SEMs and continuous LMs, the VAR(1) setting is substantially more difficult. This difficulty arises from the equation (\ref{eq:var1_sig_lam}) -- also known as the discrete Lyapunov equation, whose structure differs fundamentally from the corresponding equations in these two models. Unlike linear SEMs (\ref{eq:sem_sig_lam}), it does not yield a convenient expression for either $\Sigma$ or $\Sigma^{-1}$ without introducing an inverse depending on $\Lambda$. On the other hand, while in continous LM, $\Sigma$ admits a compact integral representation in terms of $\Lambda$ and $\Omega$ (\cite[Proposition~2.1]{varando2020graphical}), this representation is replaced by an infinite series of Kronecker product of $\Lambda$ as written in (\ref{eq:vec_sig_sum}) in the discrete case. This more intricate structure, inherent to discrete-time dynamics, makes the analysis of structural identifiability considerably more challenging.
    \item \textbf{New illustration of trek rule.} One of our main results (Theorem \ref{thm:supp_sig_maxc}) characterizes the non-zero pattern in a generic covariance matrix $\Sigma$ with the existence of a specific graphical structure that we call “maximal classes”, which can be viewed as sets of nodes that receive information from a common "source" in a directed graph (formal definition and relevant reconstruction algorithms are provided in later sections). A maximal class is also a maximal set of nodes that are pairwise connected by \textit{treks} (\cite[Definition~2.6]{sullivant2010trek}) with the same top. This zero pattern, or the \textit{trek rule}, also exists in linear SEM (\cite[Theorem~4.2]{drton2018algebraic}) and continuous LM (\cite[Corollary~2.3]{varando2020graphical}). Our work validate it in VAR(1), offering an additional insight into its generality.
    \item \textbf{Practical graphical identifiability conditions and ecological insight.} We propose easily verifiable graphical conditions for generic identifiability, showing in particular that graphs with different maximal classes generically yield distinct stationary distributions. These conditions are also interpretable, since maximal classes serve as a bridge between dynamical network and the corresponding stationary distribution. We illustrate the applicability of our results through examples inspired by ecological networks.
\end{enumerate}

\paragraph{Outline.}The rest of this paper is organized as follows: Section \ref{sec:VAR_iden} formally defines the statistical model and introduces the definitions and testing criteria for global and generic identifiability with an overview of the main results of this paper in the end. Section \ref{sec:Jacobian_structure} introduces results on the Jacobian structure of the parametrization of the model. Section \ref{sec:maximal_classes} presents the concept of “maximal classes” for directed graphs and its relationship with the support of the covariance matrix $\Sigma$, which is the core of this study. Section \ref{sec:iden_res} introduces two sufficient conditions of generic identifiability based on maximal classes, one of which is only valid for characterizing models with the same dimension, while the other one can be extended to models with any dimension. A result that allows us to calculate the dimension of the model for reflexive graphs with no reciprocal edges, is also provided to make the two conditions applicable. Section \ref{sec:illustrations} is an illustration of how our results could be applied in ecological research.

\section{Identifiability of the VAR(1) model in a stationary setting}\label{sec:VAR_iden}
In this section, we start with introducing the parametrization and statistical model of our settings, then define global and generic identifiability, along with the identifiability criteria used throughout this paper. Finally, we provide an overview of our main results, i.e., sufficient conditions for generic identifiability.

The definition and criteria for identifiability are based on the concept of a \textit{generic point}. We define a point in a set as generic if it does not lie within a proper (i.e., strictly included) algebraic subvariety of the set. In concrete terms, a generic point is any point in the set except for those that satisfy a specific non-trivial polynomial equation. Because such proper algebraic subvarieties have Lebesgue measure zero, another interpretation is that a parameter value sampled from a continuous probability distribution is generic with probability one.

\subsection{Problem settings}\label{sec:Problem_settings}
Given a VAR(1) model (\ref{eq:VAR}) and the associated directed graph $G = \left(V,\mathfrak{E}_G\right)$, where $V = [n] = \{1,2,\cdots,n\}$, define the set of interaction matrices $M_G$ associated to a graph $G$ as:
\begin{equation*}
	M_G := \left\{\Lambda=\left(\lambda_{ij}\right)\in M_n\left(\mathbb{R}\right) \mid \rho(\Lambda) < 1,\mbox{ and } \lambda_{ij} = 0 \mbox{ if } \left(i,j\right)\notin \mathfrak{E}_G\right\},
\end{equation*}
where $\rho(\cdot)$ is the spectral radius, and the condition $\rho(\Lambda) < 1$ ensures the existence of the equilibrium. The \textit{parameter space} is then $M_G\times \mathbb{R}^+$, corresponding to elements of $\Lambda$ and $\omega$.

Next, we formally define the model.

\begin{definition}\label{def:mod_para_map}
For the stationary VAR(1) model corresponding to a graph $G$, define $\mathbf{M}_G$ the following set of matrices:
\begin{equation*}
	\mathbf{M}_G = \left\{\Sigma \mid \Sigma = \phi_G\left(\Lambda,\omega\right), \Lambda \in M_G, \omega\in\mathbb{R}^+\right\},
\end{equation*}
where
\begin{align*}
	\phi_G : M_G \times \mathbb{R}^+ & \rightarrow M_n\left(\mathbb{R}\right) \\
	\left(\Lambda,\omega\right) & \mapsto \Sigma, \mbox{ s.t. } \Sigma = \Lambda^T\Sigma\Lambda + \omega I_n
\end{align*}
is the \textit{parametrization map} defined by (\ref{eq:var1_sig_lam}). By abuse of notation, we call the set $\mathbf{M}_G$ the \textit{stationary VAR(1) model}.
\end{definition}

\begin{remark}[Structural Properties]
We highlight two structural characteristics of the parameter space $M_G \times \mathbb{R}^+$ and the model $\mathbf{M}_G$:

\vspace{0.5em}
\noindent \textbf{1. Nested structure:} The definition of $M_G$ allows for edge weights to vanish (i.e., $\lambda_{ij}=0$ for $(i,j)\in\mathfrak{E}_G$). Consequently, the parameter space of any subgraph $G' \subset G$ is naturally embedded within $M_G$. This inclusion property extends to the models (i.e., $\mathbf{M}_{G'} \subseteq \mathbf{M}_G$), ensuring a continuous transition between the model $\mathbf{M}_G$ and its sub-models.

\vspace{0.5em}
\noindent \textbf{2. Error structure:} The covariance matrix of the error term in (\ref{eq:VAR}) is assumed to be diagonal with identical values $\omega$. While this assumption is central to our calculations, the underlying framework is generalizable to other diagonal matrices with distinct entries. However, the calculations do not extend to non-diagonal matrices. We adopted the identical diagonal assumption to minimize the number of parameters to estimate, allowing us to focus specifically on the reconstruction of $\Lambda$.
\end{remark}

In this paper, we study the identifiability of a finite family of stationary VAR(1) models, i.e., assume that we are given a finite list of possible graphs, we try to see under what conditions they yield different stationary distributions, and in particular different stationary covariance relationships. This is particularly useful in ecological research where we could have prior information about the network, and were able to reduce the network to a finite list of candidates. Besides, as illustrated in Example \ref{eg:glb_iden} below, identifiability of the underlying graph from the covariance matrix $\Sigma$ in a conventional sense is not achieved because there exist two networks that yield the same stationary distribution.

For a finite family of stationary VAR(1) models, we characterize them by studying the dimension of the intersection of the images of the parametrization map, in the sense that two models are generically identifiable if they only intersect in a set with smaller dimension (detailed definitions of identifiability introduced in the next section). The dimension of the model $\mathbf{M}_G$ is defined as follows:
\begin{equation*}
    \dim\left(\mathbf{M}_G\right) := \dim\left(\overline{\mathbf{M}_G}\right),
\end{equation*}
where $\overline{\mathbf{M}_G}$ denotes the Zariski closure of $\mathbf{M}_G$ in the ambient affine space of symmetric positive definite matrices. The dimension on the right hand side is the usual Krull dimension of this algebraic variety. Since the parametrization map is a rational function, i.e., it is a quotient of polynomials, (proof in Appendix \ref{app:para_rat}), the dimension of $\overline{\mathbf{M}_G}$ is the rank of the (transpose of the) Jacobian matrix of the parametrization map evaluated at a \textit{generic point}. At regular points of the model, the algebraic dimension coincides with the local dimension of $\mathbf{M}_G$ viewed as a smooth manifold. A brief review of relevant notions from Zariski topology is provided in Appendix \ref{app:alge_no}.

The next section presents the definitions for global and generic identifiability.

\subsection{Global and generic identifiability}\label{sec:gene_iden}

\begin{definition}\label{def:glb_iden}
Given a family of stationary VAR(1) models $\left\{\mathbf{M}_k\right\}_{k=1}^K$ and associated graphs $\left\{G_{k}=\left(
V,\mathfrak{E}_{G_k}\right)\right\}_{k=1}^K$, where $K\in\mathbb{N}$, then the model, i.e., the discrete parameter $k$ is \textit{globally identifiable} if for any distinct pair $\left(k_1,k_2\right)$ of values from $1$ to $K$, $\mathbf{M}_{k_1}\cap\mathbf{M}_{k_2}=\emptyset$.
\end{definition}

In fact, global identifiability is a highly restrictive property of the models. There exist cases where models with different underlying graphs, i.e., different supports of $\Lambda$, yield the same $\Sigma$ (see Example \ref{eg:glb_iden}). 

\begin{example}\label{eg:glb_iden}
Consider the case where $n=3$, and the following two interaction matrices, which represent different graphs:
\begin{equation*}
    \Lambda_1 = \left[\begin{array}{ccc}
        0.50 & 0.70 & 0 \\
        0 & 0.90 & 0 \\
        0 & 0.80 & 0.40
    \end{array}\right], \quad \Lambda_2 = \left[\begin{array}{ccc}
        0.50 & 0.67 & -0.01 \\
        0 & 0.94 & 0.02 \\
        0 & 0 & 0.38
    \end{array}\right]
\end{equation*}
and if we let $\omega$ in both cases equals to $1$, then by (\ref{eq:var1_sig_lam}), both models yield the following covariance matrix:
\begin{equation*}
    \Sigma = \left[\begin{array}{ccc}
        1.33 & 0.85 & 0 \\
        0.85 & 22.85 & 0.60 \\
        0 & 0.60 & 1.19
    \end{array}\right].
\end{equation*}
\end{example}

Therefore, we focus instead on a less restrictive definition of identifiability, called \textit{generic identifiability}, which is widely used in the field of algebraic statistics. It allows the images of parametrization maps to intersect only in a set with Lebesgue measure zero.

\begin{definition}\label{def:gene_iden_min}
Let $\left\{\mathbf{M}_k\right\}_{k=1}^K$ be a family of stationary VAR(1) models as in Definition \ref{def:glb_iden}. Then the discrete parameter $k$ is \textit{generically identifiable} if for any distinct pair $\left(k_1,k_2\right)$ of values from $1$ to $K$,
\begin{equation*}
    \dim\left(\mathbf{M}_{k_1}\cap\mathbf{M}_{k_2}\right)<\min\left\{\dim\left(\mathbf{M}_{k_1}\right),\dim\left(\mathbf{M}_{k_2}\right)\right\}.
\end{equation*}
\end{definition}

The geometric interpretation of Definition \ref{def:gene_iden_min} is that a family of models are generically identifiable if the intersection of any two models in the family is a Lebesgue measure zero subset of both models. But this definition also implies that in some circumstances, models with different dimensions are not generically identifiable, which is not desired. Instead, the following definition of generic identifiability is used. 

\begin{definition}\label{def:gene_iden_max}
Let $\left\{\mathbf{M}_k\right\}_{k=1}^K$ be a family of stationary VAR(1) models as above. Then the discrete parameter $k$ is \textit{generically identifiable} if for each pair $\left(k_1,k_2\right)$ of values of $k$, where $k_1\neq k_2$, 
\begin{equation*}
    \dim\left(\mathbf{M}_{k_1}\cap\mathbf{M}_{k_2}\right)<\max\left\{\dim\left(\mathbf{M}_{k_1}\right),\dim\left(\mathbf{M}_{k_2}\right)\right\}.
\end{equation*}
\end{definition}

Definition \ref{def:gene_iden_max} states that the models in the family are generically identifiable from each other if the intersection of any two models in the family is a Lebesgue measure zero subset of the union of the models. This definition yields a simple criterion: two stationary VAR(1) models $\mathbf{M}_1$, $\mathbf{M}_2$ are generically identifiable if
\begin{equation*}
    \dim\left(\mathbf{M}_1\right) \neq \dim\left(\mathbf{M}_2\right).
\end{equation*}
This property is one of the main tools that we used in this paper to test generic identifiability. Note that it does not imply that two models whose graphs are nested are generically identifiable, because as illustrated by the third row of Table \ref{tab:sum_iden_res} at the end of the paper, adding one edge to a graph might not change the dimension of the corresponding model. On the other hand, this property indicates that if we find a way to access the dimensions of the models, we can focus on the identifiability of models with the same dimension, in which case Definitions \ref{def:gene_iden_min} and \ref{def:gene_iden_max} are equivalent.

In fact, there exist many tools to test generic identifiability \cite[Chapter~16]{sullivant2018algebraic}. This paper applies one of them that uses the Jacobian matroid derived from the parametrization of the model, developed by \cite{drton2025identifiability}. The following section provides an introduction to this method.

\subsection{Identifiability criteria}
Before introducing the tools to test generic identifiability, we present some necessary algebraic notions.

\begin{definition}
Let $\phi$ be a $\mathcal{C}^1$ map from $\mathbb{R}^d$ to $\mathbb{R}^p$: $\phi\left(\theta_1,\cdots,\theta_d\right) = \left(\phi_1\left(\bm{\theta}\right),\cdots,\phi_p\left(\bm{\theta}\right)\right)$. Then the Jacobian matrix of this map is:
\begin{equation*}
    \mathbf{J}_\phi=\left(\frac{\partial\phi_i}{\partial\theta_j}\right), 1\leq i\leq p, 1 \leq j\leq d.
\end{equation*}
\end{definition}

Next we define the Jacobian matrix for the stationary VAR(1) model, which is the transpose of the usual Jacobian matrix of the parametrization map.

\begin{definition}
Let $\mathbf{M}_G$ be a stationary VAR(1) model, and $\phi_G$ in Definition \ref{def:mod_para_map} parametrizes $\mathbf{M}_G$. Then the Jacobian matrix of the model is
\begin{equation*}
    \mathbf{J}_G=\left(\frac{\partial\phi_j}{\partial\theta_i}\right), 1\leq i\leq E_G+1, 1 \leq j\leq \frac{n(1+n)}{2},
\end{equation*}
where $E_G$ is the number of edges in $G$, $\phi_j$'s are the distinct entries of $\Sigma$, and $\theta_i$'s are the entries of $\Lambda$ and $\omega$.
\end{definition}

Now we introduce the Jacobian matroid, the main object used to characterize models. To do so, we first introduce the linear matroid associated with a matrix. The general definition of a matroid is omitted here, as it is not necessary for the tools used in this paper. We refer the interested readers to \cite{sullivant2018algebraic} for further details.

\begin{definition}\label{def:linear_matroid}
Let $A\in M_{m\times n}\left(\mathbb{R}\right)$ be a matrix, then the \textit{linear matroid} of $A$ is the pair $\left\{E,\mathcal{I}\right\}$, where $E=\{1,\cdots,n\}$, representing the $n$ columns of $A$, and 
\begin{equation*}
    \mathcal{I} = \left\{S\subseteq E \mid A^S \mbox{ are linearly independent}\right\},
\end{equation*}
where $A^S$ represents the set of columns of $A$ corresponding to the coordinates $S$.
\end{definition}

\begin{example}
Let
\begin{equation*}
    A = \left[\begin{array}{cccc}
        2 & 0 & 2 & 4 \\
        0 & 1 & 1 & 0 \\
        0 & 0 & 1 & 0 \\
        4 & 0 & 4 & 8
    \end{array}\right],
\end{equation*}
then the linear matroid of $A$ is $\left\{E,\mathcal{I}\right\}$, where $E=\left\{1,2,3,4\right\}$, representing the $4$ columns of $A$ respectively, and
\begin{equation*}
    \mathcal{I} = \left\{\emptyset, \left\{1\right\}, \left\{2\right\}, \left\{3\right\}, \left\{4\right\}, \left\{1,2\right\}, \left\{1,3\right\}, \left\{2,3\right\}, \left\{2,4\right\}, \left\{3,4\right\}, \left\{1,2,3\right\}, \left\{2,3,4\right\}\right\}.
\end{equation*}
Note that $\left\{1,4\right\}$ is not included in the matroid because the first and the fourth column of $A$ are not independent, which also excludes $\left\{1,2,4\right\}$, $\left\{1,3,4\right\}$ and $\left\{1,2,3,4\right\}$.
\end{example}

\begin{definition}
The \textit{Jacobian matroid} of a stationary VAR(1) model $\mathbf{M}_G$, denoted as $\mathcal{J}\left(\mathbf{M}_G\right)$, is the linear matroid of the Jacobian matrix.
\end{definition}

By abuse of notation, we identify $\mathcal{J}\left(\mathbf{M}_G\right)$ with $\mathcal{I}$ in Definition \ref{def:linear_matroid}, representing the set of linearly independent columns sets of the Jacobian matrix.

Next, we introduce the criteria for generic identifiability based on Jacobian matroids. The criterion is extensively introduced in \cite{sullivant2018algebraic} and \cite{hollering2021identifiability}. 



\begin{proposition}\cite[Proposition~10]{hollering2021identifiability}\label{prop:gene_iden_J_dim}
Let $\mathbf{M}_1$ and $\mathbf{M}_2$ be two stationary VAR(1) models corresponding to graphs $G_1$ and $G_2$ with the same set of nodes. Without loss of generality, assume that $\dim\left({\mathbf{M}_1}\right)\geq\dim\left({\mathbf{M}_2}\right)$. If there exists a subset $S$ of the coordinates such that
\begin{equation*}
    S\in \mathcal{J}\left(\mathbf{M}_2\right)\backslash\mathcal{J}\left(\mathbf{M}_1\right),
\end{equation*}
then $\dim\left(\mathbf{M}_1\cap\mathbf{M}_2\right)<\min\left\{\dim\left(\mathbf{M}_1\right),\dim\left(\mathbf{M}_2\right)\right\}$.
\end{proposition}

In Section \ref{sec:dim_mod}, we will prove that the dimension of a stationary VAR(1) model. i.e., the rank of the Jacobian matrix, is calculable for a subset of graphs. Therefore, we can focus specially on the identifiability of models with the same dimension, in which case the role of the two models are equivalent. Proposition \ref{prop:gene_iden_J_dim} then indicates that two models $\mathbf{M}_1$, $\mathbf{M}_2$ with the same dimension are generically identifiable if there exists a subset $S$ of the coordinates such that
\begin{equation*}
    S\in \mathcal{J}\left(\mathbf{M}_1\right)\backslash\mathcal{J}\left(\mathbf{M}_2\right) \mbox{ or } S \in \mathcal{J}\left(\mathbf{M}_2\right)\backslash\mathcal{J}\left(\mathbf{M}_1\right).
\end{equation*}


Definition \ref{def:gene_iden_max} and Proposition \ref{prop:gene_iden_J_dim} will be the main tools used throughout this paper to test generic identifiability. 

\subsection{Main results}
In this paper, we introduce maximal classes for directed graphs (formally defined in Section \ref{sec:maximal_classes}). We demonstrate that maximal classes are strongly related to the zero columns of the Jacobian matrix of the model (Proposition \ref{prop:J_G_0}), based on which we conclude that among models with the same dimension, two models with different sets of maximal classes are generically identifiable (Theorem \ref{thm:iden_maxc_dim}). For a general family of models, where dimensions may not coincide, we give a more restrictive sufficient condition for generic identifiability based on maximal classes in Theorem \ref{thm:iden_maxc}. Finally, we derive a formula to calculate the dimension of the model for a class of reflexive graphs with no reciprocal edges (Theorem \ref{thm:rank_no_multi_edges}), which enables us to characterize more models since different dimensions indicate generic identifiability by Definition \ref{def:gene_iden_max}. A summary and illustration of all the results in this paper are presented in Table \ref{tab:sum_iden_res} at the end of Section \ref{sec:iden_res}.

\section{Jacobian Structure of VAR(1) in the Stationary Setting}\label{sec:Jacobian_structure}
In order to study the Jacobian matroid of the model, we need to calculate the Jacobian matrix. The goal of this section is to provide an explicit formula for the Jacobian matrix $\mathbf{J}_G$ for any directed graph $G$. Starting by deriving the Jacobian matrix for complete graphs, we then define a projection that maps it to that of any graphs, and in the end, we present a simplified formula for $\mathbf{J}_G$ based on the projection.

For any stationary VAR(1) model $\mathbf{M}_G$, the Jacobian matrix $\mathbf{J}_G$ is a matrix of size $\left(E_G+1\right)\times \left(n(n+1)\right) / 2$, where $E_G=|\mathfrak{E}_G|$ is the number of edges in $G$. Each row of $\mathbf{J}_G$ corresponds to $\lambda_{ij}$ where $(i,j)\in\mathfrak{E}_G$ or $\omega$, and each column corresponds to $\sigma_{ab}$ where $a\leq b\in[n]$. Note that we only consider $a\leq b$ because $\sigma_{ab}=\sigma_{ba}$ for all $a,b\in[n]$.

Consider a stationary VAR(1) model $\mathbf{M}_{\overline{G}}$ with $\overline{G}=\left(V,\mathfrak{E}_{\overline{G}}\right)$ complete, i.e., $\mathfrak{E}_{\overline{G}}=\left\{(i,j) \mid i,j\in[n]\right\}$. Denote the Jacobian matrix of this model as $\mathbf{J}$, which is of size $\left(n^2+1\right)\times \left(n(n+1)\right) / 2$. Define an "extended" Jacobian matrix $\overline{\mathbf{J}}$, which is $\mathbf{J}$ with additional columns that correspond to $\sigma_{ab}$ where $a > b\in[n]$, meaning that $\overline{\mathbf{J}}$ is of size $\left(n^2+1\right)\times n^2$. The following lemma provides a formula for this $\overline{\mathbf{J}}$.

\begin{lemma}\label{lem:J_G_ext_complete}
Let $\mathbf{M}_G$ be a stationary VAR(1) model with $G$ complete, and $P_{i,j}\in M_n\left(\mathbb{R}\right)$ for all $i,j\in[n]$ be the identity matrix switching the $i^{th}$ and $j^{th}$ columns. Then the extended Jacobian matrix $\overline{\mathbf{J}}$ satisfies:
\begin{equation}\label{eq:J_extended}
    \overline{\mathbf{J}} = \overline{\mathbf{J}}\left(\Lambda\otimes\Lambda\right) + B,
\end{equation}
where
\begin{equation*}
    B = \left[\begin{array}{c}
        \left(\Sigma\Lambda\otimes I_n\right)\left(I_{n^2} + \mathbf{P}\right) \\ \vec\left(I_{n^2}\right)^T
        \end{array}\right],
\end{equation*}
and
\begin{equation*}
    \mathbf{P}=\prod_{i\geq 1}\prod_{j> i}P_{\left(i-1\right)n+j, \left(j-1\right)n+i}.
\end{equation*}
\end{lemma}
\begin{proof}
The proof is given in Appendix \ref{app:prf_lem_J_G_ext_complete}.
\end{proof}

In Lemma \ref{lem:J_G_ext_complete}, $P_{\left(i-1\right)n+j, \left(j-1\right)n+i}$ is the identify matrix exchanging the $i^{th}$ and $j^{th}$ columns of the $j^{th}$ and $i^{th}$ blocks, and hence $\mathbf{P}$ is just the indentity matrix exchanging columns $n(n-1) / 2$ times. The following example is a visualization of the function of $\mathbf{P}$.

\begin{example}\label{eg:P}
	When $n=3$, $\mathbf{P}$ does the following: for all $a,\cdots,i\in\mathbb{R}$,
	\begin{align*}
		&\left(\left[\begin{array}{ccc}
			a & b & c \\
			d & e & f \\
			g & h & i
		\end{array}\right]\otimes I_3\right) \mathbf{P} \\
        &= \left[\begin{array}{ccccccccc}
			a & & & b & & & c & & \\
			& a & & & b & & & c & \\
			&& a &&& b & & & c \\
			d & & & e & & & f & & \\
			& d & & & e & & & f & \\
			&& d &&& e & & & f \\
			g & & & h & & & i & & \\
			& g & & & h & & & i & \\
			&& g &&& h & & & i \\
		\end{array}\right]\mathbf{P}
        = \left[\begin{array}{ccccccccc}
			a & b & c & & & & & & \\
			& & & a & b & c & & & \\
			&&&&&& a & b & c \\
			d & e & f & & & & & & \\
			& & & d & e & f & & & \\
			&&&&&& d & e & f \\
			g & h & i & & & & & & \\
			& & & g & h & i & & & \\
			&&&&&& g & h & i \\
		\end{array}\right]
	\end{align*}
\end{example}

Let $\mathbf{M}_G$ be a stationary VAR(1) model with $G=\left(V,\mathfrak{E}_G\right)$, possibly not complete. Define again a corresponding extended Jacobian matrix $\overline{\mathbf{J}}_G$, which is $\mathbf{J}_G$ with additional columns that correspond to $\sigma_{ab}$ where $a > b\in[n]$. Then, $\overline{\mathbf{J}}_G$ is of size $\left(E_G+1\right)\times n^2$. Now we introduce a projection from $\overline{\mathbf{J}}$ to $\overline{\mathbf{J}}_G$, which will then be used to derive an explicit expression of $\overline{\mathbf{J}}_G$.

\begin{definition}\label{def:psi}
Let $\mathbf{M}_G$ be a stationary VAR(1) model with $G=\left(V,\mathfrak{E}_G\right)$ not necessarily complete, and $\overline{\mathbf{J}}$ be the extended Jacobian matrix for a complete graph. Define the projection
\begin{align*}
    \psi_G : M_{\left(n^2+1\right), n^2}\left(\mathbb{R}\right) & \rightarrow M_{\left(E_G+1\right), n^2}\left(\mathbb{R}\right) \\
    \overline{\mathbf{J}} &\mapsto \psi_G\left(\overline{\mathbf{J}}\right),
\end{align*}
as follows: $\psi_G$ removes from $\overline{\mathbf{J}}$ all rows in that correspond to $\lambda_{ij}$ such that $\left(i,j\right)\notin \mathfrak{E}_G$. In the remaining entries, all parameters $\lambda_{ij}$ with $\left(i,j\right)\notin \mathfrak{E}_G$ are set to zero.
\end{definition}
Under this definition, we have
\begin{equation*}
    \psi_G\left(\overline{\mathbf{J}}\right) = \overline{\mathbf{J}}_G.
\end{equation*}
The proof is given in Appendix \ref{app:prf_eq_psi}. By applying this projection to both sides of (\ref{eq:J_extended}), we get
\begin{equation*}
    \overline{\mathbf{J}}_G = \overline{\mathbf{J}}_G\left(\Lambda\otimes\Lambda\right) + \psi_G(B),
\end{equation*}
where $B$ is defined in Lemma \ref{lem:J_G_ext_complete}. We next illustrate how to apply this equation to calculate the extended Jacobian matrix of a specific graph.

\begin{example}
Consider a stationary VAR(1) model $\mathbf{M}_G$ with a directed graph $G=\left(V,\mathfrak{E}_G\right)$, where $V=\left\{1,2,3\right\}$ and $\mathfrak{E}_G=\left\{\left(1,1\right),\left(2,2\right),\left(3,3\right),\left(1,3\right),\left(2,1\right)\right\}$. The graph and the corresponding interaction matrix are given below.

\begin{minipage}{0.48\textwidth}
\centering
\includegraphics[width=0.35\textwidth]{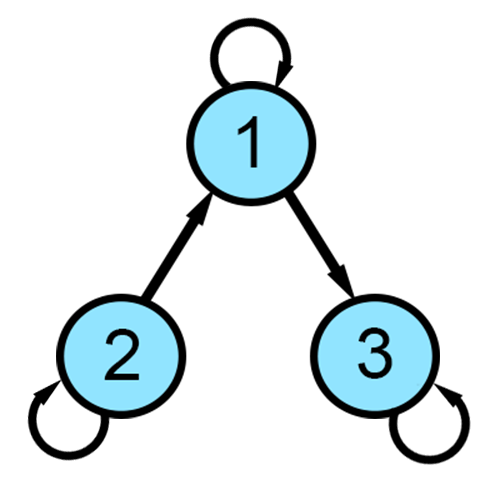}
\end{minipage}
\begin{minipage}{0.48\textwidth}
\begin{equation*}
    \Lambda = \left[\begin{array}{ccc}
        \lambda_{11} & 0 & \lambda_{13} \\
        \lambda_{21} & \lambda_{22} & 0 \\
        0 & 0 & \lambda_{33}
    \end{array}\right].
\end{equation*}
\end{minipage}                                                                                                                      

In this case, all computations can be explicit:
\begin{equation*}
    \Sigma\Lambda = \left[\begin{array}{ccc}
        \sigma_{11}\lambda_{11} + \sigma_{12}\lambda_{21} & \sigma_{12}\lambda_{22} & \sigma_{11}\lambda_{13}+\sigma_{13}\lambda_{33} \\
        \sigma_{21}\lambda_{11} + \sigma_{22}\lambda_{21} & \sigma_{22}\lambda_{22} & \sigma_{21}\lambda_{13}+\sigma_{23}\lambda_{33} \\
        \sigma_{31}\lambda_{11} + \sigma_{32}\lambda_{21} & \sigma_{32}\lambda_{22} & \sigma_{31}\lambda_{13}+\sigma_{33}\lambda_{33}
    \end{array}\right].
\end{equation*}
Then the extended Jacobian matrix $\overline{\mathbf{J}}_G$ satisfies:
\begin{equation*}
    \overline{\mathbf{J}}_G = \overline{\mathbf{J}}_GA + \psi_G\left(B\right),
\end{equation*}
where
\begin{align*}
    A &= \Lambda\otimes\Lambda \\
    &= \left[\begin{array}{ccc}
        \begin{array}{ccc}
            \lambda_{11}^2 & 0 & \lambda_{11}\lambda_{13} \\
            \lambda_{11}\lambda_{21} & \lambda_{11}\lambda_{22} & 0 \\
            0 & 0 & \lambda_{11}\lambda_{33}
        \end{array} & \mathbf{0} & \begin{array}{ccc}
            \lambda_{11}\lambda_{13} & 0 & \lambda_{13}^2 \\
            \lambda_{13}\lambda_{21} & \lambda_{13}\lambda_{22} & 0 \\
            0 & 0 & \lambda_{13}\lambda_{33}
        \end{array} \\
        \begin{array}{ccc}
        \lambda_{21}\lambda_{11} & 0 & \lambda_{21}\lambda_{13} \\
        \lambda_{21}^2 & \lambda_{21}\lambda_{22} & 0 \\
        0 & 0 & \lambda_{21}\lambda_{33}
    \end{array} & \begin{array}{ccc}
        \lambda_{22}\lambda_{11} & 0 & \lambda_{22}\lambda_{13} \\
        \lambda_{22}\lambda_{21} & \lambda_{22}^2 & 0 \\
        0 & 0 & \lambda_{22}\lambda_{33}
    \end{array} & \mathbf{0} \\
    \mathbf{0} & \mathbf{0} & \begin{array}{ccc}
            \lambda_{33}\lambda_{11} & 0 & \lambda_{33}\lambda_{13} \\
            \lambda_{33}\lambda_{21} & \lambda_{33}\lambda_{22} & 0 \\
            0 & 0 & \lambda_{33}^2
        \end{array}
    \end{array}\right],
\end{align*}
and
\begin{align*}
    & \psi_G\left(B\right) = \left[\begin{array}{cccc}
        2\left(\sigma_{11}\lambda_{11} + \sigma_{12}\lambda_{21}\right) & \sigma_{12}\lambda_{22} & \sigma_{11}\lambda_{13}+\sigma_{13}\lambda_{33} & \sigma_{12}\lambda_{22} \\
        0 & 0 & \sigma_{11}\lambda_{11} + \sigma_{12}\lambda_{21} & 0 \\
        2\left(\sigma_{21}\lambda_{11} + \sigma_{22}\lambda_{21}\right) & \sigma_{22}\lambda_{22} & \sigma_{21}\lambda_{13}+\sigma_{23}\lambda_{33} & \sigma_{22}\lambda_{22} \\
        0 & \sigma_{21}\lambda_{11} + \sigma_{22}\lambda_{21} & 0 & \sigma_{21}\lambda_{11} + \sigma_{22}\lambda_{21} \\
        0 & 0 & \sigma_{31}\lambda_{11} + \sigma_{32}\lambda_{21} & 0 \\
        1 & 0 & 0 & 0
    \end{array}\right. \\
    &\left.\begin{array}{ccccc}
        0 & 0 & \sigma_{11}\lambda_{13}+\sigma_{13}\lambda_{33} & 0 & 0 \\
        0 & \sigma_{12}\lambda_{22} & \sigma_{11}\lambda_{11} + \sigma_{12}\lambda_{21} & \sigma_{12}\lambda_{22} & 2\left(\sigma_{11}\lambda_{13}+\sigma_{13}\lambda_{33}\right) \\
        0 & 0 & \sigma_{21}\lambda_{13}+\sigma_{23}\lambda_{33} & 0 & 0 \\
        2\sigma_{22}\lambda_{22} & \sigma_{21}\lambda_{13}+\sigma_{23}\lambda_{33} & 0 & \sigma_{21}\lambda_{13}+\sigma_{23}\lambda_{33} & 0 \\
        0 & \sigma_{32}\lambda_{22} & \sigma_{31}\lambda_{11} + \sigma_{32}\lambda_{21} & \sigma_{32}\lambda_{22} & 2\left(\sigma_{31}\lambda_{13}+\sigma_{33}\lambda_{33}\right) \\
        1 & 0 & 0 & 0 & 1
    \end{array}\right].
\end{align*}
In this example, if $\left(I_{n^2} - A\right)$ is invertible, then we will have an explicit formula for $\overline{\mathbf{J}}_G$.
\end{example}

From previous arguments, it's clear that if $\left(I_{n^2}-\Lambda\otimes\Lambda\right)$ is invertible, we would be able to define $\overline{\mathbf{J}}_G$ as a function of $\Lambda$ and $\omega$. In fact, the parameters such that the matrix $\left(I_{n^2} - \Lambda\otimes\Lambda\right)$ is not invertible has Lebesgue measure zero (proof in Appendix \ref{app:prf_I_lam_invert}). Because of the generic settings, we can always assume that $\left(I_{n^2} - \Lambda\otimes\Lambda\right)$ is invertible. This argument of generality and will be applied several times in the following sections.

We now have an explicit formula for the extended Jacobian matrix.

\begin{theorem}\label{thm:J_G_ext_cal}
Let $\mathbf{M}_G$ be a stationary VAR(1) model. Then the extended Jacobian matrix is generically:
\begin{equation*}
    \overline{\mathbf{J}}_G = \psi_G(B)\left(I_{n^2} - \Lambda\otimes\Lambda\right)^{-1},
\end{equation*}
where $B$ is defined in Lemma \ref{lem:J_G_ext_complete}.
\end{theorem}
\begin{proof}
This is a direct result of previous arguments of this section.
\end{proof}

Theorem \ref{thm:J_G_ext_cal} enables us to calculate the extended Jacobian matrix $\overline{\mathbf{J}}_G$ for any directed graph $G$, meaning that we can also calculate the Jacobian matrix $\mathbf{J}_G$, since $\overline{\mathbf{J}}_G$ is just $\mathbf{J}_G$ with additional columns that are identical to a subset of columns of $\mathbf{J}_G$. Recall that the Jacobian matroid is the set of coordinates of columns that are linearly independent. In the next section, we introduce \textit{maximal classes}, which will later be used to study the Jacobian matroid.

\section{Maximal classes}\label{sec:maximal_classes}
In this section, we introduce the definition and properties of maximal classes and demonstrate that they are closely related to the zero entries of the covariance matrix $\Sigma$. We also provide an algorithm for computing the set of maximal classes from a given graph (Algorithm \ref{alg:maxc} in Appendix \ref{app:alg_mc}), together with a heuristic procedure for the inverse problem: reconstructing possible graphs from a prescribed set of maximal classes (Algorithm \ref{alg:maxc_graphs} in Appendix \ref{app:alg_mc}). Combined with Algorithm \ref{alg:sig_maxc}, which reconstructs the set of maximal classes from the support of $\Sigma$, these algorithms provide a first step toward graph reconstruction for stationary VAR(1) models and suggest directions for future work.

\subsection{Definition and properties}\label{sec:mc_def_prop}
Recall that a Strongly Connected Component (SCC) of a directed graph is a subgraph in which every node is reachable from every other node by a directed path. By abuse of notation, we also use SCC to refer to the set of nodes in such a subgraph. An SCC may consist of a single node.

The \textit{in-degree} of a node $v$ in a directed graph is the number of edges directed into $v$. We extend this notion to SCCs as follows. Let $G=(V,\mathfrak{E}_G)$ be a directed graph and $C\subset V$ be an SCC of $G$, the \textit{in-degree of the SCC} is the number of edges entering $C$ from outside, namely
\begin{equation*}
\deg^-(C) = |\left\{\left(v,v'\right)\in \mathfrak{E}_G \mid v'\in C, v\in V\backslash C \right\}|.
\end{equation*}

Throughout the paper, the in-degree of either a node or an SCC is understood to exclude self-loops.

An SCC with in-degree zero plays a distinguished role: when the graph is viewed as a network, such a component can be interpreted as a source feeding the rest of the network. We therefore call any SCC with in-degree zero a \textit{source} of the graph. A maximal class is the maximal set of descendants generated by such a source, as formalized below.
\begin{definition}\label{def:maxc}
Let $S$ be a source of a directed graph and let $i\in S$. The maximal class associated with $S$, denoted by $\mathcal{MC}_{[i]}$, is set of all nodes reachable from $i$ by a directed path.
\begin{equation*}
    \mathcal{MC}_{[i]}=\{j\in V : \text{there exists a directed path from } i \text{ to } j\}.
\end{equation*}
Equivalently, \(\mathcal{MC}_{[i]}\) is the maximal set of descendants of the source \(S\). The nodes in $S$ are called source nodes.
\end{definition}


\begin{example}\label{eg:mc}
\begin{figure}[!tb]
    \centering
    \begin{subfigure}{.48\textwidth}
  \centering
  \includegraphics[width=0.45\linewidth]{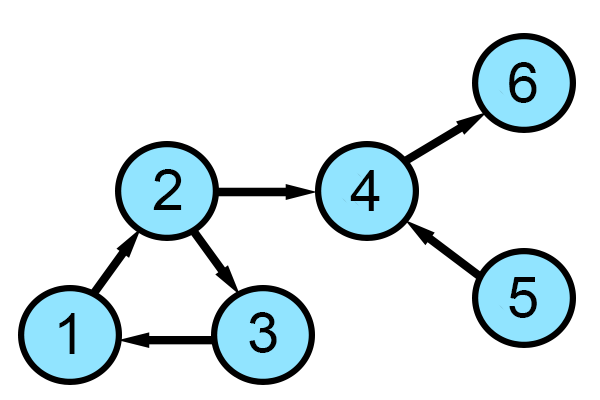}
  \caption{$G_1$}
\end{subfigure}%
\begin{subfigure}{.48\textwidth}
  \centering
  \includegraphics[width=0.45\linewidth]{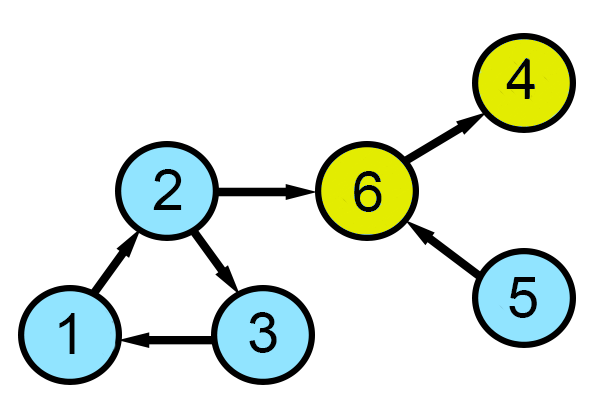}
  \caption{$G_2$}
\end{subfigure}
    \caption{Example of directed graphs}
    \label{fig:Eg_maxc}
\end{figure}

The SCCs of the graph $G_1$ in Figure \ref{fig:Eg_maxc} are:
\begin{equation*}
    \left\{1,2,3\right\},\{4\},\{5\},\{6\},
\end{equation*}
It's clear that $\left\{1,2,3\right\}$ and $\{5\}$ have in-degree zero, and hence are the sources. Therefore, the set of maximal classes is:
\begin{equation*}
    \left\{\left\{\mathbf{1},\mathbf{2},\mathbf{3},4,6\right\}, \left\{4,\mathbf{5},6\right\}\right\},
\end{equation*}
where the nodes in bold are source nodes of the respective maximal classes. In other words, there are two maximal classes in this graph:
\begin{equation*}
    \mathcal{MC}_{[1]} = \mathcal{MC}_{[2]} = \mathcal{MC}_{[3]} = \{1,2,3,4,6\} \mbox{, and } \mathcal{MC}_{[5]} = \{4,5,6\}.
\end{equation*}
\end{example}

Note that SCCs have also been used to analyze cyclic graphical models; for example, \cite[Section~6]{bongers2021foundations} use them to decompose cyclic structural causal models into feedback blocks that can be collapsed into solution functions under suitable unique-solvability conditions. Although our use of SCCs is different, this connection highlights the broader relevance of SCC-based decompositions in graphical models with cycles.

MC is closely related to the concept of \textit{trek} in linear SEM and continuous LM.
\begin{definition}{\cite[Definition~4.1]{drton2018algebraic}}
A trek $\tau$ from node $i$ to node $j$ is a semi-walk from $i$ to $j$ of the form
\begin{enumerate}
    \item[(a)] 
    \(i \leftarrow i_l \leftarrow \cdots \leftarrow i_1 \leftarrow i_0
    \longleftrightarrow j_0 \rightarrow j_1 \rightarrow \cdots \rightarrow j_r \rightarrow j,\)
    or
    \item[(b)] 
    \(i \leftarrow i_l \leftarrow \cdots \leftarrow i_1 \leftarrow i_0
    \longrightarrow j_1 \rightarrow \cdots \rightarrow j_r \rightarrow j.\)
\end{enumerate}
A trek has a left- and a right-hand side, denoted
\(\operatorname{left}(\tau)\) and \(\operatorname{right}(\tau)\), respectively.
We have \(\operatorname{left}(\tau)=\{i_0,\ldots,i_l,i\}\) and
\(\operatorname{right}(\tau)=\{j_0,\ldots,j_r,j\}\) in case (a), and
\(\operatorname{left}(\tau)=\{i_0,\ldots,i_l,i\}\) and
\(\operatorname{right}(\tau)=\{i_0,j_1,\ldots,j_r,j\}\) in case (b).
In case (b), the top node \(i_0\) belongs to both sides. Both the left- and the right- hand side are allowed to contain no edges.
\end{definition}

Under this definition, a maximal class is a maximal set of nodes that are pairwise connected by treks with the same top, or source node. Indeed, for any directed graph $G=(V,\mathfrak{E}_G)$ and any two nodes $i,j\in V$, the following statements are equivalent
\begin{enumerate}
\item \(i\) and \(j\) belong to the same maximal class;
\item there exists a trek between \(i\) and \(j\);
\item either there exists a node \(k\in V\setminus\{i,j\}\) such that directed paths from \(k\) to \(i\) and from \(k\) to \(j\) both exist, or there exists a directed path from \(i\) to \(j\) or from \(j\) to \(i\).
\end{enumerate}

That said, we still formulate our results in terms of maximal classes because, as shown in the following sections, they capture precisely the information encoded by the support of a generic $\Sigma$ and by the locations of zero columns in the Jacobian matrix; treks do not play this role directly. An MC contains less information about the graph than a trek: it is an unordered set of nodes, whereas a trek is specified by a directed semi-path. For example, in the graphs shown in Figure \ref{fig:Eg_maxc}, swapping nodes 4 and 6 in $G_1$ yields $G_2$. This operation leaves the MCs unchanged, while changing the trek between nodes $5$ and $6$. More importantly, the collection of maximal classes does not determine the graph uniquely, whereas the full collection of treks is sufficient to reconstruct the graph.

The following proposition proves the uniqueness of the set of maximal classes for any graph, which is a fundamental property of maximal classes.

\begin{proposition}\label{prop:maxc_uniq}
Let $G$ be any directed graph, then the set of maximal classes associated with $G$ is unique.
\end{proposition}
\begin{proof}
First, two maximal classes of a graph cannot have the same source. This is because by definition \ref{def:maxc}, a maximal class is the set of all nodes reachable by the source, and if two maximal classes have the same source, then they must be the same maximal class. Moreover, for any directed graph $G$, the set of sources is unique since the in-degree of any node or SCC is uniquely defined. Therefore, each source generates a unique maximal class. Hence, the set of maximal classes is unique.
\end{proof}

Using the uniqueness of maximal classes, we provide Algorithm \ref{alg:maxc} in Appendix \ref{app:alg_mc} to compute them from a given graph. The algorithm first identifies the strongly connected components, and then, for each SCC with in-degree zero, extracts the corresponding maximal class by depth-first search. Conversely, to reconstruct graphs associated with a given set of maximal classes, we also provide a heuristic procedure in Algorithm \ref{alg:maxc_graphs} in Appendix \ref{app:alg_mc}.





\subsection{Maximal class and the support of \texorpdfstring{$\Sigma$}{Lg}}\label{sec:sig_to_mc}
In this section, we describe the relationship between maximal classes and the covariance matrix $\Sigma$ and propose an explicit procedure to reconstruct the set of maximal classes from the support of $\Sigma$.

We start with a Lemma on the values of the parameter $\Lambda$.

\begin{lemma}\label{lem:M_P_M_S}
Let $\mathbf{M}_G$ be a stationary VAR(1) model, and $\lambda_{ij}^{[k]}=\left(\Lambda^k\right)_{ij}$, for all $k\in\mathbb{N}$, then generically
\begin{equation*}
    \Lambda \in M^P \cap M^S,
\end{equation*}
where
\begin{align*}
    M^P := &\left\{\Lambda\in M_n\left(\mathbb{R}\right) \mid \forall i,j\in V, \exists \mbox{ a directed path in }G \mbox{ from } i \mbox{ to }j \right. \\
    & \left. \Rightarrow \exists K\in\mathbb{N}, s.t. \lambda_{ij}^{\left[k\right]}\neq 0, \forall k > K\right\}, \\
    M^S := &\left\{\Lambda\in M_n\left(\mathbb{R}\right) \mid \forall i,j\in V, \sum_{k=0}^{+\infty}\sum_{a=1}^n\lambda_{ai}^{\left[k\right]}\lambda_{aj}^{\left[k\right]} = 0 \Leftrightarrow \forall k\in\mathbb{N},a\in\left[n\right], \lambda_{ai}^{\left[k\right]}\lambda_{aj}^{\left[k\right]} = 0,\right. \\ 
    &\left.\mbox{and }\sum_{s=1}^n\sum_{k=0}^{+\infty}\sum_{l=1}^n\lambda_{li}^{\left[k\right]}\lambda_{ls}^{\left[k\right]}\lambda_{sj} = 0 \Leftrightarrow \forall k\in\mathbb{N}, s,l\in\left[n\right], \lambda_{li}^{\left[k\right]}\lambda_{ls}^{\left[k\right]}\lambda_{sj} = 0\right\}.
\end{align*}
\end{lemma}
\begin{proof}
The proof is given in Appendix \ref{app:prf_lem_M_P_M_S}.
\end{proof}

In fact, $M^P$ is a set of matrices $\Lambda$ such that for all $i,j\in[n]$, if there exists a directed path from $i$ to $j$, then after a certain power, the corresponding coefficient $\lambda_{ij}$ stays non-zero. $M^S$ is a set of matrices $\Lambda$ such that the infinite sum of certain elements equals zero is equivalent to requiring that each part of the sum equals zero individually. In other words, we prevent cancellations among elements of different powers of $\Lambda$ in this set.

Because we consider generic properties of the parameters, we assume that $\Lambda\in M^S\cap M^P$. It is reasonable because Lemma \ref{lem:M_P_M_S} implies that $M_G\backslash \left(M^S\cap M^P\right)$ has Lebesgue measure zero. This type of genericity argument was already used in Section \ref{sec:Jacobian_structure}. The assumption is needed to establish the relationship between maximal classes and the support of \(\Sigma\) in Theorem \ref{thm:supp_sig_maxc}, and to derive the Jacobian support characterization in Lemma \ref{lem:supp_Sig_Lam_maxc}. The necessity for this restriction arises from the infinite-series representation of the stationary covariance in the discrete-time setting as shown in (\ref{eq:vec_sig_sum}); in contrast, no such assumption is needed for deriving similar results in linear SEMs (\cite[Theorem~4.2]{drton2018algebraic}) or continuous LMs (\cite[Corollary~2.3]{varando2020graphical}).

\begin{theorem}\label{thm:supp_sig_maxc}
Let $\mathbf{M}_G$ be a stationary VAR(1) model with corresponding directed graph $G=\left(V,\mathfrak{E}_G\right)$, then for any two nodes $i,j\in [n]$, $\sigma_{ij}=\sigma_{ji}=0$ if and only if $i$ and $j$ do not belong to the same maximal class of $G$.
\end{theorem}
\begin{proof}
For all $s,t\in\left\{1,\cdots,n\right\}$, let $\lambda_{st}^{\left[k\right]}=\left(\Lambda^k\right)_{st}$. Assume $i$ and $j$ do not belong to the same maximal class, i.e., by definition, there is no directed path between $i$ and $j$, and additionally, they don't have the same "ancestor". Therefore, by the properties of matrix multiplication, for all $k\in\mathbb{N}$,
\begin{equation*}
    \lambda_{ij}^{\left[k\right]} = \lambda_{ji}^{\left[k\right]} = 0,
\end{equation*}
and for all $l\neq i,j\in \left[n\right]$,
\begin{equation*}
    \lambda_{li}^{\left[k\right]} = 0 \mbox{ or } \lambda_{lj}^{\left[k\right]} = 0.
\end{equation*}

In Section \ref{sec:Jacobian_structure} we argued that $\left(I_{n^2} - \Lambda\otimes\Lambda\right)$ was generically invertible, which implies that $\left(I_{n^2} - \Lambda^T\otimes\Lambda^T\right)$ is generically invertible since $\left(I_{n^2} - \Lambda^T\otimes\Lambda^T\right)$ is just the transpose of $\left(I_{n^2} - \Lambda\otimes\Lambda\right)$. Hence, from (\ref{eq:vec_Sigma}),
\begin{equation}\label{eq:vec_sig_sum}
\begin{aligned}
    &\vec\left(\Sigma\right) = \left(I_{n^2}-\Lambda^T\otimes\Lambda^T\right)^{-1}\vec\left(\omega I_n\right) 
    = \omega\left(\sum_{k=0}^{+\infty}\left(\Lambda^T\otimes\Lambda^T\right)^k\right) \vec\left(I_n\right) \\
    &= \omega\left(\sum_{k=0}^{+\infty}\left(\Lambda^T\right)^k\otimes \left(\Lambda^T\right)^k\right)\vec\left(I_n\right) 
    = \omega\left(\sum_{k=0}^{+\infty}\left(\Lambda^k\right)^T\otimes \left(\Lambda^k\right)^T\right)\vec\left(I_n\right) \\
    &= \omega\sum_{k=0}^{+\infty}\left[\begin{array}{ccccccc}
        \lambda_{11}^{\left[k\right]}\left(\Lambda^k\right)^T & \cdots & \lambda_{i1}^{\left[k\right]}\left(\Lambda^k\right)^T & \cdots & \lambda_{j1}^{\left[k\right]}\left(\Lambda^k\right)^T & \cdots & \lambda_{n1}^{\left[k\right]}\left(\Lambda^k\right)^T\\
        \vdots & \ddots & \vdots & \ddots & \vdots & \ddots & \vdots \\
        \lambda_{1i}^{\left[k\right]}\left(\Lambda^k\right)^T & \cdots & \lambda_{ii}^{\left[k\right]}\left(\Lambda^k\right)^T & \cdots & 0 & \cdots & \lambda_{ni}^{\left[k\right]}\left(\Lambda^k\right)^T \\
        \vdots & \ddots & \vdots & \ddots & \vdots & \ddots & \vdots \\
         \lambda_{1j}^{\left[k\right]}\left(\Lambda^k\right)^T & \cdots & 0 & \cdots & \lambda_{jj}^{\left[k\right]}\left(\Lambda^k\right)^T & \cdots & \lambda_{nj}^{\left[k\right]}\left(\Lambda^k\right)^T \\
         \vdots & \ddots & \vdots & \ddots & \vdots & \ddots & \vdots \\
         \lambda_{1n}^{\left[k\right]}\left(\Lambda^k\right)^T & \cdots & \lambda_{in}^{\left[k\right]}\left(\Lambda^k\right)^T & \cdots & \lambda_{jn}^{\left[k\right]}\left(\Lambda^k\right)^T & \cdots & \lambda_{nn}^{\left[k\right]}\left(\Lambda^k\right)^T\\
    \end{array}\right]\left[\begin{array}{c}
        1 \\ 0 \\ \vdots \\ 0 \\ 1 \\ \vdots \\ 0 \\ \vdots \\ 1
    \end{array}\right],
\end{aligned}
\end{equation}
where the second equality is true because the $\rho(\Lambda) < 1$ implies $\rho(\Lambda^T \otimes \Lambda^T) < 1$. Therefore,
\begin{equation*}
    \sigma_{ij} = \omega\sum_{k=0}^{+\infty}\sum_{l=1}^n\lambda_{li}^{\left[k\right]}\lambda_{lj}^{\left[k\right]} = 0 = \sigma_{ji},
\end{equation*}

On the other hand, by Lemma \ref{lem:M_P_M_S}, if for $i,j\in V$, $\sigma_{ij} = \sigma_{ji} = 0$, then generically $\lambda_{li}^{\left[k\right]}\lambda_{lj}^{\left[k\right]} = 0$ for all $k\in\mathbb{N}$ and $l\in V$. Therefore, there is no directed path between $i$ and $j$, and they do not have the same ancestor. Thus, they do not belong to the same maximal class.
\end{proof}

Theorem \ref{thm:supp_sig_maxc} shows that the maximal classes of the underlying graph are encoded in the support pattern of \(\Sigma\). We now make this correspondence explicit and turn it into a reconstruction procedure.

Before stating the result, we introduce the following notation: For all $i\in [n]$, define the set of nodes that belong to the same maximal class as $i$: 
\begin{equation*}
    \mathcal{C}_i := \{j \mid j\in[n], i,j \mbox{ belong to the same maximal class}\}.
\end{equation*}
And from Theorem \ref{thm:supp_sig_maxc},
\begin{equation*}
    \mathcal{C}_i = \{j \mid j\in[n], \sigma_{ij}\neq 0\}.
\end{equation*}

Note that for all $i\in[n]$, $\mathcal{C}_i$ exists, but may not be a maximal class. For example, in $G_1$ of Figure \ref{fig:Eg_maxc}, $\mathcal{C}_4 = \{1,2,3,4,5,6\}$ is not a maximal class, while $\mathcal{C}_5 = \{4,5,6\}$ is a maximal class.

The following theorem identifies the maximal classes as specific subsets determined by the support of $\Sigma$.

\begin{proposition}\label{prop:set_maxc}
Let $\mathbf{M}_G$ be a stationary VAR(1) model. Then the set of maximal classes $\mathfrak{MC}$ satisfies
\begin{equation*}
    \mathfrak{MC} =\mathfrak{C} \backslash \mathfrak{C}',
\end{equation*}
where
\begin{align*}
    \mathfrak{C} &= \left\{\mathcal{C}_i \mid i\in[n]\right\}, \\
    \mathfrak{C}' &= \left\{\mathcal{C}_i\in\mathfrak{C} \mid i\in[n],\mbox{ and } \exists k,l\in\mathcal{C}_i\ s.t.\ \sigma_{kl}=0\right\}.
\end{align*}
\end{proposition}
\begin{proof}
First, we prove that $\mathfrak{MC}\subseteq\mathfrak{C}$. 
Consider $\mathcal{MC}\in\mathfrak{MC}$, there exists a source node $i\in[n]$, s.t. $\mathcal{MC}=\mathcal{MC}_{[i]}$. By Theorem \ref{thm:supp_sig_maxc}, for all $j\in\mathcal{MC}_{[i]}, \sigma_{ij}\neq 0$, hence $j\in\mathcal{C}_i$. Therefore,
\begin{equation*}
    \mathcal{MC}_{[i]} \subseteq \mathcal{C}_i.
\end{equation*}
Assume $\exists j\in \mathcal{C}_i\backslash \mathcal{MC}_{[i]}$, then one of the following circumstances is true:

\vspace{0.5em}
\noindent 1. There exists a directed path from $i$ to $j$, then $j\in\mathcal{MC}_{[i]}$, which is a contradiction.

\vspace{0.5em}
\noindent 2. There exists a directed path from $j$ to $i$, then $\deg^-(i)\neq 0$, which means that the source of $\mathcal{MC}_{[i]}$ is a SCC containing $i$. Hence either $j$ belongs to the SCC, then $j\in\mathcal{MC}_{[i]}$, which is a contradiction, or $j$ does not belong to the SCC, then the in-degree of the SCC is non-zero, i.e., the SCC containing $i$ is not the source of $\mathcal{MC}_{[i]}$, which is also a contradiction.

\vspace{0.5em}
\noindent 3. There exists another node $c\in[n]$ s.t. there are paths from $c$ to $i$, and from $c$ to $j$ respectively, then $\deg^-(i)\neq 0$, which means that the source of $\mathcal{MC}_{[i]}$ is a SCC containing $i$. Hence either $c$ belongs to the SCC, then $j\in\mathcal{MC}_{[i]}$, which is a contradiction, or $c$ does not belong to the SCC, then the in-degree of the SCC is non-zero, i.e., the SCC containing $i$ is not the source of $\mathcal{MC}_{[i]}$, which is also a contradiction.

Therefore
\begin{equation*}
    \mathcal{MC}_{[i]} = \mathcal{C}_i.
\end{equation*}
This proves that $\mathfrak{MC}\subseteq\mathfrak{C}$.

On the other hand, consider $\mathcal{C}_i\in\mathfrak{C}'$ such that $\mathcal{C}_i\notin \mathfrak{MC}$. Then $\mathcal{C}_i$ is not a maximal class, because it contains two nodes that do not belong to the same maximal class by definition. Therefore 
\begin{equation*}
    \mathfrak{MC} \subseteq \mathfrak{C}\backslash \mathfrak{C}'.
\end{equation*}
Let $\mathcal{C}_i\in\mathfrak{C}\backslash\mathfrak{C}'$ for certain $i\in[n]$ and assume that $\mathcal{C}_i\notin \mathfrak{MC}$. We know that $i$ must belong to one maximal class in $\mathfrak{MC}$. Denote the maximal class containing $i$ as $\mathcal{MC}_i$. Then by Theorem \ref{thm:supp_sig_maxc}, 
\begin{equation*}
    \mathcal{MC}_i \subseteq \mathcal{C}_i.
\end{equation*}
Assume $k\in \mathcal{C}_i \backslash \mathcal{MC}_i$. Note that $|\mathcal{MC}_i|>1$ because the graph is weakly connected. Then $\exists l\neq i\in\mathcal{MC}_i$ s.t. $\sigma_{kl}=0$, because otherwise $k$ belongs to the same maximal class with every node in $\mathcal{MC}_i$, in particular with the source node, meaning that $k\in\mathcal{MC}_i$, which is a contradiction. However, $k,l\in\mathcal{C}_i$, and $\sigma_{kl}=0$ means that $\mathcal{C}_i\in\mathfrak{C}'$, which is also a contradiction. Therefore, 
\begin{equation*}
    \mathcal{MC}_i = \mathcal{C}_i,
\end{equation*}
contradicting the fact that $\mathcal{C}_i\notin \mathfrak{MC}$. 

In conclusion,
\begin{equation*}
    \mathfrak{MC} = \mathfrak{C}\backslash\mathfrak{C}'.
\end{equation*}
\end{proof}

Algorithm \ref{alg:sig_maxc} is then obtained via applying the criterion in Proposition \ref{prop:set_maxc}.

\begin{algorithm}[H]
\caption{Algorithm for the reconstruction of maximal classes via the support of $\Sigma$}\label{alg:sig_maxc}
\hspace*{\algorithmicindent} \textit{Input:} $\mbox{supp}(\Sigma)$ \\
\hspace*{\algorithmicindent} \textit{Output:} the list of maximal classes of $G$: $\mathfrak{MC}$
\begin{algorithmic}[1]
\State $\mathfrak{MC}\gets$ an empty list of lists
\ForAll {nodes $i\in[n]$}
\State $\mathcal{C}_i\gets\{j \mid j\in[n], \sigma_{ij}\neq 0\}$
\State add $\mathcal{C}_i$ to $\mathfrak{MC}$
\EndFor
\ForAll {$k,l\in[n]\ s.t.\ k<l$ and $\sigma_{kl}==0$}
\ForAll{$\mathcal{C}_i\in\mathfrak{MC}$}
\If {$k,l\in\mathcal{C}_i$}
\State $\mathfrak{MC}\gets\mathfrak{MC} \backslash \mathcal{C}_i$
\EndIf
\EndFor
\EndFor
\ForAll{$i\in[|\mathfrak{MC}|]$} \Comment{At this point, there might be repetitive elements in $\mathfrak{MC}$}
\ForAll{$i<j\leq |\mathfrak{MC}|$}
\If{$\mathfrak{MC}[i] == \mathfrak{MC}[j]$}
\State $\mathfrak{MC} \gets \mathfrak{MC} \backslash \mathfrak{MC}[i]$ \Comment{Keep only one of the repetitive elements}
\EndIf
\EndFor
\EndFor
\State \Return $\mathfrak{MC}$
\end{algorithmic}
\end{algorithm}

Theorem \ref{thm:supp_sig_maxc} and Proposition \ref{prop:set_maxc} imply that there is a bijection between the support of the covariance matrix $\Sigma$ and the set of maximal classes of the underlying graph. In this sense, the set of maximal classes captures exactly the graphical information that is recoverable from the support of $\Sigma$. We emphasize that this statement does not extend to treks. In general, the support of \(\Sigma\) does not allow one to reconstruct the set of treks in the graph.

Additionally, by Algorithm \ref{alg:maxc_graphs} in Appendix \ref{app:alg_mc}, we are able to reconstruct a list of possible graphs admitting this set of maximal classes. The following example is an illustration of how we could apply the results above.

\begin{example}
Let $n=3$, assume the support of $\Sigma$ is
\begin{align*}
    \Sigma = \left[\begin{array}{ccc}
        \ast & \ast & 0 \\
        \ast & \ast & \ast \\
        0 & \ast & \ast 
    \end{array}\right],
\end{align*}
where $\ast$ represents a non-zero element. In this case, 
\begin{equation*}
    \mathfrak{C} = \left\{\{1,2\},\{1,2,3\},\{2,3\}\right\},
\end{equation*}
and 
\begin{equation*}
    \mathfrak{C}' = \left\{\{1,2,3\}\right\},
\end{equation*}
because $\sigma_{13} = 0$. Hence, the set of maximal classes is 
\begin{equation*}
    \mathfrak{MC} = \mathfrak{C}\backslash \mathfrak{C}' = \left\{\{1,2\},\{2,3\}\right\}.
\end{equation*}
Note that because the node $2$ is in both of the two maximal classes, then it cannot be a source node in either of them, because otherwise there would have been only one maximal class with source $2$. Hence, $1$ and $3$ are the sources in respective maximal classes. Figure \ref{fig:Eg_sig_mc} shows the only possible graph $G=\left(V,\mathfrak{E}_G\right)$ associated with this $\Sigma$, i.e., $V=\{1,2,3\}$, $\mathfrak{E}_G = \left\{(1,1),(2,2),(3,3),(1,2),(3,2)\right\}$.
\begin{figure}[!tb]
    \centering
    \includegraphics[width=0.2\textwidth]{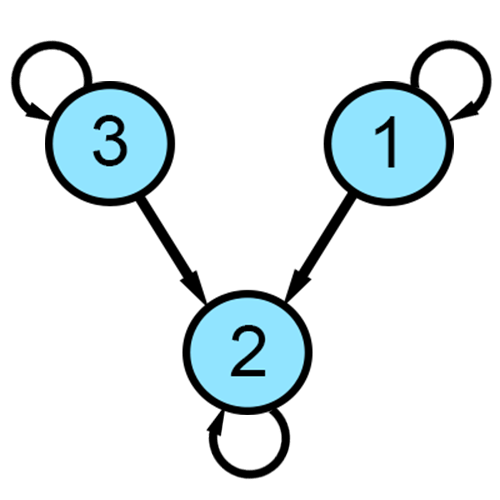}
    \caption{The graph associated with maximal classes $\left\{\{1,2\},\{2,3\}\right\}$.}
    \label{fig:Eg_sig_mc}
\end{figure}
\end{example}

\section{Identifiability results}\label{sec:iden_res}
This section presents identifiability results for stationary VAR(1) models. They are divided into two parts: one of them is based on maximal classes (Section \ref{sec:iden_maxc}), and the other one examines the dimensions of the models (Section \ref{sec:dim_mod}). Moreover, we provide a summary and illustration of the results at the end of this section.

\subsection{Maximal class characterization}\label{sec:iden_maxc}
In this section, we propose two criteria for generic identifiability using maximal classes, one of which only applies to models with the same dimension, while the other one is generialized to models with any dimension. Before presenting the main results, some lemmas and propositions describing the relationship between maximal classes and the parameters of the model as well as the rank of the Jacobian matrix are introduced. These results are not only necessary in proving the main theorems, but also strengthen our understanding of the nature of maximal classes and the Jacobian matrix.

\begin{lemma}\label{lem:supp_Sig_Lam_maxc}
Let $\mathbf{M}_G$ be a stationary VAR(1) model with corresponding directed graph $G=\left(V,\mathfrak{E}_G\right)$, then for any two nodes $i,j\in [n]$, $\left(\Sigma\Lambda\right)_{ij} = \left(\Sigma\Lambda\right)_{ji} = 0$ if and only if $i$ and $j$ do not belong to the same maximal class of $G$.
\end{lemma}
\begin{proof}
By the proof of Theorem \ref{thm:supp_sig_maxc}, we know that for all $a,b\in V$, 
\begin{equation*}
    \sigma_{ab} = \omega\sum_{k=0}^{+\infty}\sum_{l=1}^n\lambda_{la}^{\left[k\right]}\lambda_{lb}^{\left[k\right]}.
\end{equation*}
Therefore, if $i$ and $j$ do not belong to the same maximal class, 
\begin{equation*}
    \left(\Sigma\Lambda\right)_{ij} = \sum_{s=1}^n\sigma_{is}\lambda_{sj} = \omega\sum_{s=1}^n\sum_{k=0}^{+\infty}\sum_{l=1}^n\lambda_{li}^{\left[k\right]}\lambda_{ls}^{\left[k\right]}\lambda_{sj} = 0.
\end{equation*}
Because if $\left(\Sigma\Lambda\right)_{ij} \neq 0$, then $\exists k\in\mathbb{N}, s,l\in V$ s.t. $\lambda_{li}^{\left[k\right]}\lambda_{ls}^{\left[k\right]}\lambda_{sj} \neq 0$, i.e., there exist the following two paths in $G$:
\begin{equation*}
    l\rightsquigarrow i, \mbox{ and }l \rightsquigarrow s \rightarrow j,
\end{equation*}
where $\rightsquigarrow$ represents a directed path, not necessarily a single edge. In this case, $i$ and $j$ have the same ancestor $l$, thus belong to the same maximal class, which contradicts the assumption. Similarly,
\begin{equation*}
    \left(\Sigma\Lambda\right)_{ji} = 0.
\end{equation*}

On the other hand, if $\left(\Sigma\Lambda\right)_{ij} = \omega\sum_{s=1}^n\sum_{k=0}^{+\infty}\sum_{l=1}^n\lambda_{li}^{\left[k\right]}\lambda_{ls}^{\left[k\right]}\lambda_{sj} = 0$, then by Lemma \ref{lem:M_P_M_S}, generically,
\begin{equation*}
    \forall k\in\mathbb{N}, s,l\in\left[n\right], \lambda_{li}^{\left[k\right]}\lambda_{ls}^{\left[k\right]}\lambda_{sj} = 0,
\end{equation*}
Let $s=j$, since $\lambda_{jj}\neq 0$,
\begin{equation*}
     \forall k\in\mathbb{N}, l\in\left[n\right], \lambda_{li}^{\left[k\right]}\lambda_{ls}^{\left[k\right]} = 0.
\end{equation*}
By Lemma \ref{lem:M_P_M_S}, it implies that $i$ and $j$ do not have the same ancestor, and they are not connected by directed paths, i.e., $i$ and $j$ do not belong to the same maximal class.
\end{proof}

The following proposition highlights the link between maximal classes and the null columns of the Jacobian matrix, which is the foundation of characterizing models using maximal classes.

\begin{proposition}\label{prop:J_G_0}
Let $\mathbf{M}_G$ be a stationary VAR(1) model with a corresponding directed graph $G=\left(V,\mathfrak{E}_G\right)$, then the Jacobian matrix satisfies: $\forall i\leq j\in [n]$, $\mathbf{J}_G^{\left[\sigma_{ij}\right]} = \mathbf{0}$ if and only if $i$ and $j$ do not belong to the same maximal class. Here $\mathbf{J}_G^{\left[\sigma_{ij}\right]}$ represents the column of $\mathbf{J}_G$ that corresponds to $\sigma_{ij}$.
\end{proposition}
\begin{proof}
Firstly, we prove that if $i$ and $j$ do not belong to the same maximal class, then $\mathbf{J}_G^{\left[\sigma_{ij}\right]} = \mathbf{0}$. By Theorem \ref{thm:supp_sig_maxc} and Lemma \ref{lem:supp_Sig_Lam_maxc}, we know that
\begin{equation*}
    \sigma_{ij} = \sigma_{ji} = 0, \mbox{ and } \left(\Sigma\Lambda\right)_{ij} = \left(\Sigma\Lambda\right)_{ji} = 0.
\end{equation*}
Recall from Theorem \ref{thm:J_G_ext_cal} that
\begin{equation*}
    \overline{\mathbf{J}}_G = \psi_G\left(B\right)\left(I_{n^2} - \Lambda\otimes\Lambda\right)^{-1},
\end{equation*}
where
\begin{align*}
    B &= \left[\begin{array}{c}
        \left(\Sigma\Lambda\otimes I_n\right)\left(I_{n^2} + \mathbf{P}\right) \\
        \vec\left(I_{n^2}\right)^T
    \end{array}\right].
\end{align*}
Therefore,
\begin{equation*}
    \mathbf{J}_G^{\left[\sigma_{ij}\right]} = \overline{\mathbf{J}}_G^{\left[\sigma_{ij}\right]} = \psi_G\left(B\right)\left(\left(I_{n^2} - \Lambda\otimes\Lambda\right)^{-1}\right)^{\left[\sigma_{ij}\right]},
\end{equation*}
where
\begin{align*}
    \left(\left(I_{n^2} - \Lambda\otimes\Lambda\right)^{-1}\right)^{\left[\sigma_{ij}\right]} = \sum_{k=0}^{+\infty} &\left[\begin{array}{cccccc}
        \lambda_{1i}^{\left[k\right]}\lambda_{1j}^{\left[k\right]} & \cdots & \lambda_{1i}^{\left[k\right]}\lambda_{nj}^{\left[k\right]} & \lambda_{2i}^{\left[k\right]}\lambda_{1j}^{\left[k\right]} & \cdots & \lambda_{2i}^{\left[k\right]}\lambda_{nj}^{\left[k\right]}
    \end{array}\right. \\
    &\left. \begin{array}{cccc}
        \cdots & \lambda_{ni}^{\left[k\right]}\lambda_{1j}^{\left[k\right]} & \cdots & \lambda_{ni}^{\left[k\right]}\lambda_{nj}^{\left[k\right]}
    \end{array}\right]^T.
\end{align*}
Recall that $\psi_G$ removes the rows in $\overline{\mathbf{J}}$ that correspond to $\lambda_{ab}$, where $\left(a,b\right)\notin \mathfrak{E}_G$. For all $s,t \in \left[n\right]$ s.t. $\left(s,t\right)\in \mathfrak{E}_G$, denote the row in $\psi_G\left(B\right)$ that corresponds to $\lambda_{st}$ as $\psi_G\left(B\right)^{\left(\lambda_{st}\right)}$, and the element in $\mathbf{J}_G^{\left[\sigma_{ij}\right]}$ that corresponds to $\lambda_{st}$ as $\mathbf{J}_G^{\left[\lambda_{st},\sigma_{ij}\right]}$. Then
\begin{align*}
    \psi_G\left(B\right)^{\left(\lambda_{st}\right)} = &\left[
    \begin{array}{ccc}
        \begin{array}{ccccc}
            0 & \cdots & \left(\Sigma\Lambda\right)_{s1} & \cdots & 0
        \end{array} & \cdots & \begin{array}{ccccc}
            \left(\Sigma\Lambda\right)_{s1} & \cdots & \left(\Sigma\Lambda\right)_{st} & \cdots &  \left(\Sigma\Lambda\right)_{sn}
        \end{array}
    \end{array}\right. \\
    &\left.\begin{array}{cccccc}
           \cdots & 0 & \cdots & \left(\Sigma\Lambda\right)_{sn} & \cdots & 0
        \end{array}
    \right].
\end{align*}
Therefore,
\begin{align*}
    \mathbf{J}_G^{\left[\lambda_{st},\sigma_{ij}\right]} &= \psi_G\left(B\right)^{\left(\lambda_{st}\right)}\left(\left(I_{n^2} - \Lambda\otimes\Lambda\right)^{-1}\right)^{\left[\sigma_{ij}\right]} \\
    &= \sum_{k=0}^{+\infty}\left(\sum_{a=1}^n\left(\Sigma\Lambda\right)_{sa}\lambda_{ai}^{\left[k\right]}\lambda_{tj}^{\left[k\right]} + \sum_{b=1}^n\left(\Sigma\Lambda\right)_{sb}\lambda_{ti}^{\left[k\right]}\lambda_{bj}^{\left[k\right]}\right).
\end{align*}
Consider the first term of the sum: for all $a\in\left[n\right]$ and $k\in\mathbb{N}$, if $\left(\Sigma\Lambda\right)_{sa}\lambda_{ai}^{\left[k\right]}\lambda_{tj}^{\left[k\right]} \neq 0$, then from the fact that $\left(\Sigma\Lambda\right)_{sa}\neq 0$, we know by Lemma \ref{lem:supp_Sig_Lam_maxc} that $s$ and $a$ belong to the same maximal class, i.e., there is a directed path between $s$ and $a$ (Case \ref{case:1_prf_J_G_maxc}), or there exists a node $l\neq s,a\in\left[n\right]$ s.t. there are directed paths from $l$ to $s$ and $a$ respectively (Case \ref{case:2_prf_J_G_maxc}). Additionally, there exist direct paths among the following nodes:
\begin{equation*}
    s\rightarrow t\rightsquigarrow j \mbox{, and } a\rightsquigarrow i.
\end{equation*}
\begin{itemize}[leftmargin=*]
    \item[] \begin{case}\label{case:1_prf_J_G_maxc} without loss of generality, assume the path is from $s$ to $a$, then there exist directed paths: $s\rightsquigarrow t\rightsquigarrow j$, and $s\rightsquigarrow a\rightsquigarrow i$. It contradicts the assumption. \end{case}
    \item[] \begin{case}\label{case:2_prf_J_G_maxc} assume there exist directed paths: $l\rightsquigarrow s$, and $l\rightsquigarrow a$, then there exist directed paths: $l\rightsquigarrow a \rightsquigarrow i$, and $l\rightsquigarrow s\rightarrow t\rightsquigarrow j$. It also contradicts the assumption.\end{case}
\end{itemize}
Therefore, for all $a\in\left[n\right]$ and $k\in\mathbb{N}$, $\left(\Sigma\Lambda\right)_{sa}\lambda_{ai}^{\left[k\right]}\lambda_{tj}^{\left[k\right]} = 0$. The second term of the sum is also $0$ by similar arguments. Hence, for all $s,t\in\left[n\right]$ s.t. $\left(s,t\right)\in \mathfrak{E}_G$,
\begin{equation*}
    \mathbf{J}_G^{\left[\lambda_{st},\sigma_{ij}\right]} = 0.
\end{equation*}
In addition,
\begin{equation*}
    \mathbf{J}_G^{\left[\omega,\sigma_{ij}\right]} = \vec\left(I_{n^2}\right)\left(\left(I_{n^2} - \Lambda\otimes\Lambda\right)^{-1}\right)^{\left[\sigma_{ij}\right]} 
    = \sum_{k=0}^{+\infty}\sum_{a=1}^n\lambda_{ai}^{\left[k\right]}\lambda_{aj}^{\left[k\right]} 
    = 0,
\end{equation*}
because if there exists $k\in\mathbb{N}$ and $a\in\left[n\right]$ s.t. $\lambda_{ai}^{\left[k\right]}\lambda_{aj}^{\left[k\right]}\neq 0$, then there exist directed paths: $a\rightsquigarrow i$ and $a\rightsquigarrow j$, which contradicts the assumption.

In conclusion,
\begin{equation*}
    \mathbf{J}_G^{\left[\sigma_{ij}\right]} = \mathbf{0}.
\end{equation*}

Next, we prove that if $i$ and $j$ belong to the same maximal class, i.e., there is a directed path between $i$ and $j$, or there exists $k\neq i,j\in\left[n\right]$ s.t. there are directed paths from $k$ to $i$ and $j$ respectively, then $\mathbf{J}_G^{\left[\sigma_{ij}\right]} \neq \mathbf{0}$.

In particular,
\begin{equation*}
    \mathbf{J}_G^{\left[\omega,\sigma_{ij}\right]} = \vec\left(I_{n^2}\right)\left(\left(I_{n^2} - \Lambda\otimes\Lambda\right)^{-1}\right)^{\left[\sigma_{ij}\right]} 
    = \sum_{k=0}^{+\infty}\sum_{a=1}^n\lambda_{ai}^{\left[k\right]}\lambda_{aj}^{\left[k\right]} 
    \neq 0,
\end{equation*}
because if $\mathbf{J}_G^{\left[\omega,\sigma_{ij}\right]} = 0$, then by Lemma \ref{lem:M_P_M_S}, for all $k\in\left[n\right]$ and $a\in\left[n\right]$, $\lambda_{ai}^{\left[k\right]}\lambda_{aj}^{\left[k\right]} = 0$. It means that there is no directed path from $a$ to $i$ and $j$ respectively (here $a$ may equal to $i$ or $j$), which contradicts the assumption. Therefore, $\mathbf{J}_G^{\left[\omega,\sigma_{ij}\right]} \neq 0$, and moreover,
\begin{equation*}
    \mathbf{J}_G^{\left[\sigma_{ij}\right]} \neq \mathbf{0}.
\end{equation*}
\end{proof}

Proposition \ref{prop:J_G_0} shows that two nodes that do not belong to the same maximal class will result in a null column in the Jacobian matrix. This is already an important piece of information on the Jacobian matroid. With one more proposition on the rank of the Jacobian matrix, we are able to prove the first two main results of this paper.

\begin{lemma}\label{lem:rank_J_G_psi}
Let $\mathbf{M}_G$ be a stationary VAR(1) model with an associated directed graph $G$. Then
\begin{equation*}
    rank\left(\mathbf{J}_G\right) = rank\left(\psi_G\left(B\right)\right).
\end{equation*}
\end{lemma}
\begin{proof}
The proof is given in Appendix \ref{app:prf_lem_rank_J_G_psi}.
\end{proof}

Lemma \ref{lem:rank_J_G_psi} ensures that when we need to study the rank of $\mathbf{J}_G$, we can look at the simpler matrix $\psi_G(B)$ instead. The following proposition provides an upper and lower bound for the rank of the Jacobian matrix based on this property.

\begin{proposition}\label{prop:rank_n_n_r_n_c}
Let $\mathbf{M}_G$ be a stationary VAR(1) model. Then 
\begin{equation*}
    n \leq rank\left(\mathbf{J}_G\right) \leq \min\left\{n_r,n_c'\right\},
\end{equation*}
where
\begin{align*}
    n_r &= E_G + 1; \\
    n_c' &= |\left\{\{a,b\} \mid a, b\in\left[n\right], a,b \mbox{ belong to the same maximal class}\right\}|.
\end{align*}
\end{proposition}
\begin{proof}
Consider a $n\times n$ submatrix of $\psi_G(B)$, with rows and columns corresponding to $\lambda_{ii}$ and $\sigma_{jj}$ respectively, for all $i,j\in[n]$. Then it is a diagonal matrix with non-zero entries on the diagonal. Therefore, the rank of this submatrix is $n$, and the rank of $\psi_G(B)$ is at least $n$ by the Guttman rank additivity formula (\cite{guttman1944general}). The result follows from the fact that $\mbox{rank}\left(\mathbf{J}_G\right) = \mbox{rank}\left(\psi_G(B)\right)$.

By definition, $\mathbf{J}_G$ is of size $\left(E_G+1\right)\times\frac{n\left(n+1\right)}{2}$. So
\begin{equation*}
    rank\left(\mathbf{J}_G\right) \leq \min\left\{n_r,\frac{n\left(n+1\right)}{2}\right\}.
\end{equation*}
In fact, some of the columns of $\mathbf{J}_G$ are zero vectors. We only count the non-zero columns, which is $n_c'$ by Proposition \ref{prop:J_G_0}. Therefore,
\begin{equation*}
    rank\left(\mathbf{J}_G\right) \leq \min\left\{n_r,n_c'\right\}.
\end{equation*}
\end{proof}

Now, we present the first main result that characterizes models with the same dimension using maximal classes, based on the lemmas and propositions introduced above.

\begin{theorem}\label{thm:iden_maxc_dim}
Let $\left\{\mathbf{M}_k\right\}_{k=1}^K$ be a finite set of stationary VAR(1) models corresponding to graphs $\left\{G_{k}\right\}_{k=1}^K$. If the models have the same dimension and for all distinct pairs $\left(k_1,k_2\right)$ of values from $1$ to $K$, $G_{k_1}$ and $G_{k_2}$ have different maximal classes, then these models (or the discrete parameter $k$) are generically identifiable.
\end{theorem}
\begin{proof}
Let $\mathbf{M}_1$ and $\mathbf{M}_2$ be two stationary VAR(1) models with the same dimension and different maximal classes. Then there exist two nodes $a,b\in[n]$ s.t. $a$ and $b$ belong to the same maximal class in $G_1$ and do not belong to the same maximal class in $G_2$, because otherwise by Theorem \ref{thm:supp_sig_maxc} and Proposition \ref{prop:set_maxc}, they have the same set of maximal classes. Therefore, by Proposition \ref{prop:J_G_0},
\begin{equation*}
    \mathbf{J}_{G_1}^{\left[\sigma_{ab}\right]} \neq \mathbf{0} \mbox{, and } \mathbf{J}_{G_2}^{\left[\sigma_{ab}\right]} = \mathbf{0}.
\end{equation*}
Let $A,B\in\mathbb{R}^n$. The notation $A\indep B$ indicates that $A$ and $B$ are linearly independent, and $A\dep B$ indicates that $A$ and $B$ are not linearly independent. Then, there exist $a',b'\in[n]$ where $\{a',b'\}\neq\{a,b\}$ s.t.
\begin{equation*}
    \mathbf{J}_{G_1}^{\left[\sigma_{ab}\right]} \indep \mathbf{J}_{G_1}^{\left[\sigma_{a'b'}\right]},
\end{equation*}
because otherwise for all $a',b'\in[n]$ where $\{a',b'\}\neq\{a,b\}$, we have
\begin{equation*}
    \mathbf{J}_{G_1}^{\left[\sigma_{ab}\right]} \dep \mathbf{J}_{G_1}^{\left[\sigma_{a'b'}\right]},
\end{equation*}
meaning that $\mbox{rank}\left(\mathbf{J}_{G_1}\right)\leq 1$, which contradicts the fact that $\mbox{rank}\left(\mathbf{J}_{G_1}\right) \geq n \geq 2$ (see Proposition \ref{prop:rank_n_n_r_n_c}). Additionally, since $\mathbf{J}_{G_2}^{\left[\sigma_{ab}\right]} = \mathbf{0}$, we have
\begin{equation*}
    \mathbf{J}_{G_2}^{\left[\sigma_{ab}\right]} \dep \mathbf{J}_{G_2}^{\left[\sigma_{a'b'}\right]}.
\end{equation*}
Combining the results above, we have found a pair of columns of the Jacobian matrix corresponding to $\sigma_{ab}$ and $\sigma_{a'b'}$ such that they are linearly independent in $\mathbf{M}_1$, and not in $\mathbf{M}_2$. Therefore, $\mathcal{J}\left(\mathbf{M}_1\right) \neq \mathcal{J}\left(\mathbf{M}_2\right)$, and hence by Proposition \ref{prop:gene_iden_J_dim}, the two models are generically identifiable.
\end{proof}

Note that Theorem \ref{thm:iden_maxc_dim} is restricted to models with the same dimension. The following theorem removes this restriction, at the cost of imposing stronger conditions on the maximal classes of the graphs. In particular, each of the two conditions in the theorem below implies the condition in Theorem \ref{thm:iden_maxc_dim}, but the equal-dimension assumption is no longer required. We illustrate this difference in Table  \ref{tab:sum_iden_res}, where graphs in rows 1 and 3 can be distinguished only by Theorem \ref{thm:iden_maxc_dim} and Theorem \ref{thm:iden_maxc}, respectively.

\begin{theorem}\label{thm:iden_maxc}
Let $\left\{\mathbf{M}_k\right\}_{k=1}^K$ be a finite set of stationary VAR(1) models corresponding to graphs $\left\{G_{k}\right\}_{k=1}^K$. If for all distinct pairs $\left(k_1,k_2\right)$ of values from $1$ to $K$, $G_{k_1}$ and $G_{k_2}$ satisfies the following two conditions:

\vspace{0.5em}
\noindent 1. There exist $i,j\in[n]$ s.t. $i,j$ belong to the same maximal class in $G_{k_1}$, but do not in $G_{k_2}$.

\vspace{0.5em}
\noindent 2. There exist $s,t\in[n]$ s.t. $s,t$ belong to the same maximal class in $G_{k_2}$, but do not in $G_{k_1}$.

\vspace{0.5em}
\noindent then the models (or the discrete parameter $k$) are generically identifiable.
\end{theorem}
\begin{proof}
Let $\mathbf{M}_1$ and $\mathbf{M}_2$ be two stationary VAR(1) models corresponding to graphs $G_1$ and $G_2$ that satisfy the two conditions in the theorem. Then, from the proof of Theorem \ref{thm:iden_maxc_dim}, we know that:
\begin{equation*}
    \exists S_1\in \mathcal{J}\left(\mathbf{M}_1\right) \backslash \mathcal{J}\left(\mathbf{M}_2\right) \mbox{ and } \exists S_2\in \mathcal{J}\left(\mathbf{M}_2\right) \backslash \mathcal{J}\left(\mathbf{M}_1\right).
\end{equation*}
Therefore, by Proposition \ref{prop:gene_iden_J_dim}, $\mathbf{M}_1$ and $\mathbf{M}_2$ are generically identifiable, no matter the dimensions.
\end{proof}

\begin{example}
Figure \ref{fig:MC_iden} presents two graphs that are generically identifiable using the results introduced in this section.
\begin{figure}[!tb]
\centering
\begin{subfigure}{.48\textwidth}
  \centering
  \includegraphics[width=0.45\linewidth]{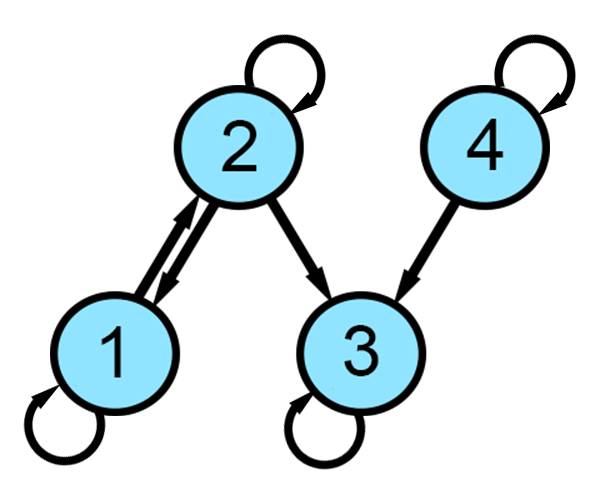}
  \caption{$G_1$}
\end{subfigure}%
\begin{subfigure}{.48\textwidth}
  \centering
  \includegraphics[width=0.45\linewidth]{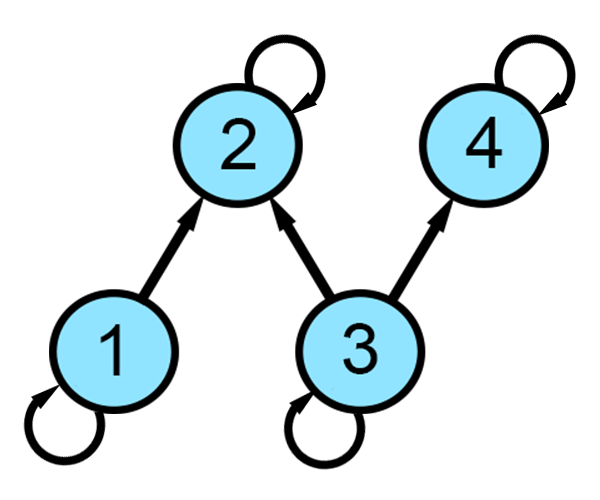}
  \caption{$G_2$}
\end{subfigure}
\caption{Identifiable graphs using Theorem \ref{thm:iden_maxc}}
\label{fig:MC_iden}
\end{figure}
The sets of maximal classes $\mathfrak{MC}_1$ of $G_1$ and $\mathfrak{MC}_2$ of $G_2$ are:
\begin{equation*}
    \mathfrak{MC}_1 = \left\{\{1,2,3\},\{3,4\}\right\}, \mbox{ and } \mathfrak{MC}_2 = \left\{\{1,2\}, \{2,3,4\}\right\}.
\end{equation*}
Because at this point, we don't know the dimension of the models, we can only apply Theorem \ref{thm:iden_maxc}. In fact, nodes $1$ and $3$ belong to the same maximal class in $G_1$, but not in $G_2$. Similarly, nodes $2$ and $4$ belong to the same maximal class in $G_2$, but not in $G_1$. By Theorem \ref{thm:iden_maxc}, the two graphs are generically identifiable.
\end{example}

While Theorems \ref{thm:iden_maxc_dim} and \ref{thm:iden_maxc} already provide criteria for generic identifiability that are easy to check, if the dimension of the model could be calculated from the graph, we would only need to focus on the identifiability of models with the same dimension, and always apply the weaker condition on maximal classes introduced in Theorem \ref{thm:iden_maxc_dim}. Indeed, the following section presents that if we consider only a subset of models whose associated graphs do not have reciprocal edges, meaning that $(i,j)\in\mathfrak{E}_G$ implies that $(j,i)\notin\mathfrak{E}_G$, we are able to derive the exact dimension of the model from the graph.

\subsection{Dimension of the model for reflexive graphs with no reciprocal edges}\label{sec:dim_mod}
This section demonstrates that for a class of models whose corresponding graphs are reflexive with no reciprocal edges, the dimension (i.e., the rank of the Jacobian matrix) can be determined from the graph. Using this result, we provide additional identifiability criteria based on the dimensions of the model, complementing those in Section \ref{sec:iden_maxc}. By integrating these results, we are able to expand the scope of models that are identifiable. On the other hand, for the cases where reciprocal edges often exist, such as in ecological research, the results in this section do not apply anymore.

The following assumption excludes reciprocal edges, which is needed in the study of the rank of the Jacobian matrix. In particular, it ensures the existence of "triplets" in Lemma \ref{lem:triplets}, which is crucial in proving that the Jacobian matrix is of full rank in Theorem \ref{thm:rank_no_multi_edges}.

\begin{assumption}\label{asmp:no_multi_edges}
Any graph $G$ is assumed to be reflexive and to have no reciprocal edges. That is, $\left(i,i\right)\in\mathfrak{E}_G$ for all $i\in [n]$, and $\forall i\neq j\in[n]$, $(i,j)\in \mathfrak{E}_G$ implies $(j,i)\notin \mathfrak{E}_G$.
\end{assumption}

\begin{theorem}\label{thm:rank_no_multi_edges}
Let $\mathbf{M}_G$ be a stationary VAR(1) model that satisfies Assumption \ref{asmp:no_multi_edges}, then
\begin{equation}\label{eq:rank_J_G}
    \mbox{rank}(\mathbf{J}_G) = \min\left\{n_r,n_c'\right\},
\end{equation}
where $n_r$ and $n_c'$ are defined in Proposition \ref{prop:rank_n_n_r_n_c}, is generically true, i.e., the set of values such that (\ref{eq:rank_J_G}) is not true has Lebesgue measure zero.
\end{theorem}

Note that by definition, $n_r$ is the number of rows, and $n_c'$ is the number of non-zero columns of $\mathbf{J}_G$. Thus Theorem \ref{thm:rank_no_multi_edges} in fact indicates that all of $\mathbf{J}_G$'s non-zero columns form a full rank matrix.

\begin{example}
Consider a stationary VAR(1) model with the corresponding graph below. The set of maximal classes $\mathfrak{MC}$, $n_r$ and $n_c'$ are:

\begin{minipage}{0.58\textwidth}
\begin{align*}
    \mathfrak{MC} = &\left\{\{1,2,3,4\}, \{3,5\}\right\}; \\
    n_r = &E_G +1 = 5+4+1 = 10; \\
    n_c' = &|\left\{\{1,1\}, \{2,2\}, \{3,3\}, \{4,4\}, \{5,5\},\right. \\
    &\left.\{1,2\}, \{1,3\}, \{1,4\}, \{2,3\}, \{2,4\}, \{3,4\}, \{3,5\}\right\}| = 12.
\end{align*}
\end{minipage}
\begin{minipage}{0.38\textwidth}
\centering
\includegraphics[width=0.6\textwidth]{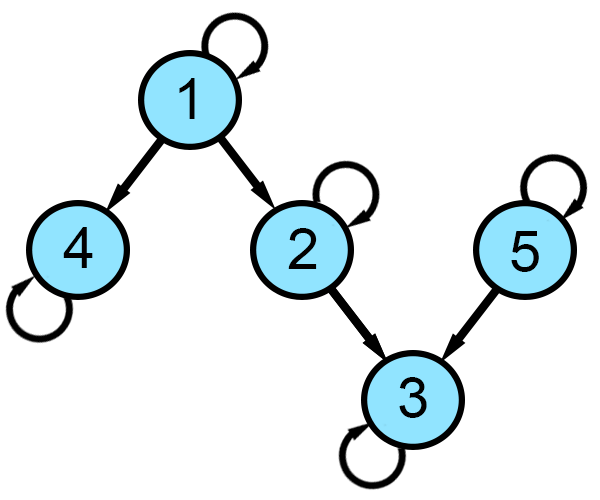}
\end{minipage}
Therefore,
\begin{equation*}
    \dim\left(\mathbf{M}_G\right) = \mbox{rank}\left(\mathbf{J}_G\right) = \min\left\{n_r,n_c'\right\} = 10.
\end{equation*}
\end{example}

The proof of Theorem \ref{thm:rank_no_multi_edges} is technical, and aligns with generic settings throughout this paper. It is divided into two separate cases: $n_r\leq n_c'$ and $n_r> n_c'$. We only introduce the proof for the first case, since the proof of the second case resembles the first one. For a complete proof of the second case, see Appendix \ref{app:prf_2_thm_rank_no_multi_edges}.

Now, assume $n_r\leq n_c'$, and the goal is to prove that under this condition, $rank\left(\mathbf{J}_G\right)= n_r$. For this, Lemmas \ref{lem:exist_kl}-\ref{lem:B_G_rank_nr} are needed. They transform the rank of the Jacobian matrix to that of another matrix that is derived from the graph, and prove that this matrix is full rank generically, based on a special graphical structure that exists in all graphs in this case.

Let $\mathfrak{E}_G'$ be the set of edges excluding self-loops, and $E_G' = |\mathfrak{E}_G'|$. Define the set of pairs of nodes such that they are not directly connected by one edge, but belong to the same maximal class:
\begin{equation*}
    \mathfrak{C}_G^{mc} := \left\{\{k,l\} \mid k,l\in[n], (k,l),(l,k)\notin \mathfrak{E}_G', \mbox{ and } k,l \mbox{ belong to the same maximal class}\right\},
\end{equation*}
and denote the cardinality of $\mathfrak{C}_G^{mc}$ as $C_G^{mc}$. Under Assumption \ref{asmp:no_multi_edges}, the following lemma introduces a necessary and sufficient condition for $n_r\leq n_c'$.

\begin{lemma}\label{lem:exist_kl}
Let $\mathbf{M}_G$ be a stationary VAR(1) model that satisfies Assumption \ref{asmp:no_multi_edges}, then $n_r\leq n_c'$ if and only if there exist two nodes $k,l\in\left[n\right]$, s.t. $\{k,l\}\in \mathfrak{C}_G^{mc}$.
\end{lemma}
\begin{proof}
Assume $\forall k,l \in [n]$, $\{k,l\}\notin \mathfrak{C}_G^{mc}$. Then either $(k,l)$ or $(l,k)\in\mathfrak{E}_G$ or $k,l$ do not belong to the same maximal class, i.e., $(\Sigma\Lambda)_{kl}=(\Sigma\Lambda)_{lk}=0$. In this case, $n_c' = E_G$ because for all pairs of nodes $i,j\in[n]$, $(\Sigma\Lambda)_{ij}\neq 0$ if and only if $\left(i,j\right)$ or $\left(j,i\right)\in \mathfrak{E}_G$ (but never at the same time because of Assumption \ref{asmp:no_multi_edges}). Since $n_r=E_G+1$, $n_r>n_c'$, which contradicts the premise.
\end{proof}

Lemma \ref{lem:exist_kl} indicates that under Assumption \ref{asmp:no_multi_edges}, the graphs that don't satisfy $n_r\leq n_c'$ are the ones such that there are at least two maximal classes, and the undirected subgraph of each maximal class is complete.

Next, we introduce a new matrix $B_G$ derived from the Jacobian matrix. And it will be shown in Lemma \ref{lem:rank_J_B_nr} that the Jacobian matrix is full rank if the matrix $B_G$ is full rank.

\begin{definition}\label{def:B_G_1}
Let $\mathbf{M}_G$ be a stationary VAR(1) model that satisfies the condition $n_r\leq n_c'$. Then consider a square matrix of size $E_G'+1$, where the rows correspond to $\lambda_{ij}$, s.t. $\left(i,j\right)\in \mathfrak{E}_G'$, and $\omega$, and the columns correspond to $\sigma_S$, where 
\begin{equation*}
    S = \left\{\{i,j\} \mid (i,j)\in \mathfrak{E}_G'\right\} \cup \{\{k,l\}\}, \mbox{ for any } \{k,l\} \in\mathfrak{C}_G^{mc}.
\end{equation*}
This matrix, denoted as $B_G$, is defined as
\begin{align*}
    B_G^{\left[\lambda_{ij},\sigma_{ab}\right]} &= \delta_{ja}\left[\left(\Sigma\Lambda\right)_{ib}\left(\Sigma\Lambda\right)_{jj} - \left(\Sigma\Lambda\right)_{ij}\left(\Sigma\Lambda\right)_{jb}\right] + \delta_{jb}\left[\left(\Sigma\Lambda\right)_{ia}\left(\Sigma\Lambda\right)_{jj} - \left(\Sigma\Lambda\right)_{ij}\left(\Sigma\Lambda\right)_{ja}\right]; \\
    B_G^{\left[\omega,\sigma_{ab}\right]} &= \left(\Sigma\Lambda\right)_{ab}\left(\Sigma\Lambda\right)_{aa}^{-1} + \left(\Sigma\Lambda\right)_{ba}\left(\Sigma\Lambda\right)_{bb}^{-1}, \\
\end{align*}
where $a,b\in[n]$ and $\{a,b\}\in S$.
\end{definition}

Note that given a directed graph $G$, $B_G$ might not be unique, because when $C_G^{mc}>1$, there are multiple options for $S$.

\begin{example}
Consider a stationary VAR(1) model $\mathbf{M}_G$ with a corresponding directed graph $G=\left(V,\mathfrak{E}_G\right)$, where $V=\left\{1,2,3\right\}$ and $\mathfrak{E}_G=\left\{\left(1,2\right),\left(2,3\right)\right\}$. Here there exists only one $B_G$, which is 
\begin{align*}
    B_G = &\left[\begin{array}{cc}
        \left(\Sigma\Lambda\right)_{11}\left(\Sigma\Lambda\right)_{22}-\left(\Sigma\Lambda\right)_{12}\left(\Sigma\Lambda\right)_{21} & \left(\Sigma\Lambda\right)_{13}\left(\Sigma\Lambda\right)_{22}-\left(\Sigma\Lambda\right)_{12}\left(\Sigma\Lambda\right)_{23} \\
        0 & \left(\Sigma\Lambda\right)_{22}\left(\Sigma\Lambda\right)_{33}-\left(\Sigma\Lambda\right)_{23}\left(\Sigma\Lambda\right)_{32} \\
        \left(\Sigma\Lambda\right)_{12}\left(\Sigma\Lambda\right)_{11}^{-1}+
        \left(\Sigma\Lambda\right)_{21}\left(\Sigma\Lambda\right)_{22}^{-1} & \left(\Sigma\Lambda\right)_{23}\left(\Sigma\Lambda\right)_{22}^{-1}+
        \left(\Sigma\Lambda\right)_{32}\left(\Sigma\Lambda\right)_{33}^{-1}
    \end{array}\right. \\
    &\left. \begin{array}{c}
        0 \\ \left(\Sigma\Lambda\right)_{21}\left(\Sigma\Lambda\right)_{33}-\left(\Sigma\Lambda\right)_{23}\left(\Sigma\Lambda\right)_{31} \\ \left(\Sigma\Lambda\right)_{13}\left(\Sigma\Lambda\right)_{11}^{-1}+
        \left(\Sigma\Lambda\right)_{31}\left(\Sigma\Lambda\right)_{33}^{-1}
    \end{array}\right].
\end{align*}
\end{example}

\begin{lemma}\label{lem:rank_J_B_nr}
Let $\mathbf{M}_G$ be a stationary VAR(1) model that satisfies Assumption \ref{asmp:no_multi_edges} and the condition $n_r\leq n_c'$. Then $\mbox{rank}\left(\mathbf{J}_G\right)=n_r$ if there exists a $B_G$ that is full rank, i.e., $\mbox{rank}\left(B_G\right)=E_G'+1$.
\end{lemma}
\begin{proof}[Sketch of proof]
By Lemma \ref{lem:rank_J_G_psi}, $\mbox{rank}\left(\mathbf{J}_G\right) = \mbox{rank}\left(\psi_G(B)\right)$, and therefore it is sufficient to prove that $\mbox{rank}\left(\psi_G(B)\right) = n_r$. Consider a $\left(E_G+1\right)\times\left(E_G+1\right)$ submatrix of $\psi_G\left(B\right)$, where the columns correspond to the set
\begin{equation*}
    \left\{\sigma_{ij} \mid i,j\in\left[n\right],(i,j) \mbox{ or } (j,i)\in\mathfrak{E}_G\right\} \cup \left\{\sigma_{kl}\right\},
\end{equation*}
where $\left\{k,l\right\}\in\mathfrak{C}_G^{mc}$. Note that $\mathfrak{C}_G^{mc}\neq \emptyset$ by Lemma \ref{lem:exist_kl}. In fact, this submatrix is a block matrix after reordering rows and columns. Apply the the Guttman rank additivity formula (see \cite{guttman1944general}), this matrix is full rank if and only if $B_G$ is full rank, up to a multiplication of a diagonal matrix with non-zero diagonals. For complete proof, see Appendix \ref{app:prf_lem_J_B_nr}.
\end{proof}

Next, we prove that under Assumption \ref{asmp:no_multi_edges} and the condition $n_r\leq n_c'$, $B_G$ is generically full rank, i.e., the Jacobian matrix $\mathbf{J}_G$ is full rank.

\begin{lemma}\label{lem:triplets}
Let $\mathbf{M}_G$ be a stationary VAR(1) model that satisfies the condition $n_r\leq n_c'$ and Assumption \ref{asmp:no_multi_edges}. Then there exists three nodes: $k,a,l\in\left[n\right]$ s.t. one of the following circumstances holds:

\vspace{0.5em}
\noindent 1. $\left(k,a\right),\left(a,l\right)\in\mathfrak{E}_G'$ and $\left(k,l\right),\left(l,k\right)\notin \mathfrak{E}_G'$.

\vspace{0.5em}
\noindent 2. $\left(a,k\right),\left(a,l\right)\in\mathfrak{E}_G'$ and $\left(k,l\right),\left(l,k\right)\notin \mathfrak{E}_G'$.
\end{lemma}
\begin{proof}
From Lemma \ref{lem:exist_kl}, we know that there exist two nodes $k,l\in[n]$ such that $\{k,l\}\in\mathfrak{C}_G^{mc}$, i.e., $(k,l),(l,k)\notin \mathfrak{E}_G'$, but $k$ and $l$ belong to the same maximal class. This ensures the existence of the triplets.
\end{proof}

The following lemma states that the set of values such that $B_G$ is not full rank has zero Lebesgue measure. We keep the proof in the main text because it again aligns with the generic settings of this paper, and uses the existence of the "triplet" structure introduced in Lemma \ref{lem:triplets}, which is technical and insightful.

\begin{lemma}\label{lem:B_G_rank_nr}
Let $\mathbf{M}_G$ be a stationary VAR(1) model that satisfies Assumption \ref{asmp:no_multi_edges} and $n_r\leq n_c'$. Define a subset $M_G^B$ of $M_G$:
\begin{equation*}
    M_G^B := \left\{\Lambda\in M_G \mid \det\left(B_G\right)=0\right\},
\end{equation*}
then
\begin{equation*}
    \mu_G\left(M_G^B\right) = 0,
\end{equation*}
where $\mu_G$ is the Lebesgue measure defined on $\mathbb{R}^{E_G}$, with $E_G$ denoting the number of edges of graph $G$.
\end{lemma}
\begin{proof}
First, we prove that the function
\begin{align*}
    f_G: M_G \times \mathbb{R}^+ &\rightarrow \mathbb{R} \\
    \left(\Lambda,\omega\right) &\mapsto \det\left(B_G\right)
\end{align*}
is rational. From Appendix \ref{app:para_rat}, we know that the parametrization map
\begin{align*}
	\phi_G : M_G \times \mathbb{R}^+ & \rightarrow M_n\left(\mathbb{R}\right) \\
	\left(\Lambda,\omega\right) & \mapsto \Sigma, \mbox{ s.t. } \vec\left(\Sigma\right) = \left(I_n-\Lambda^T\otimes\Lambda^T\right)^{-1}\vec\left(\omega I_n\right),
\end{align*}
is rational, i.e., $\Sigma$ is a matrix whose elements are rational functions with respect to the elements of $\Lambda$ and $\omega$. Therefore, $\left(\Sigma\Lambda\right)$ is also a matrix whose elements are rational functions. By definition, the elements of $B_G$ are the elements of $\left(\Sigma\Lambda\right)$ after summations, multiplications, and inversions. Hence, they are also rational. Since the determinant is also a result of summations and multiplications of the elements of the matrix, $f_G$ is rational, i.e., $\exists g,h\in \mathbb{R}\left[\left(\lambda_{ij} \mid \left(i,j\right)\in\mathfrak{E}_G\right),\omega\right]$ polynomials s.t.
\begin{equation*}
    f_G = \frac{g}{h}.
\end{equation*}

Next, we prove that the set of parameters such that $h=0$ has Lebesgue measure zero. Recall from (\ref{eq:vec_Sigma}), for all $i,j\in\left[n\right]$,
\begin{align*}
    \left(\Sigma\Lambda\right)_{ij} &= \sum_{k=1}^n\sigma_{ik}\lambda_{kj} = \sum_{k=1}^n\left(\vec\left(\Sigma\right)\right)_{\left(k-1\right)n+i}\lambda_{kj} \\
    &= \omega\sum_{k=1}^n\sum_{l=1}^n\left(I_{n^2}-\Lambda^T\otimes\Lambda^T\right)^{-1}_{\left(k-1\right)n+i,\left(l-1\right)n+l}\lambda_{kj}.
\end{align*}
Since
\begin{equation*}
    \left(I_{n^2}-\Lambda^T\otimes\Lambda^T\right)^{-1} = \frac{1}{\det\left(I_{n^2}-\Lambda^T\otimes\Lambda^T\right)}\adj\left(I_{n^2}-\Lambda^T\otimes\Lambda^T\right),
\end{equation*}
\begin{equation*}
    \left(\Sigma\Lambda\right)_{ij} = \frac{P}{Q},
\end{equation*}
where $\adj(\cdot)$ is the adjugate matrix, $P\in \mathbb{R}\left[\left(\lambda_{ij} \mid \left(i,j\right)\in\mathfrak{E}_G\right),\omega\right]$, and $Q=\det\left(I_{n^2}-\Lambda^T\otimes\Lambda^T\right)\in\mathbb{R}\left[\lambda_{ij}\mid \left(i,j\right)\in\mathfrak{E}_G\right]$. Therefore $\exists m\in\mathbb{N}$ s.t. 
\begin{equation*}
    h=Q^m\prod_{i=1}^n\left(\Sigma\Lambda\right)_{ii}.
\end{equation*}
From the fact that $\left(\Sigma\Lambda\right)_{ii}\neq 0$ for all $i\in\left[n\right]$,
\begin{equation*}
    h=0 \Leftrightarrow Q=0 \Leftrightarrow \det\left(I_{n^2} - \Lambda^T\otimes\Lambda^T\right) = 0 \Leftrightarrow \det\left(I_{n^2} - \Lambda\otimes\Lambda\right) = 0.
\end{equation*}
Therefore,
\begin{equation*}
    \mu_G\left(\left\{\Lambda\in M_G \mid h=0\right\}\right) = \mu_G\left(\left\{\Lambda\in M_G \mid\left(I_{n^2} - \Lambda\otimes\Lambda\right)\mbox{ is not invertible}\right\}\right) = 0.
\end{equation*}

Finally, we prove that the set of parameters such that $g=0$ has Lebesgue measure zero. Since $\det\left(B_G\right)$ is rational, it's sufficient to prove that for any directed graph $G=\left(V,\mathfrak{E}_G\right)$ that satisfies the premises, there exists $\Lambda_0\in M_G$, s.t. $\det\left(B_{G_0}\right)\neq 0$, where $B_{G_0}$ is $B_G$ with $\Lambda=\Lambda_0$.
From Lemma \ref{lem:triplets}, there exists three nodes $a,k,l\in\left[n\right]$, s.t. either $\left(k,a\right),\left(a,l\right)\in\mathfrak{E}_G'$ and $\left(k,l\right),\left(l,k\right)\notin\mathfrak{E}_G'$ (Case \ref{case:1_prf_B_G_rank_nr}), or $\left(a,k\right),\left(a,l\right)\in\mathfrak{E}_G'$ and $\left(k,l\right),\left(l,k\right)\notin\mathfrak{E}_G'$ (Case \ref{case:2_prf_B_G_rank_nr}). In both scenarios, $k,l$ belong to the same maximal class. Consider $B_G$ with rows correspond to $\lambda_{ij}$, where $\left(i,j\right)\in\mathfrak{E}_G'$ and $\omega$, and columns correspond to $\sigma_{ab}$, where $\left(a,b\right)\in\mathfrak{E}_G'$, and $\sigma_{kl}$. We discuss the two cases separately.

\setcounter{case}{0}
\begin{case}\label{case:1_prf_B_G_rank_nr} Consider $\Lambda_0=\left(\lambda_{ij}^0\right)$ that satisfies:
\begin{equation*}
    \lambda_{ij}^0 = \left\{\begin{array}{lcl}
        \lambda_{ij}^0\neq 0 &,& i=j \mbox{ or } \left(i,j\right) = \left(k,a\right) \mbox{ or } \left(i,j\right) = \left(a,l\right) \\
        0 &,& \mbox{otherwise}
    \end{array}\right.,
\end{equation*}
then
\begin{equation*}
    \left(\Sigma_0\Lambda_0\right)_{ij} \neq 0 \Leftrightarrow i=j \mbox{ or } \left\{i,j\right\} \in\left\{\left\{k,a\right\}, \left\{a,l\right\}, \left\{k,l\right\}\right\}.
\end{equation*}
Order the rows and columns of $B_G$ such that the rows correspond to $\lambda_{ka}$, $\lambda_{al}$, $\omega$ and $\lambda_{ij}$, and the columns correspond to $\sigma_{ka}$, $\sigma_{al}$, $\sigma_{kl}$ and $\sigma_{ij}$ respectively, where $\left(i,j\right)$ represents all the edges in $\mathfrak{E}_G'$ other than $\left(k,a\right)$ and $\left(a,l\right)$. Note that the columns are ordered along with the corresponding rows. Consider the matrix $B_{G_0}$, which is defined as $B_G$ with $\Lambda=\Lambda_0$, and define it as a block matrix:
\begin{equation*}
    B_{G_0} = \left[\begin{array}{cc}
        A & B \\
        C & D
    \end{array}\right],
\end{equation*}
where $A$ is a $3\times 3$ matrix consisting of the first three rows and columns of $B_{G_0}$. Then
\begin{align*}
    &A = \\
    &\left[\begin{array}{ccc}
        \left(\Sigma_0\Lambda_0\right)_{kk}\left(\Sigma_0\Lambda_0\right)_{aa} - \left(\Sigma_0\Lambda_0\right)_{ka}\left(\Sigma_0\Lambda_0\right)_{ak} & \left(\Sigma_0\Lambda_0\right)_{kl}\left(\Sigma_0\Lambda_0\right)_{aa} - \left(\Sigma_0\Lambda_0\right)_{ka}\left(\Sigma_0\Lambda_0\right)_{al} & 0 \\
        0 & \left(\Sigma_0\Lambda_0\right)_{aa}\left(\Sigma_0\Lambda_0\right)_{ll} - \left(\Sigma_0\Lambda_0\right)_{al}\left(\Sigma_0\Lambda_0\right)_{la} & \left(\Sigma_0\Lambda_0\right)_{ak}\left(\Sigma_0\Lambda_0\right)_{ll} - \left(\Sigma_0\Lambda_0\right)_{al}\left(\Sigma_0\Lambda_0\right)_{lk} \\
        \left(\Sigma_0\Lambda_0\right)_{ka}\left(\Sigma_0\Lambda_0\right)_{kk}^{-1} + \left(\Sigma_0\Lambda_0\right)_{ak}\left(\Sigma_0\Lambda_0\right)_{aa}^{-1} & \left(\Sigma_0\Lambda_0\right)_{al}\left(\Sigma_0\Lambda_0\right)_{aa}^{-1} + \left(\Sigma_0\Lambda_0\right)_{la}\left(\Sigma_0\Lambda_0\right)_{ll}^{-1} & \left(\Sigma_0\Lambda_0\right)_{kl}\left(\Sigma_0\Lambda_0\right)_{kk}^{-1} + \left(\Sigma_0\Lambda_0\right)_{lk}\left(\Sigma_0\Lambda_0\right)_{ll}^{-1}
    \end{array}\right].
\end{align*}
For all $\left(i,j\right)\in\mathfrak{E}_G'$, s.t. $\left\{i,j\right\}\notin\left\{\left\{a,k\right\},\left\{a,l\right\}, \left\{k,l\right\}\right\}$, it's clear that
\begin{equation*}
    B_{G_0}^{\left[\omega,\sigma_{ij}\right]} = \left(\Sigma_0\Lambda_0\right)_{ij}\left(\Sigma_0\Lambda_0\right)_{ii}^{-1} + \left(\Sigma_0\Lambda_0\right)_{ji}\left(\Sigma_0\Lambda_0\right)_{jj}^{-1} = 0.
\end{equation*}
And
\begin{align*}
    B_{G_0}^{\left[\lambda_{ka},\sigma_{ij}\right]} = &\delta_{ai}\left[\left(\Sigma_0\Lambda_0\right)_{kj}\left(\Sigma_0\Lambda_0\right)_{aa} - \left(\Sigma_0\Lambda_0\right)_{ka}\left(\Sigma_0\Lambda_0\right)_{aj}\right] \\
    &+ \delta_{aj}\left[\left(\Sigma_0\Lambda_0\right)_{ki}\left(\Sigma_0\Lambda_0\right)_{aa} - \left(\Sigma_0\Lambda_0\right)_{ka}\left(\Sigma_0\Lambda_0\right)_{ai}\right].
\end{align*}
When $a=i$, we know that $j\neq a,k,l$, so $\left(\Sigma_0\Lambda_0\right)_{kj}=\left(\Sigma_0\Lambda_0\right)_{aj}=0$, thus \\ $\delta_{ai}\left[\left(\Sigma_0\Lambda_0\right)_{kj}\left(\Sigma_0\Lambda_0\right)_{aa} - \left(\Sigma_0\Lambda_0\right)_{ka}\left(\Sigma_0\Lambda_0\right)_{aj}\right]=0$. Similarly, \\ $\delta_{aj}\left[\left(\Sigma_0\Lambda_0\right)_{ki}\left(\Sigma_0\Lambda_0\right)_{aa} - \left(\Sigma_0\Lambda_0\right)_{ka}\left(\Sigma_0\Lambda_0\right)_{ai}\right]=0$. Therefore, $B_{G_0}^{\left[\lambda_{ka},\sigma_{ij}\right]}=0$, i.e.
\begin{equation*}
    B = \mathbf{0}.
\end{equation*}
Use a similar argument, we have
\begin{equation*}
    C = \mathbf{0}.
\end{equation*}
Therefore, $B_{G_0}$ is in fact a block diagonal matrix:
\begin{equation*}
    B_{G_0} = \left[\begin{array}{cc}
        A & \mathbf{0} \\
        \mathbf{0} & D
    \end{array}\right].
\end{equation*}
Note that $\det\left(D\right)$ is again rational. Let $\Lambda_0'=2I_n$, then $D$ with $\Lambda_0=\Lambda_0'$, denoted as $D'$ is diagonal, and all the elements on the diagonal are non-zero. So $\det\left(D'\right)\neq 0$. Therefore, the set of values of $\Lambda_0$ such that $\det\left(D\right)=0$ has Lebesgue measure zero, i.e., $\det\left(D\right)\neq 0$, generically. On the other hand, let $\Lambda_0$ satisfies:
\begin{equation*}
    \lambda_{ij}^0 = \left\{\begin{array}{lcl}
        2 &,& i=j \mbox{ or } \left(i,j\right) = \left(k,a\right) \mbox{ or } \left(i,j\right) = \left(a,l\right) \\
        0 &,& \mbox{otherwise}
    \end{array}\right.,
\end{equation*}
then $\det\left(A\right) \neq 0$. This implies that $\det(A)\neq 0$ generically. Because the union of finitely many Lebesgue measure zero sets has Lebesgue measure zero, we conclude that $\det\left(B_{G_0}\right) = \det\left(A\right)\det\left(D\right) \neq 0$ generically.
\end{case}

\begin{case}\label{case:2_prf_B_G_rank_nr} Let $G_0=\left(V,\mathfrak{E}_{G_0}\right)$ where
\begin{equation*}
    \mathfrak{E}_{G_0} = \left\{\left(a,k\right),\left(a,l\right)\right\},
\end{equation*}
and
\begin{equation*}
    \lambda_{ij}^0 = \left\{\begin{array}{lcl}
        2 &,& i=j \mbox{ or } \left(i,j\right) = \left(a,k\right) \mbox{ or } \left(i,j\right) = \left(a,l\right) \\
        0 &,& \mbox{otherwise}
    \end{array}\right..
\end{equation*}
Then using similar arguments as in Case 1, we can prove that $\det\left(B_{G_0}\right)\neq 0$.
\end{case}

Combining all the arguments above, we can conclude that the set of values of $\Lambda$ such that $\det\left(B_G\right)=0$ has Lebesgue measure zero.
\end{proof}

\begin{proof}[Proof of case $1$ in Theorem \ref{thm:rank_no_multi_edges}]
This is a direct result of Lemmas \ref{lem:rank_J_B_nr} and \ref{lem:B_G_rank_nr}.
\end{proof}

In conclusion, for graphs without reciprocal edges, the rank of the Jacobian matrix (the dimension of the model) is calculable.

As explained in Section \ref{sec:gene_iden}, models with different dimensions are generically identifiable. With Theorem \ref{thm:rank_no_multi_edges}, the dimensions are known for a subset of models that satisfies Assumption \ref{asmp:no_multi_edges}. One interpretation of this result is that for a family of models $\left\{\mathbf{M}_k\right\}_{k=1}^K$ that satisfy Assumption \ref{asmp:no_multi_edges} and the condition $n_r<n_c'$, they are generically identifiable if the number of edges in the corresponding graphs is different, i.e., the number of edges is identifiable. See the next section for a summary and illustration of all the results above.

\subsection{Summary and illustration}
In this section, we first provide a summary of all the identifiability results (Table \ref{tab:sum_iden_res}), in which conditions for two stationary VAR(1) models $\mathbf{M}_1$ and $\mathbf{M}_2$ to satisfy the identifiability criteria or not and corresponding examples are listed. 
In the table, $\mathfrak{MC}_i$ is the set of maximal classes of the model $\mathbf{M}_i$. 
\begin{table}[!tb]
\caption{A summary of identifiability results with examples}
\label{tab:sum_iden_res}
\begin{center}
\begin{tabular}{ | p{0.15\textwidth}| p{0.2\textwidth}| p{0.1\textwidth}| p{0.17\textwidth} p{0.17\textwidth}|}
 \hline
 Conclusion & Conditions & References & Examples & \\ [0.5ex] 
 \hline\hline
 \multirow{3}{6em}{$\mathbf{M}_1$ and $\mathbf{M}_2$ are identifiable} &$G_1$ and $G_2$ do not contain reciprocal edges,  $\dim\left(\mathbf{M}_1\right) = \dim\left(\mathbf{M}_2\right)$, and $\mathfrak{MC}_1 \neq \mathfrak{MC}_2$ & Theorem \ref{thm:iden_maxc_dim} & $G_1$ \newline  \includegraphics[width=0.15\textwidth]{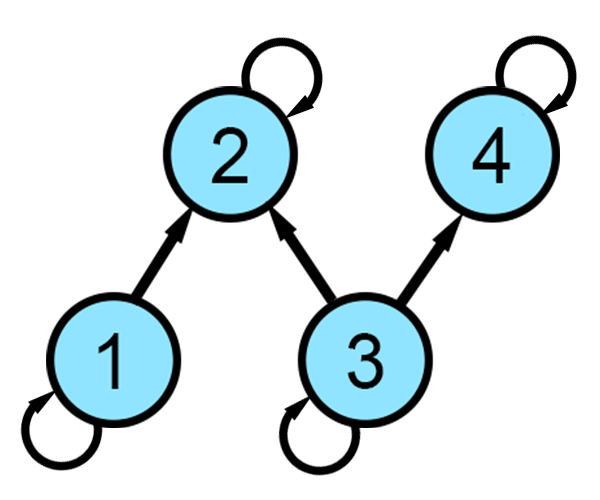}& $G_2$ \newline \includegraphics[width=0.15\textwidth]{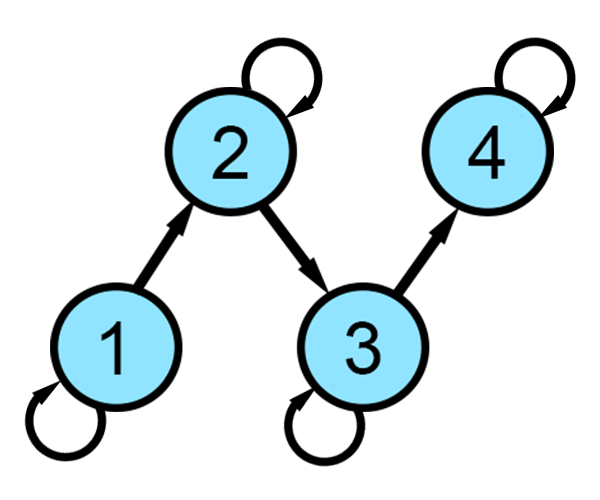} \\
 \cline{2-5}
 &$\dim\left(\mathbf{M}_1\right) \neq \dim\left(\mathbf{M}_2\right)$ & Theorem \ref{thm:rank_no_multi_edges} & $G_1$ \newline  \includegraphics[width=0.15\textwidth]{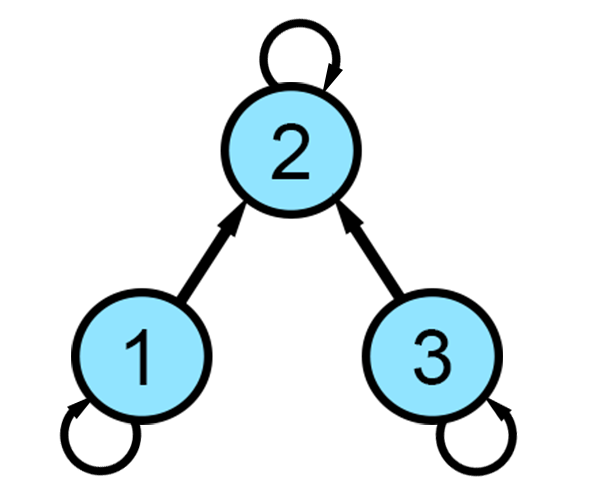}& $G_2$ \newline \includegraphics[width=0.15\textwidth]{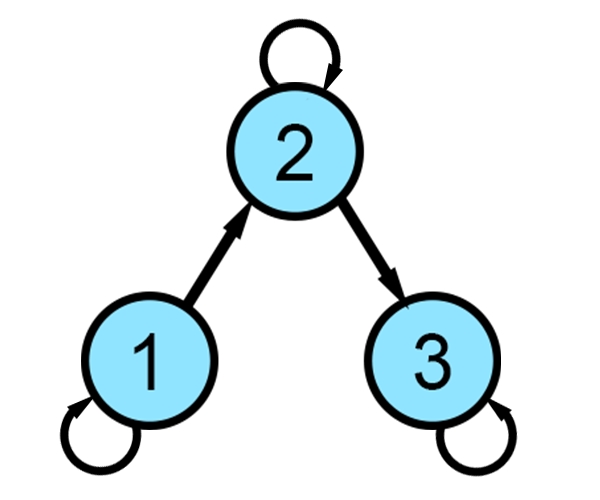} \\
 \cline{2-5}
 &$G_1$ or $G_2$ contains reciprocal edges, and conditions in Theorem \ref{thm:iden_maxc} are satisfied & Theorem \ref{thm:iden_maxc} &Figure \ref{fig:MC_iden} &\\
 \hline
 \multirow{3}{6em}{$\mathbf{M}_1$ and $\mathbf{M}_2$ do not satisfy the identifiability criteria in this paper} &$\dim\left(\mathbf{M}_1\right) = \dim\left(\mathbf{M}_2\right)$ and $\mathfrak{MC}_1 = \mathfrak{MC}_2$ & - &$G_1$ \newline  \includegraphics[width=0.15\textwidth]{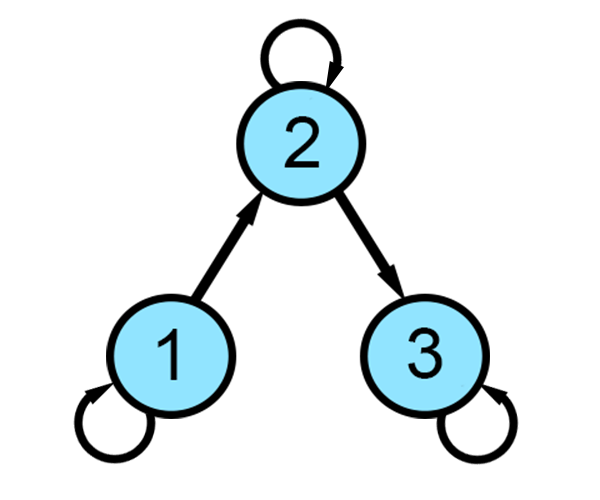}& $G_2$ \newline \includegraphics[width=0.15\textwidth]{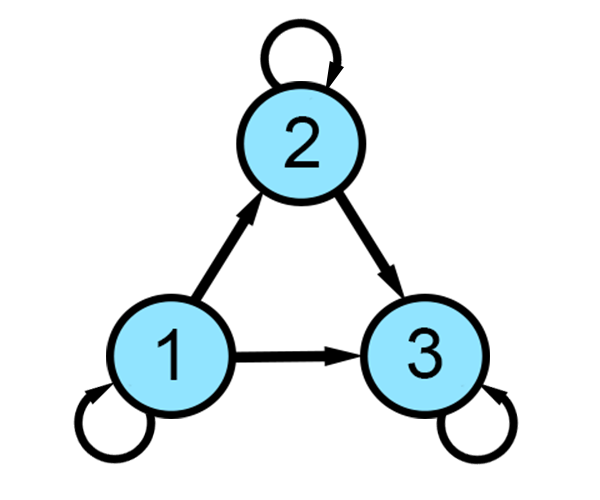}\\
 \cline{2-5}
 &$G_1$ or $G_2$ contains reciprocal edges and conditions in Theorem \ref{thm:iden_maxc} are not satisfied & - &$G_1$ \newline  \includegraphics[width=0.15\textwidth]{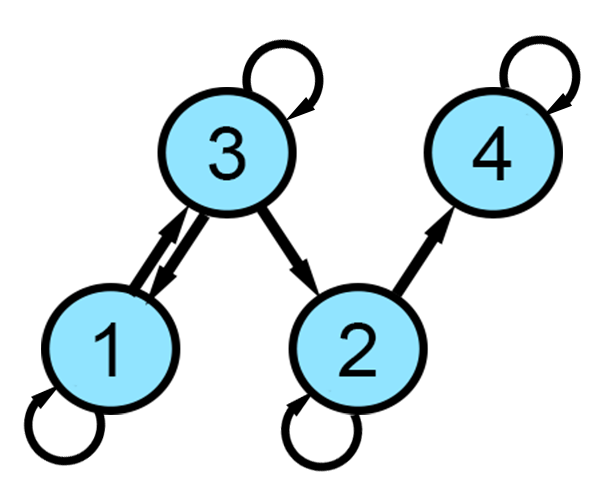}& $G_2$ is any graph\\
 \hline
\end{tabular}
\end{center}
\end{table}

Note that although we present cases where the identifiability criteria are not satisfied, this does not necessarily mean the models are unidentifiable. These graphs fall into two categories: (i) those containing reciprocal edges, the dimensions of which cannot be determined using the results in this paper (final row, Table \ref{tab:sum_iden_res}); and (ii) those that cannot be distinguished by model dimension or maximal classes (second-to-last row, Table \ref{tab:sum_iden_res}). The latter case could suggest a certain degree of homogeneity.

We end this section with an example of a family of graphs that are identifiable from each other.

\begin{example}
The family of graphs in Figure \ref{fig:fam_iden} are identifiable.
\begin{figure}[!tb]
\centering
\begin{subfigure}{.23\textwidth}
  \centering
  \includegraphics[width=0.8\linewidth]{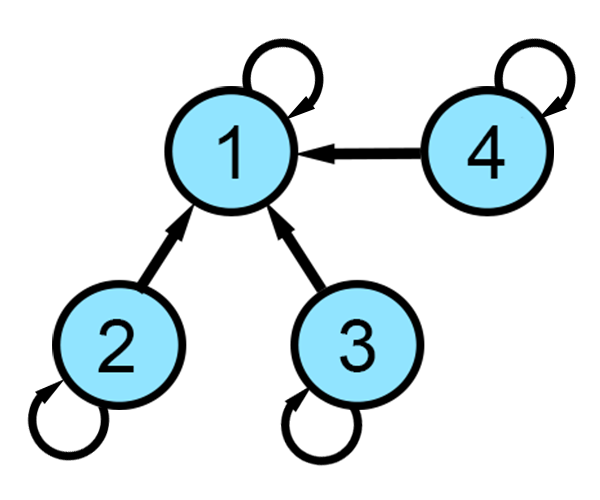}
  \caption{$G_1$}
\end{subfigure}%
\begin{subfigure}{.23\textwidth}
  \centering
  \includegraphics[width=0.8\linewidth]{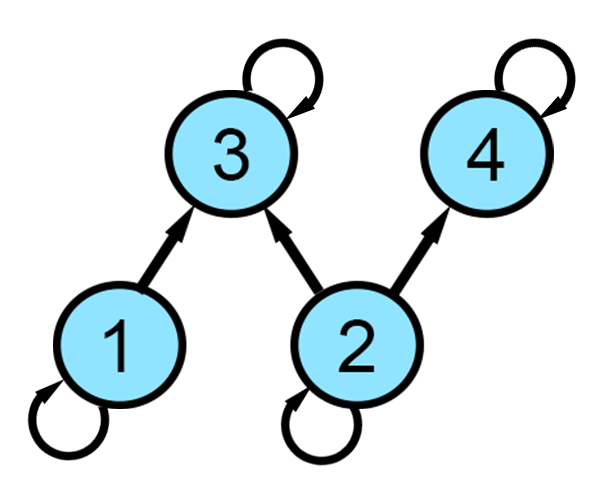}
  \caption{$G_2$}
\end{subfigure}
\begin{subfigure}{.23\textwidth}
  \centering
  \includegraphics[width=0.8\linewidth]{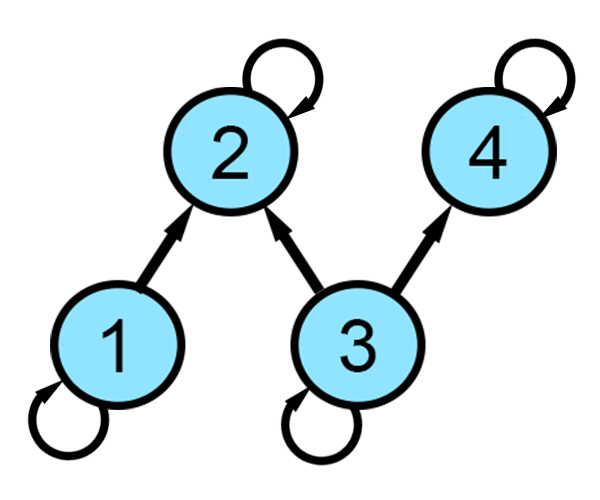}
  \caption{$G_3$}
\end{subfigure}
\begin{subfigure}{.23\textwidth}
  \centering
  \includegraphics[width=0.8\linewidth]{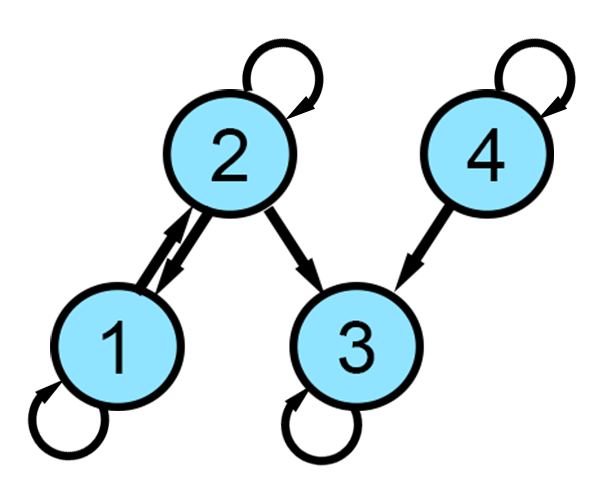}
  \caption{$G_4$}
\end{subfigure}
\caption{A family of identifiable graphs}
\label{fig:fam_iden}
\end{figure}
The relevant properties of the graphs are listed below:
\begin{center}
\begin{tabular}{| p{0.15\textwidth}| p{0.35\textwidth}| p{0.35\textwidth}|}
    \hline
    Graph & Set of maximal classes & Dimension \\ [0.5ex]
    \hline
    \hline
    $G_1$ & $\left\{\{1,2\}, \{1,3\}, \{1,4\}\right\}$ & $7$ \\
    \hline
    $G_2$ & $\left\{\{1,3\}, \{2,3,4\}\right\}$ & $8$ \\
    \hline
    $G_3$ & $\left\{\{1,2\}, \{2,3,4\}\right\}$ & $8$ \\
    \hline
    $G_4$ & $\left\{\{1,2,3\}, \{3,4\}\right\}$ & Not given by results of this paper\\
    \hline
\end{tabular}
\end{center}
By the dimensions, $G_1$ is identifiable from $G_2$ and $G_3$. Moreover, $G_2$ and $G_3$ are identifiable because the maximal classes are different. Finally, $G_4$ is identifiable from $G_1$, $G_2$ and $G_3$ because the sets of maximal class do not overlap.

The entry “Not given by results of this paper” means that the dimension of the model associated with $G_4$ cannot be determined using the results developed here, since $G_4$ contains reciprocal edges. This dimension can in principle be computed using computer algebra systems. Such computations are beyond the scope of the present work, which aims to provide theoretical criteria without relying on computational algebraic tools.
\end{example}

\section{Illustrations of ecological networks}\label{sec:illustrations}
This section presents possible applications of the identifiability results introduced in Section \ref{sec:iden_res}.

\subsection{Bipartite graphs with prior classification}\label{sec:illu_bip_prior}
Recall that bipartite graphs refer to a type of graph whose nodes can be divided into two groups, and all edges are directed from one group to the other.

\begin{proposition}
Any directed bipartite graphs are (generically) identifiable if the nodes are primarily classified, all edges are directed from one part to the other, and the direction is known.
\end{proposition}

\begin{example}
Figure \ref{fig:tro_net_c} is an example of a hierarchical ecological network.
\begin{figure}[!tb]
    \centering
    \includegraphics[width=0.4\textwidth]{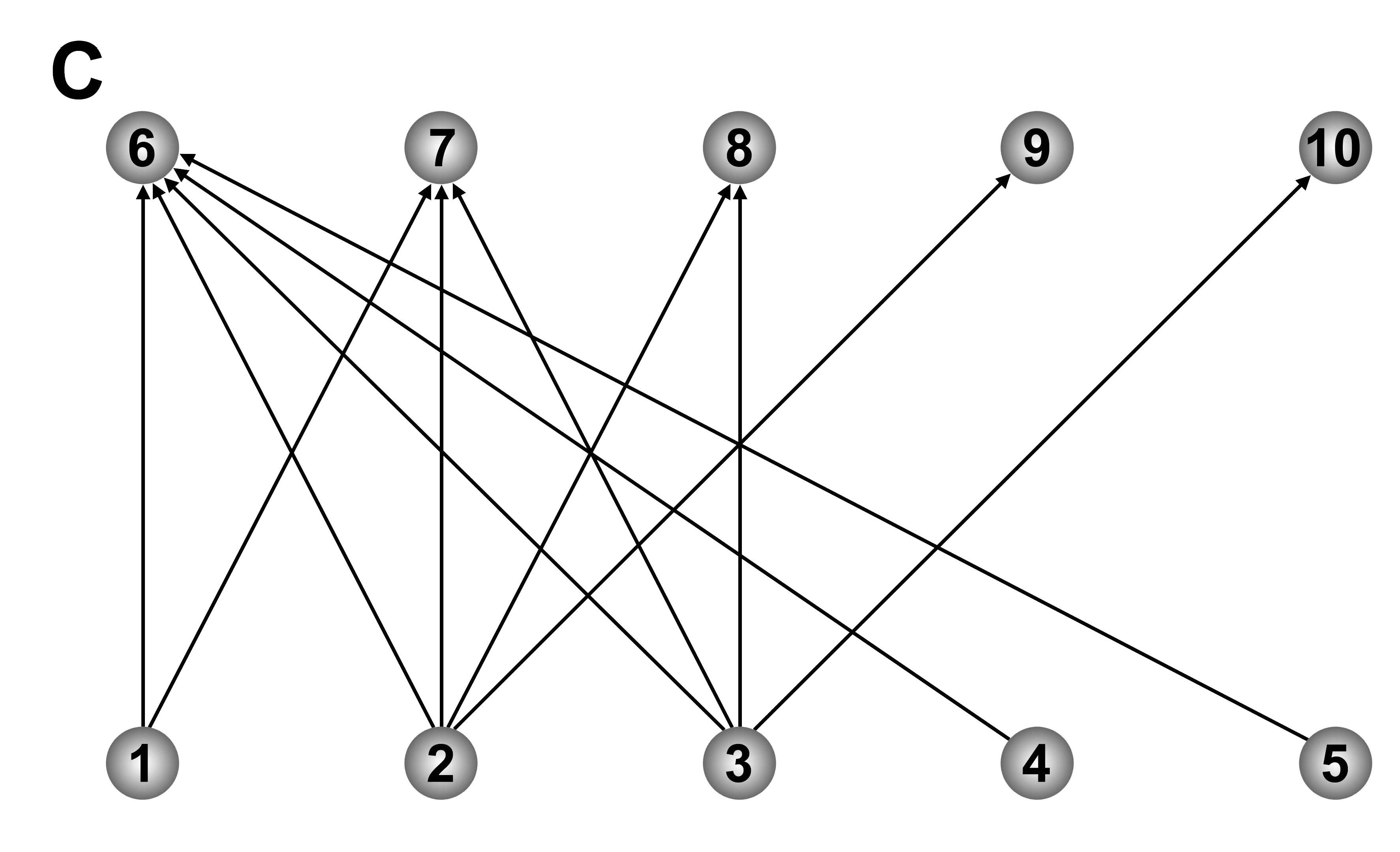}
    \caption{Example of hierarchical ecological networks (\cite{ings2009ecological})}
    \label{fig:tro_net_c}
\end{figure}
The maximal classes are:
\begin{equation*}
    \left\{\{1,6,7\}, \{2,6,7,8,9\},\{3,6,7,8,9,10\},\{4,6\},\{5,6\}\right\}.
\end{equation*}
If we already know that species $1-5$ are resources, and species $6-10$ are consumers, then each maximal class contains exactly one resource and all the consumers it feeds, thus the whole graph is identifiable.
\end{example}

\subsection{Bipartite graphs without prior classification}\label{sec:illu_bip_no_prior}
Generally, for a directed bipartite graph, whose edges are directed from one part to the other, and the direction is known, if we do not know the classification of the nodes, the whole graph is not identifiable. However, maximal classes can help to recover at least part of the graph. As shown in the following example.

\begin{example}
Consider again the network shown in Figure \ref{fig:tro_net_c}, if we do not know who are the resources, or who are the consumers, then we will not be able to recover the whole graph. But we know that each maximal class contains exactly one resource and all the consumers it feeds.
\end{example}

\subsection{Resilience of the network}\label{sec:illu_resi}
Maximal classes can sometimes indicate the resilience of the network.

\begin{proposition}
If two maximal classes are disjoint, the nodes from one maximal class are completely unrelated to the ones from the other.
\end{proposition}

\begin{example}
Figure \ref{fig:tro_net_a-b} shows two examples of ecological networks.
\begin{figure}[!tb]
    \centering
    \includegraphics[width=0.4\textwidth]{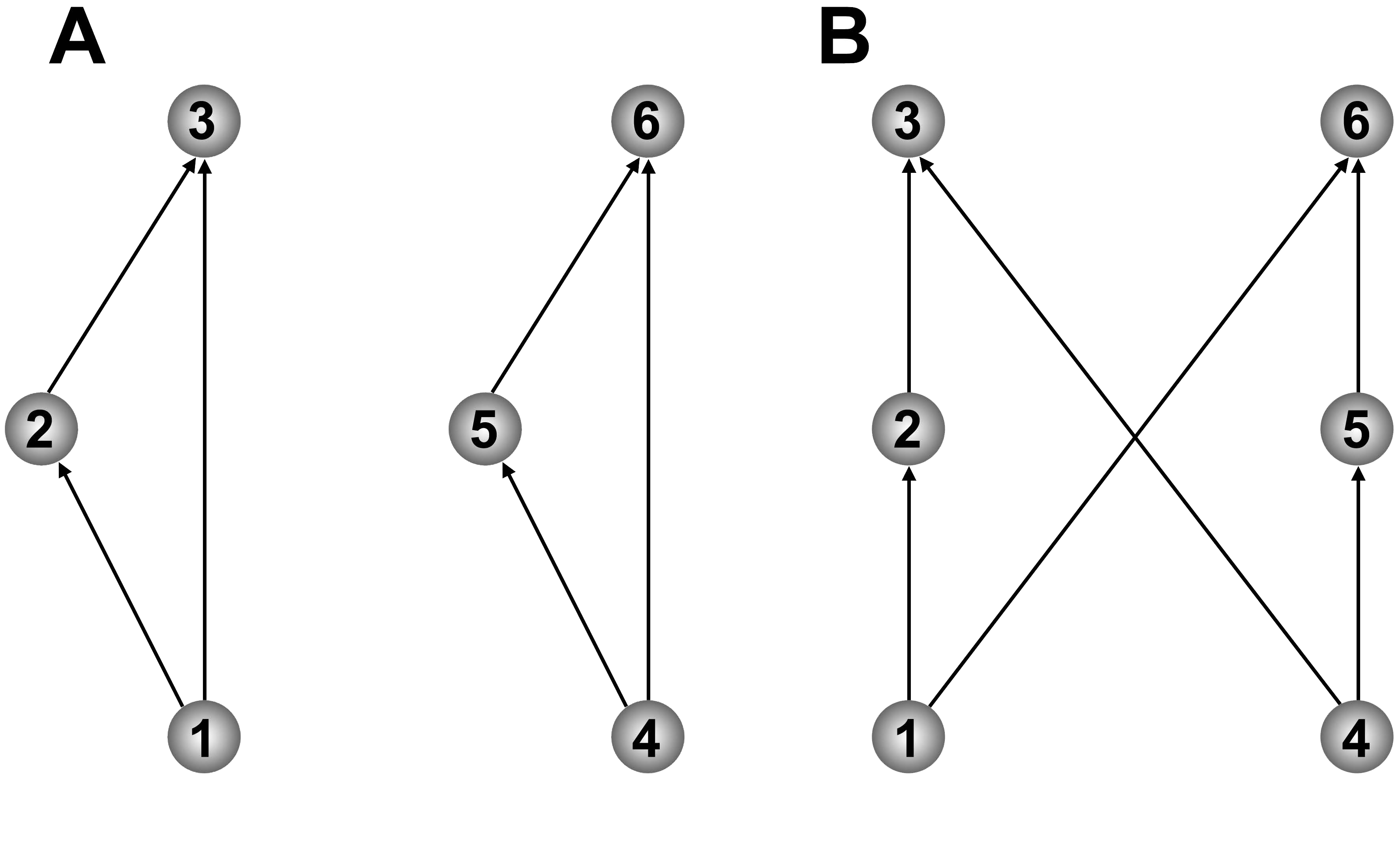}
    \caption{Example of hierarchical ecological networks (\cite{ings2009ecological})}
    \label{fig:tro_net_a-b}
\end{figure}
The Maximal classes of $A$ are
\begin{equation*}
    \left\{\{1,2,3\},\{4,5,6\}\right\},
\end{equation*}
and the maximal classes of $B$ are
\begin{equation*}
    \left\{\{1,2,3,6\},\{3,4,5,6\}\right\}.
\end{equation*}
In this case, $A$ is more resilient than $B$, because the only two maximal classes of $A$ are disjoint, and hence nodes $1-3$ do not affect nodes $4-6$ at all. In particular, if we take node $3$ away from the system, then in $A$, only nodes $1$ and $2$ will be affected, while in $B$, all the other nodes will be affected. 
\end{example}

\section*{Acknowledgments}
I would like to express my deepest gratitude to my PhD supervisors: Stéphane Robin, Viet Chi Tran and Elisa Thébault.  I am especially thankful to Stéphane Robin and Viet Chi Tran, for their unwavering support, guidance, and encouragement throughout the research journey, as well as for the passionate and intellectually stimulating spark they brought to our countless discussions, without which this research would not have been possible. I am also deeply grateful to Elisa Thébault for her valuable contribution in bringing an ecological perspective to this work, which enriched it tremendously. This research benefited from the support of the Chair Modélisation Mathématiques et Biodiversité VEOLIA – École polytechnique-MNHN – FX.

{\footnotesize 
\bibliographystyle{acm}
\bibliography{biblio.bib}
}

\newpage

\begin{appendix}
\section{Algebraic notions}
\label{app:alge_no}
This appendix introduces formally the Krull dimension with respect to Zariski topology in the framework of algebraic geometry (\cite{hartshorne2013algebraic}), and explains why the dimension can be defined as the rank of the Jacobian matrix at a generic point.

First, we introduce the notion of variety, on which the Krull dimension is defined.
\begin{definition}
Let $\mathbb{R}\left[X\right]$ be the set of polynomials of $n$ variables over $\mathbb{R}$, where $X\in\mathbb{R}^n$, and $S\subseteq \mathbb{R}\left[X\right]$ be a subset. The \textit{variety} defined by $S$ is the set:
\begin{equation*}
    \mathcal{V}(S) := \left\{X\in \mathbb{R}^n \mid f(X)=0,\forall f\in S\right\}.
\end{equation*}
\end{definition}
The variety $\mathcal{V}(S)$ is also called the zero set of $S$. A variety of algebraic geometry is not always a manifold in differential geometry because it could contain singularities (e.g. sharp points). Conversely, varieties without singular points fall into the domain of differential geometry as a smooth manifold.

Now, we define the Krull dimension of a variety.
\begin{definition}
Let $\mathcal{V}$ be a variety. Define the \textit{Krull dimension} of $\mathcal{V}$ (denoted as $\dim(\mathcal{V})$) to be the supremum of all integers $n$ such that there exists a chain $\mathcal{V}_0\subsetneq \mathcal{V}_1\subsetneq \cdots \subsetneq \mathcal{V}_n$ of distinct subvarieties of $\mathcal{V}$.
\end{definition}

An example of the Krull dimension is $\dim\left(\mathbb{R}^1\right) = 1$, (\cite[Chapter~1]{hartshorne2013algebraic}). Indeed, the only subvarieties of $\mathbb{R}^1$ are the whole space and single points.

Krull dimension is also well defined for non-variety sets, which is simply the dimension of its Zariski closure, i.e., the smallest variety containing the set. For example, $\mathbb{R^+}$ is not a variety. Its dimension is $1$ because its algebraic closure is $\mathbb{R}^1$. 

In our setting, $\mathbf{M}_G$ is originally a statistical model, not necessarily an algebraic variety. It is a subset of the cone of positive definite symmetric matrices $PD_n$, and hence a subset of the affine coordinate space $\mathbb{A}_{\mathbb{R}}^{n(n+1)/2}$, or simply $\mathbb{R}^{n(n+1)/2}$, with coordinates $\sigma_{ij}$, $1\leq i\leq j\leq n$.

To formally define the dimension of the model, first we define the vanishing ideal of $\mathbf{M}_G$ by
\begin{equation*}
    \mathcal{I}(\mathbf{M}_G) = \left\{ f\in \mathbb{R}[\sigma_{ij}:1\leq i\leq j\leq n] \mid f(\Sigma)=0 \text{ for all }\Sigma\in \mathbf{M}_G \right\}.
\end{equation*}
The Zariski closure of $\mathbf{M}_G$ is then
\begin{equation*}
    \overline{\mathbf{M}_G} = \mathcal{V}(\mathcal{I}(\mathbf{M}_G)) = \left\{\Sigma\in PD_n \mid f(\Sigma)=0 \text{ for all } f\in \mathcal{I}(\mathbf{M}_G) \right\}.
\end{equation*}
This follows from the standard characterization of the Zariski closure (\cite[Chapter~1, Section~4]{cox1997ideals}). Finally, as mentioned in Section \ref{sec:Problem_settings}, we define
\[
    \dim(\mathbf{M}_G) := \dim\left(\overline{\mathbf{M}_G}\right).
\]
This dimension is calculable by Gröbner bases (\cite[Section~3.4]{sullivant2018algebraic}), yet computationally demanding. Instead, we use an equivalent definition for the dimension - the rank of the Jacobian matrix of the parametrization map at a generic point. This is true because the parametrization map is rational (proof in Appendix \ref{app:para_rat})  (\cite[Chapter~16, Th.16.1.7]{sullivant2018algebraic}).

\section{Maximal class and related graphical notions}
\label{app:mc_graph_notion}
In fact, maximal classes are very similar to some common terminologies in graph theory. Recall that given a subset of nodes, an \textit{induced graph} of the set of nodes is the subgraph that contains all edges in the original graph among the nodes within the set, a \textit{directed tree} is a directed graph whose underlying undirected graph is a tree, a \textit{rooted tree} is a directed tree with a designated node called the \textit{root} and each edge is considered to be directed away from the root, a \textit{spanning tree} of a graph is a spanning subgraph that is a tree, and a \textit{rooted spanning tree} of a graph is a spanning tree that is rooted. This similarity is highlighted by the following proposition.

\begin{proposition}
The induced subgraph of a maximal class of a directed graph $G$ has a rooted spanning tree, where the root is (one of) the source node(s) of the maximal class.
\end{proposition}
\begin{proof}
For an induced subgraph of any maximal class of a directed graph $G$, keeping only the edges that direct away from (one of) the source node(s) of the maximal class would result in a rooted spanning tree.
\end{proof}

Note that the induced subgraphs of maximal classes do not construct a spanning forest because there might exist overlaps between different trees (Example \ref{eg:mc}).

\section{Proof of the fact that the parametrization map is rational}
\label{app:para_rat}
\begin{proof}
Recall that the parametrization map is:
\begin{align*}
	\phi_G : M_G \times \mathbb{R}^+ & \rightarrow M_n\left(\mathbb{R}\right) \\
	\left(\Lambda,\omega\right) & \mapsto \Sigma, \mbox{ s.t. } \vec\left(\Sigma\right) = \left(I_n-\Lambda^T\otimes\Lambda^T\right)^{-1}\vec\left(\omega I_n\right).
\end{align*}
It's sufficient to prove that for all $i\in\left[n^2\right]$, the $i^{th}$ coordinate function of $\phi_G$:
\begin{align*}
    \phi_i: M_G^B \times \mathbb{R}^+ & \rightarrow M_n\left(\mathbb{R}\right) \\
    \left(\Lambda,\omega\right) & \mapsto \vec\left(\Sigma\right)_i
\end{align*}
is rational, i.e., $\exists f_i,g_i\in\mathbb{R}\left[\lambda_{11}, \lambda_{12}, \cdots, \lambda_{nn}, \omega\right], s.t. \phi_i=f_i/g_i$, where $\vec\left(\Sigma\right)_i$ is the $i^{th}$ element of $\vec\left(\Sigma\right)$.

For all $i\in\left[n^2\right]$,
\begin{align*}
    \phi_i\left(\Lambda,\omega\right) &= \vec\left(\Sigma\right)_i \\
    &= \sum_{j=1}^{n^2}\left(I_{n^2}-\Lambda^T\otimes\Lambda^T\right)^{-1}_{ij}\vec\left(\omega I_n\right)_j \\
    &= \omega\sum_{j=1}^{n^2}\left(I_{n^2}-\Lambda^T\otimes\Lambda^T\right)^{-1}_{ij}\vec\left(I_n\right)_j \\
    &= \omega\sum_{k=1}^n\left(I_{n^2}-\Lambda^T\otimes\Lambda^T\right)^{-1}_{i,\left(\left(k-1\right)n+k\right)}.
\end{align*}
Since the finite sum of rational functions is also rational, it's sufficient to prove that \\$\left(I_{n^2}-\Lambda^T\otimes\Lambda^T\right)^{-1}$ is a matrix whose elements are rational functions.

By definition,
\begin{equation*}
    \left(I_{n^2}-\Lambda^T\otimes\Lambda^T\right)^{-1} = \frac{1}{\det\left(I_{n^2}-\Lambda^T\otimes\Lambda^T\right)}\adj\left(I_{n^2}-\Lambda^T\otimes\Lambda^T\right),
\end{equation*}
where $\adj\left(\cdot\right)$ is the adjugate matrix. Since $\det\left(\cdot\right)$ is a polynomial, and $\adj\left(\cdot\right)$ is a matrix whose elements are polynomials, $\left(I_{n^2}-\Lambda^T\otimes\Lambda^T\right)^{-1}$ is a matrix whose elements are rational functions.
\end{proof}

\section{Proof of Lemma \ref{lem:J_G_ext_complete}}
\label{app:prf_lem_J_G_ext_complete}
\begin{proof}
Recall from (\ref{eq:var1_sig_lam}):
\begin{equation*}
	\Sigma = \Lambda^T\Sigma\Lambda + \omega I_n,
\end{equation*}
and for all $i,j\in[n]$,
\begin{equation*}
    \left(\Lambda^T\Sigma\Lambda\right)_{ij} = \sum_{k,l=1}^n\lambda_{li}\sigma_{lk}\lambda_{kj} = \sum_{k,l=1}^n\sigma_{kl}\lambda_{ki}\lambda_{lj}.
\end{equation*}
Let
\begin{equation*}
    \delta_{ij} = \left\{\begin{array}{lcr}
        1 & \mbox{if} & i = j \\
        0 & \mbox{if} & i\neq j
    \end{array}\right.,\qquad \forall i,j\in[n].
\end{equation*}
For all $a,b,i,j\in[n]$,
\begin{equation}\label{eq:deri_sig_lam_ome}
\begin{split}
    \frac{\partial\sigma_{ij}}{\partial\lambda_{ab}} &= \frac{\partial\left(\Lambda^T\Sigma\Lambda\right)_{ij}}{\partial\lambda_{ab}} + \frac{\partial\left(\omega I_n\right)_{ij}}{\partial \lambda_{ab}} = \frac{\partial\left(\Lambda^T\Sigma\Lambda\right)_{ij}}{\partial\lambda_{ab}} \\
    &= \sum_{k,l=1}^n\frac{\partial\sigma_{kl}}{\partial\lambda_{ab}}\lambda_{ki}\lambda_{lj} + \sum_{k,l=1}^n\sigma_{kl}\frac{\partial\left(\lambda_{ki}\lambda_{lj}\right)}{\partial\lambda_{ab}} \\
    &= \sum_{k,l=1}^n\frac{\partial\sigma_{kl}}{\partial\lambda_{ab}}\lambda_{ki}\lambda_{lj} + 2\delta_{ij}\delta_{bi}\sum_{l=1}^n\sigma_{al}\lambda_{lb} + \left(1-\delta_{ij}\right)\left(\delta_{bi}\sum_{l=1}^n\sigma_{al}\lambda_{lj} + \delta_{bj}\sum_{l=1}^n\sigma_{al}\lambda_{li}\right), \\
    \frac{\partial\sigma_{ij}}{\partial\omega} &= \frac{\partial\left(\Lambda^T\Sigma\Lambda\right)_{ij}}{\partial\omega} + \frac{\partial\left(\omega I_n\right)_{ij}}{\partial\omega}
    = \sum_{k,l=1}^n\frac{\partial\sigma_{kl}}{\partial\omega}\lambda_{ki}\lambda_{lj} + \delta_{ij}.
\end{split}
\end{equation}
Let $\overline{\mathbf{J}}_{\lambda_{ab},\sigma_{ij}} = \partial\sigma_{ij} / \partial\lambda_{ab}$ and $\overline{\mathbf{J}}_{\omega,\sigma_{ij}} = \partial\sigma_{ij} / \partial\omega$. Then, using (\ref{eq:deri_sig_lam_ome}),
\begin{align*}
    \overline{\mathbf{J}}_{\lambda_{ab},\sigma_{ij}} &= \sum_{k,l=1}^n\overline{\mathbf{J}}_{\lambda_{ab},\sigma_{kl}}A_{\lambda_{kl},\sigma_{ij}} + B_{\lambda_{ab},\sigma_{ij}}, \\
    \overline{\mathbf{J}}_{\omega,\sigma_{ij}} &= \sum_{k,l=1}^n\overline{\mathbf{J}}_{\omega,\sigma_{kl}}A_{\lambda_{kl},\sigma_{ij}} + B_{\omega,\sigma_{ij}},
\end{align*}
where
\begin{align*}
    A_{\lambda_{ab},\sigma_{ij}} &= \lambda_{ai}\lambda_{bj}, \\
    B_{\lambda_{ab},\sigma_{ij}} &= 2\delta_{ij}\delta_{bi}\sum_{l=1}^n\sigma_{al}\lambda_{lb} + \left(1-\delta_{ij}\right)\left(\delta_{bi}\sum_{l=1}^n\sigma_{al}\lambda_{lj} + \delta_{bj}\sum_{l=1}^n\sigma_{al}\lambda_{li}\right) \\
    &= 2\delta_{ij}\delta_{bi}\left(\Sigma\Lambda\right)_{ab} + \left(1-\delta_{ij}\right)\left(\delta_{bi}\left(\Sigma\Lambda\right)_{aj} + \delta_{bj}\left(\Sigma\Lambda\right)_{ai}\right), \\
    B_{\omega,\sigma_{ij}} &= \delta_{ij},
\end{align*}
If we order the rows and columns of $\overline{\mathbf{J}}$ such that the rows correspond to $\lambda_{11}$,$\lambda_{12}$,...,$\lambda_{1n}$,..., \\ $\lambda_{n1}$,$\lambda_{n2}$,...,$\lambda_{nn}$, and $\omega$ respectively, and the columns correspond to $\sigma_{11}$,$\sigma_{12}$,...,$\sigma_{1n}$,...,\\ $\sigma_{n1}$,$\sigma_{n2}$,...,$\sigma_{nn}$, then we are able to organize $A$ and $B$ into the following matrix forms:
\begin{align*}
    A &= \Lambda\otimes\Lambda, \\
    B &= \\ 
    &{\scriptsize\setlength{\arraycolsep}{2.8pt}\left[\begin{array}{ccccccccccccccc}
        2\left(\Sigma\Lambda\right)_{11} & \left(\Sigma\Lambda\right)_{12} & \cdots & \left(\Sigma\Lambda\right)_{1n} & \left(\Sigma\Lambda\right)_{12} & 0 & \cdots & 0 & & & & \left(\Sigma\Lambda\right)_{1n} & 0 & \cdots & 0 \\
        0 & \left(\Sigma\Lambda\right)_{11} & \cdots & 0 & \left(\Sigma\Lambda\right)_{11} & 2\left(\Sigma\Lambda\right)_{12} & \cdots & \left(\Sigma\Lambda\right)_{1n} & & & & 0 & \left(\Sigma\Lambda\right)_{1n} & \cdots & 0 \\
        \vdots & \vdots & \ddots & \vdots & \vdots & \vdots & \ddots & \vdots & &\cdots& & \vdots & \vdots & \ddots & \vdots \\
        0 & 0 & \cdots & \left(\Sigma\Lambda\right)_{11} & 0 & 0 & \cdots & \left(\Sigma\Lambda\right)_{12} & & & & \left(\Sigma\Lambda\right)_{11} & \left(\Sigma\Lambda\right)_{12} & \cdots & 2\left(\Sigma\Lambda\right)_{1n} \\
        &&&&&&&&&&&&&& \\
        2\left(\Sigma\Lambda\right)_{21} & \left(\Sigma\Lambda\right)_{22} & \cdots & \left(\Sigma\Lambda\right)_{2n} & \left(\Sigma\Lambda\right)_{22} & 0 & \cdots & 0 & & & & \left(\Sigma\Lambda\right)_{2n} & 0 & \cdots & 0 \\
        0 & \left(\Sigma\Lambda\right)_{21} & \cdots & 0 & \left(\Sigma\Lambda\right)_{21} & 2\left(\Sigma\Lambda\right)_{22} & \cdots & \left(\Sigma\Lambda\right)_{2n} & & & & 0 & \left(\Sigma\Lambda\right)_{2n} & \cdots & 0 \\
        \vdots & \vdots & \ddots & \vdots & \vdots & \vdots & \ddots & \vdots & &\cdots& & \vdots & \vdots & \ddots & \vdots \\
        0 & 0 & \cdots & \left(\Sigma\Lambda\right)_{21} & 0 & 0 & \cdots & \left(\Sigma\Lambda\right)_{22} & & & & \left(\Sigma\Lambda\right)_{21} & \left(\Sigma\Lambda\right)_{22} & \cdots & 2\left(\Sigma\Lambda\right)_{2n} \\
        &&&&&&&&&&&&&& \\
        &&&&&&&&&&&&&& \\
        &&\vdots&&&&\vdots&&&&&&&\vdots& \\
        &&&&&&&&&&&&&& \\
        &&&&&&&&&&&&&& \\
        2\left(\Sigma\Lambda\right)_{n1} & \left(\Sigma\Lambda\right)_{n2} & \cdots & \left(\Sigma\Lambda\right)_{nn} & \left(\Sigma\Lambda\right)_{n2} & 0 & \cdots & 0 & & & & \left(\Sigma\Lambda\right)_{nn} & 0 & \cdots & 0 \\
        0 & \left(\Sigma\Lambda\right)_{n1} & \cdots & 0 & \left(\Sigma\Lambda\right)_{n1} & 2\left(\Sigma\Lambda\right)_{n2} & \cdots & \left(\Sigma\Lambda\right)_{nn} & & & & 0 & \left(\Sigma\Lambda\right)_{nn} & \cdots & 0 \\
        \vdots & \vdots & \ddots & \vdots & \vdots & \vdots & \ddots & \vdots & &\cdots& & \vdots & \vdots & \ddots & \vdots \\
        0 & 0 & \cdots & \left(\Sigma\Lambda\right)_{n1} & 0 & 0 & \cdots & \left(\Sigma\Lambda\right)_{n2} & & & & \left(\Sigma\Lambda\right)_{n1} & \left(\Sigma\Lambda\right)_{n2} & \cdots & 2\left(\Sigma\Lambda\right)_{nn} \\
        &&&&&&&&&&&&&& \\
        1 & 0 & \cdots & 0 & 0 & 1 & \cdots & 0 & & \cdots & & 0 & 0 & \cdots & 1
    \end{array}\right]} \\
    &= \left[\begin{array}{c}
        \left(\Sigma\Lambda\right)\otimes I_n + \left(\left(\Sigma\Lambda\right)\otimes I_n\right)\mathbf{P} \\ \vec\left(I_{n^2}\right)^T
        \end{array}\right] \\
    &= \left[\begin{array}{c}
        \left(\Sigma\Lambda\otimes I_n\right)\left(I_{n^2} + \mathbf{P}\right) \\ \vec\left(I_{n^2}\right)^T
        \end{array}\right],
\end{align*}
and
\begin{equation*}
    \overline{\mathbf{J}} = \overline{\mathbf{J}}A + B.
\end{equation*}
\end{proof}

\section{Proof of $\psi_G\left(\overline{\mathbf{J}}\right) = \overline{\mathbf{J}}_G$}
\label{app:prf_eq_psi}
\begin{proof}
Recall that $\overline{\mathbf{J}}_G$ has $E_G+1$ rows, which correspond to $\lambda_{ij}$, where $i,j\in\left[n\right]$ and $\left(i,j\right)\in \mathfrak{E}_G$ and $\omega$. By definition, $\overline{\mathbf{J}}$ has extra rows comparing to $\overline{\mathbf{J}}_G$, and in particular, rows that correspond to $\lambda_{st}$, where $s,t\in\left[n\right]$ and $\left(s,t\right)\notin \mathfrak{E}_G$. Therefore the first step of the projection $\psi_G$ is to remove these extra rows.

It's clear that for any $i,j,a,b,s,t\in\left[n\right]$, where $\left(i,j\right)\in \mathfrak{E}_G$ and $\left(s,t\right)\notin \mathfrak{E}_G$,
\begin{equation*}
    \left(\frac{\partial\sigma_{ab}}{\partial\lambda_{ij}}\right)|_{\lambda_{st}=0} = \frac{\partial\left(\sigma_{ab}|_{\lambda_{st}=0}\right)}{\partial\lambda_{ij}}, \mbox{ and } \left(\frac{\partial\sigma_{ab}}{\partial\omega}\right)|_{\lambda_{st}=0} = \frac{\partial\left(\sigma_{ab}|_{\lambda_{st}=0}\right)}{\partial\omega}.
\end{equation*}
These equations show that elements in $\overline{\mathbf{J}}_G$ are exactly the corresponding elements in $\overline{\mathbf{J}}$ setting $\lambda_{st}$ to $0$ for all $s,t\in\left[n\right]$ s.t. $\left(s,t\right)\notin \mathfrak{E}_G$.
\end{proof}

\section{Invertibility of $I_{n^2}-\Lambda\otimes\Lambda$}
\label{app:prf_I_lam_invert}
In Definition \ref{def:gene_iden_min}, we used the Lebesgue measure defined on $\mathbb{R}^{\frac{n(1+n)}{2}}$. Similarly, define $\mu_G$ the Lebesgue measure on $\mathbb{R}^{E_G}$.

\begin{proposition}\label{prop:I_lam_mea_0}
Let $\mathbf{M}_G$ be a stationary VAR(1) model, $\mu_G$ be a measure defined on $M_G$, which is the Lebesgue measure on $\mathbb{R}^{E_G}$, and $M_G^0$ be the subset of $M_G$ defined as
\begin{equation*}
    M_G^0 := \left\{\Lambda \in M_G \mid \left(I_{n^2} - \Lambda\otimes\Lambda\right) \mbox{ is not invertible}\right\},
\end{equation*}
then
\begin{equation*}
    \mu_G\left(M_G^0\right) = 0,
\end{equation*}
\end{proposition}
\begin{proof}
It's clear that $\left(I_{n^2} - \Lambda\otimes\Lambda\right)$ is not invertible if and only if $\det\left(I_{n^2} - \Lambda\otimes\Lambda\right) = 0$. Since the determinant is a result of summations and multiplications of the elements of the matrix, it is a polynomial with respect to the entries of $\Lambda$, i.e., $\det\left(I_{n^2} - \Lambda\otimes\Lambda\right) \in \mathbb{R}\left[\lambda_{ij} \mid \left(i,j\right)\in \mathfrak{E}_G\right]$.

Let $\Lambda_0 = 2I_{n}$, then $\Lambda_0\in M_G$ for any graph $G$, and
\begin{equation*}
    \det\left(I_{n^2} - \Lambda_0\otimes\Lambda_0\right) = \det\left(-3I_{n^2}\right) = \left(-3\right)^{n^2}\neq 0.
\end{equation*}
Since $\det\left(I_{n^2} - \Lambda\otimes\Lambda\right)$ is a polynomial, by continuity, it does not identically equal to $0$. Therefore, the set of $\Lambda$ such that the determinant equals zero must have zero measure, i.e., $\mu_G\left(M_G^0\right) = 0$.
\end{proof}






\section{Proof of Lemma \ref{lem:M_P_M_S}}
\label{app:prf_lem_M_P_M_S}
\begin{proof}
First, we prove that $\mu_G\left(M_G\backslash M^P\right) = 0$, i.e., $\Lambda \in M^P$ generically.
For all $i,j\in V$, s.t. there exists a directed path in $G$ from $i$ to $j$, it's clear that for all $k\in\mathbb{N}$, $\lambda_{ij}^{[k]}$ is a polynomial of entries of $\Lambda$. Consider $\Lambda_0\in M_G$ s.t. $\lambda_{st}>0$ for all $(s,t)\in \mathfrak{E}_G$, we know that there exists $K\in\mathbb{N}$ s.t. $\forall k>K$, $\lambda_{0,ij}^{[k]}\neq 0$. Therefore, $\forall k>K$, $\lambda_{ij}^{[k]}\neq 0$ generically, i.e., $\mu_G\left(M_G\backslash M^P\right) = 0$.

Next, we prove that $\mu_G\left(M_G\backslash M^S\right) = 0$. Define:
\begin{align*}
    M^{S_1} &:= \left\{\Lambda\in M_n\left(\mathbb{R}\right) \mid \forall i,j\in V, \sum_{k=0}^{+\infty}\sum_{a=1}^n\lambda_{ai}^{\left[k\right]}\lambda_{aj}^{\left[k\right]} = 0 \Leftrightarrow \forall k\in\mathbb{N},a\in\left[n\right], \lambda_{ai}^{\left[k\right]}\lambda_{aj}^{\left[k\right]} = 0\right\}, \\
    M^{S_2} &:= \left\{\Lambda\in M_n\left(\mathbb{R}\right) \mid \forall i,j\in V, \sum_{s=1}^n\sum_{k=0}^{+\infty}\sum_{l=1}^n\lambda_{li}^{\left[k\right]}\lambda_{ls}^{\left[k\right]}\lambda_{sj} = 0 \right. \\ &\left.\Leftrightarrow \forall k\in\mathbb{N}, s,l\in\left[n\right], \lambda_{li}^{\left[k\right]}\lambda_{ls}^{\left[k\right]}\lambda_{sj} = 0\right\}.
\end{align*}
By definition, $M^S = M^{S_1} \cap M^{S_2}$. Here, we present only the proof of $\mu_G\left(M_G\backslash M^{S_1}\right) = 0$, the proof of $\mu_G\left(M_G\backslash M^{S_2}\right) = 0$ is omitted because it uses the exact same technique as the other one.

Denote
\begin{equation*}
    h_{ij} = \sum_{k=0}^{+\infty}\sum_{a=1}^n\lambda_{ai}^{\left[k\right]}\lambda_{aj}^{\left[k\right]}.
\end{equation*}
$h_{ij}$ is a function defined on $M_G$, i.e., a function of entries of $\lambda_{st}$ where $(s,t)\in\mathfrak{E}_G$. Assume that for some $i,j\in V$, there exists $\Lambda_0\in M_G$, s.t. there exist $k_0\in\mathbb{N}$ and $a_0\in[n]$ s.t. $\lambda_{0,a_0i}^{[k_0]}\lambda_{0,a_0j}^{[k_0]}> 0$, and $h_{ij}\left(\Lambda_0\right)=0$. Denote $c_i = \lambda_{0,a_0i}^{[k_0]}$, and $c_j = \lambda_{0,a_0j}^{[k_0]}$. Note that $c_i$ and $c_j$ are constants in $\mathbb{R}$. Moreover, define another function $h_{ij}|_{k_0,a_0}$ on $M_G$ as $h_{ij}$ with $\lambda_{a_0i}^{[k_0]} = c_i$ and $\lambda_{a_0j}^{[k_0]} = c_j$ fixed as constants. Then $h_{ij}|_{k_0,a_0}$ is an analytic function of the entries of $\Lambda$. Consider again $\Lambda'\in M_G$ s.t. $\lambda_{st}'>0$ for all $(s,t)\in \mathfrak{E}_G$, then by definition, $h_{ij}|_{k_0,a_0}>0$. Therefore, $h_{ij}|_{k_0,a_0}$ does not constantly equal zero on $M_G$. In this case,
\begin{equation*}
    \mu_G\left(\left\{\Lambda\in M_G \mid h_{ij}|_{k_0,a_0} = 0\right\}\right) = 0,
\end{equation*}
because the set of roots of a non-zero real analytic function has Lebesgue measure zero (\cite[Section~3.1]{federer2014geometric}). This implies that
\begin{equation*}
    \mu_G\left(\left\{\Lambda\in M_G \mid \lambda_{a_0i}^{[k_0]}\lambda_{a_0j}^{[k_0]}> 0, \mbox{ and } h_{ij} = 0\right\}\right) = 0.
\end{equation*}
Therefore,
\begin{align*}
    &\mu_G\left(M_G\backslash M^{S_1}\right) \\ 
    &= \mu_G\left(\left\{\Lambda\in M_G \mid \exists i,j\in V s.t. \exists k_0\in\mathbb{N} , a_0\in[n], \lambda_{a_0i}^{[k_0]}\lambda_{a_0j}^{[k_0]}\neq 0, \mbox{ and } h_{ij}= 0\right\}\right) \\
    &= \mu_G\left(\left\{\Lambda\in M_G \mid \exists i,j\in V s.t. \exists k_0\in\mathbb{N} , a_0\in[n], \lambda_{a_0i}^{[k_0]}\lambda_{a_0j}^{[k_0]} > 0, \mbox{ and } h_{ij}= 0\right\}\right) \\
    &= \mu_G\left(\bigcup_{i,j\in V}\bigcup_{k_0\in\mathbb{N}}\bigcup_{a_0\in[n]}\left\{\Lambda \in M_G \mid \lambda_{a_0i}^{[k_0]}\lambda_{a_0j}^{[k_0]} > 0, \mbox{ and } h_{ij}= 0\right\}\right) \\
    &\leq \sum_{i,j\in V}\sum_{k_0\in\mathbb{N}}\sum_{a_0\in[n]}\mu_G\left(\left\{\Lambda \in M_G \mid \lambda_{a_0i}^{[k_0]}\lambda_{a_0j}^{[k_0]} > 0, \mbox{ and } h_{ij}= 0\right\}\right) \\
    &= 0.
\end{align*}
The inequality is true because $\mathbb{N}$ is countable (\cite[Theorem~10.2]{billingsley2017probability}).
\end{proof}

\section{Proof of Lemma \ref{lem:rank_J_G_psi}}
\label{app:prf_lem_rank_J_G_psi}
\begin{proof}
Recall that $\overline{\mathbf{J}}_G$ is $\mathbf{J}_G$ with additional columns that coincide with existing columns of $\mathbf{J}_G$. Therefore,
\begin{equation*}
    rank\left(\mathbf{J}_G\right) = rank\left(\overline{\mathbf{J}}_G\right).
\end{equation*}
Since
\begin{equation*}
    \overline{\mathbf{J}}_G = \psi_G\left(B\right)\left(I_{n^2}-\Lambda\otimes\Lambda\right)^{-1},
\end{equation*}
$rank\left(\overline{\mathbf{J}}_G\right)=rank\left(\psi_G\left(B\right)\right)$. Therefore,
\begin{equation*}
    rank\left(\mathbf{J}_G\right) = rank\left(\psi_G\left(B\right)\right).
\end{equation*}
\end{proof}

\section{Proof of Lemma \ref{lem:rank_J_B_nr}}
\label{app:prf_lem_J_B_nr}
\begin{proof}
By Lemma \ref{lem:rank_J_G_psi}, $\mbox{rank}\left(\mathbf{J}_G\right) = \mbox{rank}\left(\psi_G(B)\right)$, and therefore it is sufficient to prove that $\mbox{rank}\left(\psi_G(B)\right) = n_r$.

Consider a $\left(E_G+1\right)\times\left(E_G+1\right)$ submatrix of $\psi_G\left(B\right)$, where the columns correspond to the set
\begin{equation*}
    \left\{\sigma_{ij} \mid i,j\in\left[n\right],(i,j) \mbox{ or } (j,i)\in\mathfrak{E}_G\right\} \cup \left\{\sigma_{kl}\right\},
\end{equation*}
where $\left\{k,l\right\}\in\mathfrak{C}_G^{mc}$. Note that $\mathfrak{C}_G^{mc}\neq \emptyset$ by Lemma \ref{lem:exist_kl}. After reordering the rows and columns, the submatrix, denoted as $\psi_G\left(B\right)^{\left[E_G+1\right]}$ is
\begin{align*}
    &\psi_G\left(B\right)^{\left[E_G +1\right]} \\ 
    &=\left[\begin{array}{ccccc}
        2\left(\Sigma\Lambda\right)_{11} & 0 & \cdots & 0 & \left(\Sigma\Lambda\right)_{1b}\delta_{a1} + \left(\Sigma\Lambda\right)_{1a}\delta_{b1} \\
        0 & 2\left(\Sigma\Lambda\right)_{22} & \cdots & 0 & \left(\Sigma\Lambda\right)_{2b}\delta_{a2} + \left(\Sigma\Lambda\right)_{2a}\delta_{b2} \\
        \vdots & \vdots & \ddots & \vdots & \vdots \\
        0 & 0 & \cdots & 2\left(\Sigma\Lambda\right)_{nn} & \left(\Sigma\Lambda\right)_{nb}\delta_{an} + \left(\Sigma\Lambda\right)_{na}\delta_{bn} \\
        &&&& \\
        2\left(\Sigma\Lambda\right)_{i1}\delta_{j1} & 2\left(\Sigma\Lambda\right)_{i2}\delta_{j2} & \cdots & 2\left(\Sigma\Lambda\right)_{in}\delta_{jn} & \left(\Sigma\Lambda\right)_{ib}\delta_{ja} + \left(\Sigma\Lambda\right)_{ia}\delta_{jb} \\
        &&&& \\
        1 & 1 & \cdots & 1 & 0
    \end{array}\right].
\end{align*}
Since $n_r\leq n_c'$, it's sufficient to prove $\psi_G\left(B\right)^{\left[E_G+1\right]}$ is full rank.

Consider $\psi_G\left(B\right)^{\left[E_G+1\right]}$ as a block matrix such that
\begin{align*}
    \psi_G\left(B\right)^{\left[E_G+1\right]} = \left[\begin{array}{cc}
        A & B \\
        C & D
    \end{array}\right],
\end{align*}
where
\begin{align*}
    A = \left[\begin{array}{cccc}
        2\left(\Sigma\Lambda\right)_{11} & 0 & \cdots & 0 \\
        0 & 2\left(\Sigma\Lambda\right)_{22} & \cdots & 0 \\
        \vdots & \vdots & \ddots & \vdots \\
        0 & 0 & \cdots & 2\left(\Sigma\Lambda\right)_{nn}
    \end{array}\right].
\end{align*}
It's clear that $A$ is invertible, since $\left(\Sigma\Lambda\right)_{ii}\neq 0$ for all $i\in\left[n\right]$. Apply the Guttman rank additivity formula (see \cite{guttman1944general}), we have
\begin{equation*}
    rank\left(\psi_G\left(B\right)^{\left[E_G+1\right]}\right) = rank\left(A\right) + rank\left(D-CA^{-1}B\right).
\end{equation*}
Therefore, $\psi_G\left(B\right)^{\left[E_G+1\right]}$ is full rank if and only if $D-CA^{-1}B$ is full rank.

By definition, for any $\left(i,j\right)\in\mathfrak{E}_G'$ and $\left(a,b\right)\in S'$,
\begin{align*}
    &\left(CA^{-1}B\right)^{\left[\lambda_{ij},\sigma_{ab}\right]} \\
    &=\left[\begin{array}{cccc}
        2\left(\Sigma\Lambda\right)_{i1}\delta_{j1} & 2\left(\Sigma\Lambda\right)_{i2}\delta_{j2} & \cdots & 2\left(\Sigma\Lambda\right)_{in}\delta_{jn}
    \end{array}\right] \\ 
    &\left[\begin{array}{cccc}
        2^{-1}\left(\Sigma\Lambda\right)_{11}^{-1} & 0 & \cdots & 0 \\
        0 & 2^{-1}\left(\Sigma\Lambda\right)_{22}^{-1} & \cdots & 0 \\
        \vdots & \vdots & \ddots & \vdots \\
        0 & 0 & \cdots & 2^{-1}\left(\Sigma\Lambda\right)_{nn}^{-1}
    \end{array}\right]\left[\begin{array}{c}
        \left(\Sigma\Lambda\right)_{1b}\delta_{a1} + \left(\Sigma\Lambda\right)_{1a}\delta_{b1} \\
        \left(\Sigma\Lambda\right)_{2b}\delta_{a2} + \left(\Sigma\Lambda\right)_{2a}\delta_{b2} \\
        \vdots \\
        \left(\Sigma\Lambda\right)_{nb}\delta_{an} + \left(\Sigma\Lambda\right)_{na}\delta_{bn}
    \end{array}\right] \\
    &=\left[\begin{array}{ccccc}
        0 & \cdots & \left(\Sigma\Lambda\right)_{ij}\left(\Sigma\Lambda\right)_{jj}^{-1} & \cdots & 0
    \end{array}\right]\left[\begin{array}{c}
        \left(\Sigma\Lambda\right)_{1b}\delta_{a1} + \left(\Sigma\Lambda\right)_{1a}\delta_{b1} \\
        \left(\Sigma\Lambda\right)_{2b}\delta_{a2} + \left(\Sigma\Lambda\right)_{2a}\delta_{b2} \\
        \vdots \\
        \left(\Sigma\Lambda\right)_{nb}\delta_{an} + \left(\Sigma\Lambda\right)_{na}\delta_{bn}
    \end{array}\right] \\
    &= \delta_{ja}\left[\left(\Sigma\Lambda\right)_{ij}\left(\Sigma\Lambda\right)_{jj}^{-1}\left(\Sigma\Lambda\right)_{jb}\right] + \delta_{jb}\left[\left(\Sigma\Lambda\right)_{ij}\left(\Sigma\Lambda\right)_{jj}^{-1}\left(\Sigma\Lambda\right)_{ja}\right],
\end{align*}
and
\begin{align*}
    &\left(CA^{-1}B\right)^{\left[\omega,\sigma_{ab}\right]} \\
    &= \left[\begin{array}{cccc}
        1 & 1 & \cdots & 1
    \end{array}\right]\left[\begin{array}{cccc}
        2^{-1}\left(\Sigma\Lambda\right)_{11}^{-1} & 0 & \cdots & 0 \\
        0 & 2^{-1}\left(\Sigma\Lambda\right)_{22}^{-1} & \cdots & 0 \\
        \vdots & \vdots & \ddots & \vdots \\
        0 & 0 & \cdots & 2^{-1}\left(\Sigma\Lambda\right)_{nn}^{-1}
    \end{array}\right] \\ 
    &\left[\begin{array}{c}
        \left(\Sigma\Lambda\right)_{1b}\delta_{a1} + \left(\Sigma\Lambda\right)_{1a}\delta_{b1} \\
        \left(\Sigma\Lambda\right)_{2b}\delta_{a2} + \left(\Sigma\Lambda\right)_{2a}\delta_{b2} \\
        \vdots \\
        \left(\Sigma\Lambda\right)_{nb}\delta_{an} + \left(\Sigma\Lambda\right)_{na}\delta_{bn}
    \end{array}\right] \\
    &= \left[\begin{array}{cccc}
        2^{-1}\left(\Sigma\Lambda\right)_{11}^{-1} & 2^{-1}\left(\Sigma\Lambda\right)_{22}^{-1} & \cdots & 2^{-1}\left(\Sigma\Lambda\right)_{nn}^{-1}
    \end{array}\right]\left[\begin{array}{c}
        \left(\Sigma\Lambda\right)_{1b}\delta_{a1} + \left(\Sigma\Lambda\right)_{1a}\delta_{b1} \\
        \left(\Sigma\Lambda\right)_{2b}\delta_{a2} + \left(\Sigma\Lambda\right)_{2a}\delta_{b2} \\
        \vdots \\
        \left(\Sigma\Lambda\right)_{nb}\delta_{an} + \left(\Sigma\Lambda\right)_{na}\delta_{bn}
    \end{array}\right] \\
    &= 2^{-1}\left(\Sigma\Lambda\right)_{ab}\left(\Sigma\Lambda\right)_{aa}^{-1} + 2^{-1}\left(\Sigma\Lambda\right)_{ba}\left(\Sigma\Lambda\right)_{bb}^{-1}.
\end{align*}
Thus,
\begin{align*}
    \left(D-CA^{-1}B\right)^{\left[\lambda_{ij},\sigma_{ab}\right]} &= \delta_{ja}\left[\left(\Sigma\Lambda\right)_{ib} - \left(\Sigma\Lambda\right)_{ij}\left(\Sigma\Lambda\right)_{jj}^{-1}\left(\Sigma\Lambda\right)_{jb}\right] \\ 
    &+ \delta_{jb}\left[\left(\Sigma\Lambda\right)_{ia} - \left(\Sigma\Lambda\right)_{ij}\left(\Sigma\Lambda\right)_{jj}^{-1}\left(\Sigma\Lambda\right)_{ja}\right], \\
    \left(D-CA^{-1}B\right)^{\left[\omega,\sigma_{ab}\right]} &= -2^{-1}\left(\Sigma\Lambda\right)_{ab}\left(\Sigma\Lambda\right)_{aa}^{-1} - 2^{-1}\left(\Sigma\Lambda\right)_{ba}\left(\Sigma\Lambda\right)_{bb}^{-1}.
\end{align*}
Therefore, for $B_G$ of Definition \ref{def:B_G_1},
\begin{align*}
    B_G^{\left[\lambda_{ij},\sigma_{ab}\right]} &= \left(\Sigma\Lambda\right)_{jj}\left(D-CA^{-1}B\right)^{\left[\lambda_{ij},\sigma_{ab}\right]}, \\
    B_G^{\left[\omega,\sigma_{ab}\right]} &= \left(-2\right)\left(D-CA^{-1}B\right)^{\left[\omega,\sigma_{ab}\right]}.
\end{align*}
Note that
\begin{equation*}
    B_G = M_G\left(D-CA^{-1}B\right),
\end{equation*}
where $M_G$ is a diagonal matrix whose elements on the diagonal are $\left(\Sigma\Lambda\right)_{jj}$, s.t. \\ $\exists i\in\left[n\right], \left(i,j\right)\in \mathfrak{E}_G'$ and $-2$. Hence $M_G$ is invertible, and
\begin{equation*}
    rank\left(D-CA^{-1}B\right)=rank\left(B_G\right).
\end{equation*}
\end{proof}

\section{Proof of Case 2 in Theorem \ref{thm:rank_no_multi_edges}}
\label{app:prf_2_thm_rank_no_multi_edges}
By Lemma \ref{lem:exist_kl}, we know that under Assumption \ref{asmp:no_multi_edges}, the condition $n_r > n_c'$ is satisfied if and only if the set $\mathfrak{C}_G^{mc}$ is empty, i.e., $C_G^{mc}=0$, meaning that for all $i,j\in[n]$, either $(i,j)\mbox{ or }(j,i)\in \mathfrak{E}_G$ or $i,j$ do not belong to the same maximal class. In addition, $n_c'=E_G$.

Using similar arguments as before, we define another matrix, $B_G'$, which is derived from the Jacobian matrix, and prove that the Jacobian matrix $\mathbf{J}_G$ is full rank if and only if $B_G'$ is full rank.

\begin{definition}
Let $\mathbf{M}_G$ be a stationary VAR(1) model that satisfies $n_r > n_c'$. Define a square matrix $B_G'$ of size $E_G'$, where the rows correspond to $\lambda_{ij}$, s.t. $i,j\in[n]$ and $(i,j)\in\mathfrak{E}_G'$, and the columns correspond to $\sigma_{ab}$ s.t. $a,b\in[n]$ and $(a,b)$ or $(b,a)\in\mathfrak{E}_G'$ as follows
\begin{equation*}
    B_G'^{\left[\lambda_{ij},\sigma_{ab}\right]} = \delta_{ja}\left[\left(\Sigma\Lambda\right)_{ib}\left(\Sigma\Lambda\right)_{jj} - \left(\Sigma\Lambda\right)_{ij}\left(\Sigma\Lambda\right)_{jb}\right] + \delta_{jb}\left[\left(\Sigma\Lambda\right)_{ia}\left(\Sigma\Lambda\right)_{jj} - \left(\Sigma\Lambda\right)_{ij}\left(\Sigma\Lambda\right)_{ja}\right].
\end{equation*}
\end{definition}

Note that $B_G'$ is $B_G$ excluding the last row that corresponds to $\omega$.

\begin{lemma}
Let $\mathbf{M}_G$ be a stationary VAR(1) model that satisfies the condition $n_r > n_c'$. Then $\mbox{rank}\left(\psi_G\left(B\right)\right) = n_c'$ (i.e., $\mbox{rank}\left(\mathbf{J}_G\right)=n_c'$) if $B_G'$ is full rank, i.e., $\mbox{rank}(B_G')=E_G'$.
\end{lemma}
\begin{proof}
Consider a $E_G\times E_G$ submatrix of $\psi_G(B)$, where the rows correspond to $\lambda_{ij}$ s.t. $i,j\in[n]$ and $(i,j)\in\mathfrak{E}_G$, and the columns correspond to $\sigma_{ab}$ s.t. $a,b\in[n]$, $a\leq b$, and $(a,b)$ or $(b,a)\in\mathfrak{E}_G$. In fact, this submatrix contains all distinct non-zero columns of $\psi_G(B)$. Therefore, if this submatrix is full rank, then $\mbox{rank}\left(\psi_G\left(B\right)\right) = n_c' = E_G$.

Use the same technique as in Lemma \ref{lem:rank_J_B_nr}, where we conder $\psi_G\left(B\right)$ as a block matrix and apply the Guttman rank additivity formula, we know that
\begin{equation*}
    \mbox{rank}(\psi_G(B)) = E_G\ \Leftrightarrow\ \mbox{rank}(B_G')=E_G'.
\end{equation*}
\end{proof}

Denote a subset $\overline{M_G^B}$ of $M_G$:
\begin{equation*}
    \overline{M_G^B} := \left\{\Lambda\in M_G \mid \det(B_G') = 0\right\}.
\end{equation*}

\begin{lemma}\label{lem:B_G_rank_nc'}
Let $\mathbf{M}_G$ be a stationary VAR(1) model that satisfies Assumption \ref{asmp:no_multi_edges} and $n_r > n_c'$, then
\begin{equation*}
    \mu_G\left(\overline{M_G^B}\right) = 0.
\end{equation*}
\end{lemma}
\begin{proof}
Use similar arguments as in Lemma \ref{lem:B_G_rank_nr}, $\det(B_G')$ is a rational function of the entries of $\Lambda$ and $\omega$. Therefore, it is sufficient to find a $\Lambda_0\in M_G$ s.t. $\det(\overline{B_{G_0}})\neq 0$, where $\overline{B_{G_0}}$ is $B_G'$ with $\Lambda=\Lambda_0$. Let
\begin{equation*}
    \Lambda_0 = 2I_n,
\end{equation*}
then $\overline{B_{G_0}}$ is diagonal with non-zero diagonal entries. Therefore $\det(B_G')\neq 0$.
\end{proof}

\section{Algorithms Related to Maximal Classes}
\label{app:alg_mc}

We design Algorithm \ref{alg:maxc} to derive the set of maximal classes for any directed graph. This Algorithm first applies Kosaraju's algorithm (\cite{sharir1981strong}) to find all SCCs of the graph, and then construct the condensed graph, which is a directed acyclic graph where all SCCs are represented as a single node. Finally, we perform a Depth First Search (DFS) algorithm (\cite{thomas2009introduction}) on each of the nodes with in-degree zero in the condensed graph.

The computational complexity of Algorithm \ref{alg:maxc} is dictated by the Kosaraju's algorithm, the construction of the condensed graph, and the subsequent DFS searches. Since these components are all linear operations relative to the size of the graph, the complexity of the algorithm is upper bounded by $O(k(n + E))$, where $n$ is the number of nodes, $E$ is the number of edges in the original graph, and $k$ is the number of SCCs in the condensed graph. Consequently, the best-case scenario achieves a complexity of $O(n + E)$, while the worst-case upper bound is $O(n(n + E))$.

\begin{algorithm}[htbp]
\caption{Algorithm to identify maximal classes of a directed graph}\label{alg:maxc}
\hspace*{\algorithmicindent} \textit{Input:} a directed graph $G=\left(V,\mathfrak{E}_G\right)$ \\
\hspace*{\algorithmicindent} \textit{Output:} the number and the list of maximal classes of $G$
\begin{algorithmic}[1]
\Procedure{DFS}{$G,i,j$} \Comment{Find the set of all reachable nodes from $i$, and store it into the $j^{th}$ list of the list maxc}
\State add $i$ to maxc[$j$] \Comment{Add $i$ itself to the list}
\ForAll {nodes $w\in V$ s.t. $\left(i,w\right)\in \mathfrak{E}_G$} \Comment{Do the same thing to all children of $i$}
    \State DFS$(G,w,j)$
\EndFor
\EndProcedure

\Procedure{MAXC}{G} \Comment{Search for the maximal classes of graph $G$}
\State Perform Kosaraju's algorithm, return $L$, a list of SCCs
\State Construct the condensed graph $G' = \left(L,\mathfrak{E}_{G'}\right)$
\State $j\gets 1$
\State maxc$\gets$ an empty list of lists
\ForAll{$i\in |L|$ s.t. $\deg^-\left(L[i]\right)=0$ in $G'$}
\State DFS$(G',L[i],j)$ \Comment{Put the set of all reachable nodes from $i$ into the $j^{th}$ element of the list}
\State $j\gets j+1$
\EndFor

\Return $j-1$ (number of maximal classes) and maxc (list of maximal classes)
\EndProcedure
\end{algorithmic}
\end{algorithm}

To reconstruct the list of graphs associated with a given set of maximal classes, that is, the inverse problem of the precedent algorithm, we propose Algorithm \ref{alg:maxc_graphs}. This algorithm starts with identifying candidate sources for each maximal class. It then iterates through the complete set of possible edges $\left(u,v\right)\in V\times V$ to isolate 'forbidden' edges -- those whose inclusion would violate the definitions of the maximal classes or their sources. Finally, the algorithm evaluates subsets of the remaining 'allowable' edges, applying Algorithm \ref{alg:maxc} to verify which resulting graphs yield the target set of maximal classes.

The complexity of Algorithm \ref{alg:maxc_graphs} is significantly higher due to its combinatorial nature. While the initial source allocation and edge pruning is efficient at $O(m\cdot n^2)$ with $m$ being the number of maximal classes, the subsequent exploration of edge subsets results in a worst-case complexity of $O(2^p \cdot n(n+E))$, where $p$ is the number of 'allowable' edges. The resulting complexity is upper bounded by $O(m\cdot n^2 + 2^p \cdot n(n+E))$.

An R implementation of both the algorithms is available at \url{https://github.com/Bi-xuan/maximal_class}. While the procedures presented in this section provide a systematic way to reconstruct maximal classes from their corresponding graph structures as well as the inverse, it is important to note that these implementations are not computationally optimal and leave significant room for algorithmic refinement. However, such optimizations and a systematic empirical validation across a broad selection of synthetic data and large-scale simulated ecological networks remain outside the primary theoretical focus of this work. These enhancements and performance testing are dedicated to future research.

\begin{algorithm}[htbp]
\caption{Algorithm for the reconstruction of graphs via maximal classes}\label{alg:maxc_graphs}
\hspace*{\algorithmicindent} \textit{Input:} a list of maximal classes $\mathfrak{MC} = \left\{\mathcal{MC}_k,k=1,\cdots,m\right\}$ \\
\hspace*{\algorithmicindent} \textit{Output:} the list of graphs associated with $\mathfrak{MC}$
\begin{algorithmic}[1]
\Procedure{RECONSTRUCT}{$\mathfrak{MC}$}
\State $V \gets \bigcup_{k=1}^m\mathcal{MC}_k$ \Comment{Define the set of nodes}
\State SourceCandidates $\gets$ an empty list of lists
\ForAll{$\mathcal{MC}_k\in\mathfrak{MC}$}
\State $S_k \gets \left\{v\in\mathcal{MC}_k \mid v\notin \mathcal{MC}_j,\forall j \neq k\right\}$ \Comment{Identify nodes that belong to only one maximal class}
\State Add $S_k$ to SourceCandidates
\EndFor
\State ForbiddenEdges $\gets \emptyset$
\ForAll{$(u,v)\in V\times V$}
\ForAll{$\mathcal{MC}_k\in\mathfrak{MC}$}
\If{$u\in\mathcal{MC}_k$ and $v\notin \mathcal{MC}_k$}
\State Add $(u,v)$ to ForbiddenEdges \Comment{By definition of maximal classes}
\EndIf
\If{$v\in S_k$ and $u\notin S_k$}
\State Add $(u,v)$ to ForbiddenEdges \Comment{Source nodes must have in-degree zero from outside of their SCC}
\EndIf
\EndFor
\EndFor
\State PossibleEdges $\gets V\times V \backslash$ForbiddenEdges
\State PossibleGraphs $\gets \emptyset$
\ForAll{selection of edges $E\subseteq$ PossibleEdges} \Comment{Permute over subsets of possible edges}
\State $G\gets (V,E)$
\If{MAXC($G$) $== \mathfrak{MC}$} \Comment{Verify using Algorithm \ref{alg:maxc}}
\State Add $G$ to PossibleGraphs
\EndIf
\EndFor
\Return PossibleGraphs
\EndProcedure
\end{algorithmic}
\end{algorithm}
\end{appendix}

\end{document}